\documentclass[11pt]{article}

\usepackage[margin=1in]{geometry}
\usepackage{amsmath,amssymb,amsthm,mathtools,mathrsfs}
\usepackage{enumitem,microtype,booktabs,longtable,array,aliascnt}
\usepackage[ruled,section]{algorithm}
\usepackage{algpseudocode}
\usepackage{graphicx}
\usepackage[colorlinks=true,linkcolor=blue,citecolor=blue,urlcolor=blue]{hyperref}
\usepackage[nameinlink,capitalise]{cleveref}

\usepackage{xargs}
\let\standarddagger\dagger

\def\ddefloop#1{\ifx\ddefloop#1\else\ddef{#1}\expandafter\ddefloop\fi}
\def\ddef#1{\expandafter\def\csname bb#1\endcsname{\ensuremath{\mathbb{#1}}}}
\ddefloop ABCDEFGHIJKLMNOPQRSTUVWXYZ\ddefloop
\def\ddefloop#1{\ifx\ddefloop#1\else\ddef{#1}\expandafter\ddefloop\fi}
\def\ddef#1{\expandafter\def\csname b#1\endcsname{\ensuremath{\mathbf{#1}}}}
\ddefloop ABCDEFGHIJKLMNOPQRSTUVWXYZ\ddefloop
\def\ddef#1{\expandafter\def\csname sf#1\endcsname{\ensuremath{\mathsf{#1}}}}
\ddefloop ABCDEFGHIJKLMNOPQRSTUVWXYZ\ddefloop
\def\ddef#1{\expandafter\def\csname c#1\endcsname{\ensuremath{\mathcal{#1}}}}
\ddefloop ABCDEFGHIJKLMNOPQRSTUVWXYZ\ddefloop
\def\ddef#1{\expandafter\def\csname h#1\endcsname{\ensuremath{\widehat{#1}}}}
\ddefloop ABCDEFGHIJKLMNOPQRSTUVWXYZ\ddefloop
\def\ddef#1{\expandafter\def\csname hc#1\endcsname{\ensuremath{\widehat{\mathcal{#1}}}}}
\ddefloop ABCDEFGHIJKLMNOPQRSTUVWXYZ\ddefloop
\def\ddef#1{\expandafter\def\csname t#1\endcsname{\ensuremath{\widetilde{#1}}}}
\ddefloop ABCDEFGHIJKLMNOPQRSTUVWXYZ\ddefloop
\def\ddef#1{\expandafter\def\csname tc#1\endcsname{\ensuremath{\widetilde{\mathcal{#1}}}}}
\ddefloop ABCDEFGHIJKLMNOPQRSTUVWXYZ\ddefloop
\def\ddef#1{\expandafter\def\csname #1#1\endcsname{\ensuremath{\mathbb{#1}}}}
\ddefloop ABCDEFGHIJKLMNOPQRSTUVWXYZ\ddefloop
\def\ddef#1{\expandafter\def\csname #1\endcsname{\ensuremath{\mathbb{#1}}}}
\ddefloop ABCDEFGHIJKLMNOPQRSTUVWXYZ\ddefloop
\def\ddef#1{\expandafter\def\csname D#1\endcsname{\ensuremath{\Delta(\mathcal{#1})}}}
\ddefloop ABCDEFGHIJKLMNOPQRSTUVWXYZ\ddefloop
\def\ddefloop#1{\ifx\ddefloop#1\else\ddef{#1}\expandafter\ddefloop\fi}
\def\ddef#1{\expandafter\def\csname scr#1\endcsname{\ensuremath{\mathscr{#1}}}}
\ddefloop ABCDEFGHIJKLMNOPQRSTUVWXYZ\ddefloop

\newcommandx{\gam}[3][1=x,2=z]{\gamma_{#2,#3}(#1)}
\newcommandx{\gamp}[3][1=x,2=z]{\dot\gamma_{#2,#3}(#1)}

\newcommandx{\gamz}[4][1=x,2=z,3=\lr,4=\xz]{\gamma_{#2,#3,#4}(#1)}
\newcommandx{\gamzp}[4][1=x,2=z,3=\lr,4=\xz]{\dot\gamma_{#2,#3,#4}(#1)}

\newcommand{\xz}{x_0}

\newcommand{\lr}{r}

\newcommand{\Cov}{\mathrm{Cov}}

\renewcommand{\dagger}{\texttt{Dagger}\xspace}

\DeclareMathOperator{\Var}{Var}

\newcommand{\polylog}{\mathrm{polylog}}

\newcommand{\deq}{\coloneqq}

\def\multiset#1#2{\ensuremath{\Bigl(\kern-.3em\Bigl(\genfrac{}{}{0pt}{}{#1}{#2}\Bigr)\kern-.3em\Bigr)}}

\DeclareMathOperator{\Law}{Law}

\usepackage{todonotes}

\newtheorem{theorem}{Theorem}[section]
\newtheorem*{unnumberedtheorem}{Theorem}
\newaliascnt{proposition}{theorem}
\newtheorem{proposition}[proposition]{Proposition}
\aliascntresetthe{proposition}
\newaliascnt{lemma}{theorem}
\newtheorem{lemma}[lemma]{Lemma}
\aliascntresetthe{lemma}
\newaliascnt{corollary}{theorem}
\newtheorem{corollary}[corollary]{Corollary}
\aliascntresetthe{corollary}
\newaliascnt{remark}{theorem}
\newtheorem{remark}[remark]{Remark}
\aliascntresetthe{remark}
\theoremstyle{definition}
\newaliascnt{definition}{theorem}
\newtheorem{definition}[definition]{Definition}
\aliascntresetthe{definition}
\numberwithin{equation}{section}
\makeatletter
\def\theHALG@line{\thealgorithm.\arabic{ALG@line}}
\makeatother
\crefname{algorithm}{Algorithm}{Algorithms}
\Crefname{algorithm}{Algorithm}{Algorithms}
\crefname{theorem}{Theorem}{Theorems}
\Crefname{theorem}{Theorem}{Theorems}
\crefname{proposition}{Proposition}{Propositions}
\Crefname{proposition}{Proposition}{Propositions}
\crefname{lemma}{Lemma}{Lemmas}
\Crefname{lemma}{Lemma}{Lemmas}
\crefname{corollary}{Corollary}{Corollaries}
\Crefname{corollary}{Corollary}{Corollaries}
\crefname{remark}{Remark}{Remarks}
\Crefname{remark}{Remark}{Remarks}
\crefname{definition}{Definition}{Definitions}
\Crefname{definition}{Definition}{Definitions}
\newcommand{\Prob}{\mathbb P}
\newcommand{\dd}{\mathrm d}

\newcommand{\TV}{\operatorname{TV}}

\newcommand{\prox}{\operatorname{prox}}

\newcommand{\1}{\mathbf 1}
\renewcommand{\cH}{\mathfrak H}
\newcommand{\opP}{\mathsf P}
\newcommand{\opU}{\mathsf U}
\newcommand{\opR}{\mathsf R}
\newcommand{\opS}{\mathsf S}
\newcommand{\opH}{\mathsf H}
\newcommand{\opHperpperp}{\mathsf H_{\perp\perp}}
\newcommand{\opK}{\mathsf K}
\newcommand{\opUPP}{\mathsf U_{\mathsf P\mathsf P}}
\newcommand{\opUperpP}{\mathsf U_{\perp\mathsf P}}
\newcommand{\opUperpperp}{\mathsf U_{\perp\perp}}
\newcommand{\opVperpP}{\mathsf V_{\perp\mathsf P}}
\newcommand{\opPi}{\mathsf{\Pi}}
\newcommand{\opGammaP}{\mathsf{\Gamma}_{\mathsf P}}
\newcommand{\opL}{\mathsf L}
\newcommand{\ran}{\operatorname{ran}}

\title{Accelerated High-Accuracy Sampling from a Warm Start\\
via the Proximal Bouncy Particle Sampler}
\author{
 Fan Chen\thanks{Department of Electrical Engineering and Computer Science,
 Massachusetts Institute of Technology.
 Email: \href{mailto:fanchen@mit.edu}{\texttt{fanchen@mit.edu}}.}
 \and
 Sinho Chewi\thanks{Department of Statistics and Data Science, Yale University.
 Email: \href{mailto:sinho.chewi@yale.edu}{\texttt{sinho.chewi@yale.edu}}.}
 \and
 Jianfeng Lu\thanks{Department of Mathematics, Duke University.
 Email: \href{mailto:jianfeng@math.duke.edu}{\texttt{jianfeng@math.duke.edu}}.}
 \and
 Matthew S. Zhang\thanks{Department of Mathematics, Massachusetts Institute of Technology.
 Email: \href{mailto:shuns436@mit.edu}{\texttt{shuns436@mit.edu}}.}
}
\date{}

\begin{document}
\begingroup
\let\dagger\standarddagger
\maketitle
\endgroup

\begin{abstract}
We study the problem of sampling from
\(\mu(\dd x)\propto e^{-V(x)}\,\dd x\) on \(\R^d\), where \(V\) is
\(\alpha\)-strongly convex and \(\beta\)-smooth, and write
\(\kappa\deq\beta/\alpha\).  We design and analyze the Proximal Bouncy
Particle Sampler (Proximal BPS), a new sampler that combines ideas from the proximal sampler and the bouncy particle sampler.
From a warm start initialization with \( O(1) \) R\'enyi divergence
w.r.t.\ \(\mu\), Proximal BPS returns a sample whose law is $\varepsilon$-close to $\mu$ in total variation distance using
\(\widetilde O(\sqrt\kappa\,d^{1/4} \,\polylog(1/\varepsilon))\) gradient queries in expectation.
\end{abstract}

\tableofcontents

\section{Introduction}

We study the problem of sampling from the probability measure
\(\mu(\dd x)=Z^{-1}e^{-V(x)}\,\dd x\) on \(\R^d\), where
\(V\in C^2(\R^d)\) satisfies, for \(0<\alpha\le\beta\),
\begin{equation}
 \alpha I\preceq\nabla^2V(x)\preceq\beta I\,,
 \qquad x\in\R^d\,,
 \qquad \kappa\deq\tfrac\beta\alpha\,.
 \label{eq:SC-smooth}
\end{equation}
The main result of this paper is the following.

\begin{unnumberedtheorem}[Main theorem; informal]
Assume \eqref{eq:SC-smooth}, and suppose that the initial law \(\mu_0\)
satisfies the warm start condition in order-$2$ R\'enyi divergence: \(D_2(\mu_0\|\mu)=O(1)\).  For every
\(0<\varepsilon<1/4\), there is a randomized algorithm that, given
\(X_0\sim\mu_0\) and query access to \(\nabla V\), returns
\(\widehat X\sim\widehat\mu\) satisfying
\[
 \TV(\widehat\mu,\mu)\le\varepsilon\,,
\]
and whose expected number of queries to \(\nabla V\) is at most
\[
 \widetilde O \bigl(\sqrt\kappa\,d^{1/4}\bigr)
\,.
\]
Here \(\widetilde O\) hides universal powers of logarithms in
\(d\), \(\kappa\), and \(1/\varepsilon\).
\end{unnumberedtheorem}

The precise statement and parameter choices appear in
\cref{thm:main-general,cor:tuned}.

We compare our result with \emph{high-accuracy samplers}, that is,
samplers whose complexity depends polylogarithmically on
\(1/\varepsilon\).  Under \eqref{eq:SC-smooth}, the state of the art
until recently was the Metropolis-adjusted Langevin algorithm
(MALA)~\cite{altschuler-chewi-warm-starts,chewi-et-al-mala,wu-schmidler-chen}
and the proximal sampler~\cite{fors,fan-yuan-chen}, which achieved
\(\widetilde O(\kappa d^{1/2})\) complexity.
Very recently, Chen et al.~\cite{exact-diffusions} used the first-order
rejection sampling (FORS) mechanism to develop a method for exact
simulation of diffusions.  Together with accelerated entropic
hypocoercivity~\cite{li-lu-hypercontractivity,lu-entropy}, this yielded
a high-accuracy sampler with complexity
\(\widetilde O(\kappa^{2/3}d^{1/3})\).
Also, Lu and Luo~\cite{LuLuo26Zigzag} introduced windowed thinning,
an exact simulation method for the bouncy particle sampler (BPS) and the zigzag sampler.
Their estimates give
expected complexities of \(\widetilde O(\kappa^{1/2}d)\) gradient queries
for BPS and \(\widetilde O(\kappa d^{5/4})\) partial derivative queries
for zigzag, from a warm start.
These results were significant in part because they constituted progress
toward acceleration for log-concave sampling.

From a warm start, our
main theorem achieves complexity
\(\widetilde O(\kappa^{1/2}d^{1/4})\).  It matches the condition number
dependence of BPS while improving its dimension
dependence, and always improves upon the computational cost of the zigzag sampler (even if we consider one gradient query to be $d$  times as expensive as a partial derivative query).
It also improves over all other prior results in both parameters. In particular, it achieves full
acceleration together with $d^{1/4}$ dependence on dimension.

In a companion work~\cite{picard-hmc}, we will show how to obtain a warm start using
\(\widetilde O(\kappa^{7/6}d^{1/6})\) gradient queries.  We do not discuss the warm start
question further in this paper.

\paragraph{Proximal Bouncy Particle Sampler.}
Our algorithm, the \emph{Proximal Bouncy Particle Sampler} (Proximal
BPS), combines two ideas.  The proximal sampler~\cite{lee-shen-tian}
breaks the sampling problem into a sequence of easier local problems
by introducing a Gaussian perturbation of the current position.
The bouncy particle sampler~\cite{bps-original} explores a distribution
by moving a particle along straight lines, with random changes of direction
chosen to preserve the desired distribution.  We bring these ideas
together by designing local particle motions that retain a memory of
previous steps.  This memory acts as momentum, allowing the algorithm
to make sustained progress across successive local problems and giving
the accelerated convergence rate in our main theorem.  The local motions
follow explicitly computable curved paths, with occasional ``bounces''
that correct for the shape of the target distribution.  A gradient
evaluation at a nearby reference point makes these corrections inexpensive:
since the gradient varies little over a small region, only a small
number of bounces is needed.  Balancing the size of the local problems
against the cost of these corrections yields the
\(\widetilde O(\kappa^{1/2}d^{1/4})\) query bound. We feel that the construction of the sampler (with the help of GPT-5.6 Sol) contains some surprising ingredients.
We describe the construction and its geometry in
\cref{sec:algorithm-overview}, and give the implementation and complexity
analysis in \cref{sec:implementation}.

\paragraph{Related work.}
This work makes progress on the complexity of log-concave sampling, for which the textbook~\cite{Chewi26Book} provides an introduction and overview.

The proximal sampling framework is from Lee, Shen, and
Tian~\cite{lee-shen-tian}, with subsequent analyses and implementations
in \cite{fors,chen-chewi-salim-wibisono,fan-yuan-chen}.  These methods
alternate independent draws of the position and the auxiliary variable
from their conditional distributions.  Our update reflects the auxiliary
point through the current position, a step related to the Gaussian
overrelaxation method of Neal~\cite{neal-overrelaxation}.  We combine
this reflection with a local piecewise deterministic particle motion run for half of its
oscillation period. As we shall explain, this helps useful information to persist across iterations.

The rule for the bounces builds on the BPS of Bouchard-C\^ot\'e,
Vollmer, and Doucet~\cite{bps-original}.  The use of a quadratic
reference potential to generate curved paths, with bounces correcting
for the remaining force, is adapted from the boomerang sampler of
Bierkens, Grazzi, Kamatani, and Roberts~\cite{boomerang}.  This separation
of the force also fits the general framework for Hamiltonian piecewise
deterministic Markov processes of Andrieu, Durmus, N\"usken, and
Roussel~\cite{adnr}.

Existing convergence analyses of these processes
\cite{adnr,dobson-bierkens,lu-wang} concern continuous-time dynamics
with repeated random refreshment of the velocity.  Our local motion has
no such refreshment: it stops after a half-period, and convergence
comes from combining it with the auxiliary reflection step with occasional
resampling.  In particular, the local motion alone need not converge to
its conditional equilibrium.  Thus, the continuous-time results do not
directly give a convergence guarantee for our discrete chain.
Our analysis instead adapts the modified \(L^2\) hypocoercivity method of
Dolbeault, Mouhot, and Schmeiser~\cite{dms}, in the sharpened formulation
of Fan, Li, and Lu~\cite{fan-li-lu}, to this discrete setting. The technical details of the hypocoercivity mechanism in Appendices~\ref{sec:micro-macro} and~\ref{app:key-estimates} should be of interest to specialists.

\paragraph{AI usage.}
The algorithm and its original analysis were found by interactions with GPT-5.6 Sol. The authors recast the arguments into a more standard hypocoercive framework, prepared the manuscript, and take full responsibility for its contents.

\section{Preliminaries}
\label{sec:preliminaries}

\subsection{The standard Bouncy Particle Sampler}
\label{sec:standard-bps-prelim}

We first recall the usual Bouncy Particle Sampler (BPS) for readers
unfamiliar with the construction~\cite{bps-original}.  It is a
piecewise-deterministic Markov process on phase space
\((x,p)\in\R^d\times\R^d\), where \(p\) denotes velocity (equivalently,
momentum for unit mass).  Its invariant distribution is
\(\mu\otimes\varphi\), where \(\varphi=\cN(0,I)\).  Between events, the
particle moves along a straight line with constant \(p\):
\begin{equation}
 \dot x_t=p_t\,,\qquad \dot p_t=0\,.
 \label{eq:standard-bps-flow-prelim}
\end{equation}
At state \((x,p)\), a bounce occurs at rate
\begin{equation}
 \lambda_{\rm BPS}(x,p)
 \deq[p^\top\nabla V(x)]_+\,.
 \label{eq:standard-bps-rate-prelim}
\end{equation}
Equivalently, conditional on the current state and in the absence of a
velocity refresh, the next bounce time \(\tau_{\rm b}\) satisfies
\begin{equation}
 \Prob(\tau_{\rm b}>t\mid x,p)
 =\exp \Bigl\{-\int_0^t
 [p^\top\nabla V(x+sp)]_+\,\dd s\Bigr\}\,.
 \label{eq:standard-bps-waiting-time-prelim}
\end{equation}
At a bounce, the position is unchanged and the velocity is specularly
reflected against the gradient.  For non-zero \(h\in \R^d\), set
\begin{equation}
 R_h\deq I-2\,\frac{hh^\top}{\|h\|^2}\,,
 \qquad R_0\deq I\,.
 \label{eq:specular}
\end{equation}
The bounce update is \(p\leftarrow R_{\nabla V(x)}p\).  It preserves
\(\|p\|\), leaves the component tangent to a level set unchanged, and
reverses the normal component, since
\((R_h p)^\top h=-p^\top h\).  Thus, the particle travels freely
while moving down the potential and bounces, at a rate equal to its
instantaneous uphill directional derivative, when it moves up the
potential. Proximal BPS reuses this bounce mechanism, applied to a localized
conditional target; the construction is given in \cref{sec:algorithm}.

The standard bouncy particle sampler also uses independent velocity refreshes to ensure ergodicity, at which \(p\) is redrawn from \(\varphi\).  Writing
\[
 (\opPi_p f)(x,p)\deq\int_{\R^d}f(x,p')\,\varphi(\dd p')\,,
\]
the full generator of BPS is
\begin{equation}
 \begin{split}
 (\opL_{\rm BPS}f)(x,p)
 &=p^\top\nabla_xf(x,p)
 +[p^\top\nabla V(x)]_+
 \{f(x,R_{\nabla V(x)}p)-f(x,p)\}
+\gamma(\opPi_p f-f)(x,p)\,.
 \end{split}
 \label{eq:standard-bps-generator-prelim}
\end{equation}
For every refresh rate \(\gamma\ge0\), the process with
generator \eqref{eq:standard-bps-generator-prelim} leaves
\(\mu\otimes\varphi\) invariant~\cite{bps-original}, and its position
marginal is the desired target \(\mu\).

A standard way to simulate the inhomogeneous event clock
\eqref{eq:standard-bps-waiting-time-prelim} is Poisson
thinning~\cite{lewis-shedler}.  Suppose that, along the current deterministic segment, the
rate \(\lambda(\cdot)\) is bounded above by a constant \(\bar\lambda\).
Generate candidate times from a homogeneous Poisson process of rate
\(\bar\lambda\) and accept a candidate at time \(t\) with probability
\(\lambda(t)/\bar\lambda\); a rejected candidate causes no jump, and
the deterministic trajectory continues.  The accepted times then have rate
\(\lambda(\cdot)\), so this
procedure simulates the event clock exactly. Note that only at each candidate time, one needs to query the potential gradients. If no
convenient dominating bound is available, accepting instead with
probability \(\min\{1,\lambda(t)/\bar\lambda\}\) exactly simulates the
truncated rate \(\min\{\lambda(t),\bar\lambda\}\); the only
approximation is then the replacement of the original rate by its
truncation.  This is used in
\cref{sec:simulating-bounces}.

\subsection{Proximal augmentation and the restricted Gaussian oracle}
\label{sec:proximal-augmentation-prelim}
Our algorithm operates on the augmented distribution underlying the
proximal sampler of Lee, Shen, and Tian~\cite{lee-shen-tian}.  For a
scale \(\eta\in(0,1/\beta]\), consider the joint law
\begin{equation}
 \pi_\eta(\dd x\,\dd y)
 \propto
 \exp \Bigl\{-V(x)-\frac{\|x-y\|^2}{2\eta}\Bigr\}
 \,\dd x\,\dd y\,.
 \label{eq:augmentation}
\end{equation}
The original proximal sampler algorithm performs Gibbs sampling on $\pi_\eta$.
Namely, under \(\pi_\eta\) one may generate the joint state as
\begin{equation}
 X\sim\mu\,,\qquad Z\sim\cN(0,I)\ \text{independently}\,,
 \qquad Y=X+\sqrt\eta\,Z\,.
 \label{eq:augmentation-generative}
\end{equation}
Its \(X\)-marginal is \(\pi_\eta^X=\mu\).  Write \(\pi_\eta^Y\) for
the \(Y\)-marginal, so that
\(\pi_\eta^Y(\dd y)\propto e^{-U_\eta(y)}\,\dd y\); by
\eqref{eq:augmentation-generative}, \(\pi_\eta^Y=\mu*\cN(0,\eta I)\)
is a Gaussian smoothing of the target.
We use superscripts for marginal and conditional laws; in particular,
\begin{equation}
 \pi_\eta^{Y\mid X=x}=\cN(x,\eta I)\,,
 \qquad
 \pi_\eta^{X\mid Y=y}(\dd x)\propto
 \exp \Bigl\{-V(x)-\frac{\|x-y\|^2}{2\eta}\Bigr\}\,\dd x\,.
 \label{eq:conditional-prox-law}
\end{equation}
Gibbs sampling alternates exact draws from the two conditional laws in
\eqref{eq:conditional-prox-law}.
The scale \(\eta\) controls the strength of the quadratic coupling
potential.
Decreasing \(\eta\) makes \(X\) and \(Y\) more tightly coupled and the
conditional law of \(X\) given \(Y\) more localized and better
conditioned, which reduces the cost of conditional sampling.  The
resulting transitions are also more local, however, and generally require
more Gibbs sampling iterations.  Thus \(\eta\) balances the cost of each
iteration against the number of iterations.
The convergence and query
complexity are studied further in
\cite{fors,chen-chewi-salim-wibisono,fan-yuan-chen}.

The next proposition collects the basic geometric properties of the
augmented distribution used throughout the paper: the curvature and
concentration of the conditional law, the curvature of the
\(Y\)-marginal, and a reflection symmetry of the joint law.

\begin{proposition}[Geometry of the joint law]
\label{prop:augmentation}
The joint law \eqref{eq:augmentation} has the following properties.
\begin{enumerate}[label=\textup{(\roman*)},leftmargin=2.2em]
\item For every \(y\in\R^d\), the conditional law
\(\pi_\eta^{X\mid Y=y}\) has potential (that is, negative
log-density up to an additive constant)
\begin{equation}
 V_y(x)\deq V(x)+\frac{\|x-y\|^2}{2\eta}\,.
 \label{eq:regularized-prox-objective}
\end{equation}
This conditional potential satisfies
\begin{equation}
 (\alpha+\eta^{-1})I
 \preceq\nabla^2V_y(x)
 \preceq(\beta+\eta^{-1})I\,,
 \label{eq:conditional-curvature}
\end{equation}
so its condition number is at most
\((1+\beta\eta)/(1+\alpha\eta)\le2\).  Consequently, for every
smooth \(f\),
\begin{equation}
 \Var_{\pi_\eta^{X\mid Y=y}}(f)
 \le \frac{\eta}{1+\alpha\eta}\,
 \E_{\pi_\eta^{X\mid Y=y}}\|\nabla f\|^2\,,
 \label{eq:conditional-poincare-prelim}
\end{equation}
and
\begin{equation}
 \frac{\eta}{1+\beta\eta}I
 \preceq
 \Cov_{\pi_\eta^{X\mid Y=y}}(X)
 \preceq
 \frac{\eta}{1+\alpha\eta}I\,.
 \label{eq:conditional-covariance-prelim}
\end{equation}
\item The \(Y\)-marginal potential satisfies
\begin{equation}
 \frac{\alpha}{1+\alpha\eta}I
 \preceq\nabla^2U_\eta(y)
 \preceq\frac{\beta}{1+\beta\eta}I\,.
 \label{eq:Y-curvature}
\end{equation}
\item The state-space reflection
\begin{equation}
 \opU : (x,y)\mapsto (x,2x-y)
 \label{eq:easy-reflection}
\end{equation}
preserves \(\pi_\eta\) and is an involution.
\end{enumerate}
\end{proposition}

The Hessian
bound in part~(i) follows immediately from \eqref{eq:SC-smooth}; the
Poincar\'e and covariance bounds then follow from the standard
log-concavity facts \eqref{eq:poincare-tool} and
\cref{lem:covariance-bounds} in Appendix~\ref{app:tools}, applied with
\(m=\alpha+\eta^{-1}\) and \(L=\beta+\eta^{-1}\).  The
marginal curvature bound in part~(ii) is derived in
Appendix~\ref{app:tools}, at \eqref{eq:Moreau-Hessian}.  The
reflection statement follows directly from \eqref{eq:augmentation}.

We use the standard proximal map convention of
Moreau~\cite{moreau}:
\begin{equation}
 \prox_{\eta V}(y)
 \deq\arg\min_{x\in\R^d}
 \Bigl\{V(x)+\frac{\|x-y\|^2}{2\eta}\Bigr\}\,.
 \label{eq:prox-definition}
\end{equation}
The minimizer is unique under \eqref{eq:SC-smooth}.  The objective
minimized here is exactly the conditional potential
\eqref{eq:regularized-prox-objective}, so \(\prox_{\eta V}(y)\) is the
mode of \(\pi_\eta^{X\mid Y=y}\); the conditional draw is in this
sense a stochastic analogue of a proximal step, whence the
terminology arises.  This exact minimizer is used below only in the analysis; the implementable algorithm never evaluates it and
does not assume access to a proximal oracle.
We use the following
conditional sampling primitive, introduced in \cite{lee-shen-tian}.

\begin{definition}[Restricted Gaussian oracle]
For \(\eta>0\), a restricted Gaussian oracle (RGO) takes input
\(y\in\R^d\) and, using fresh randomness, returns a sample with law
\(\pi_\eta^{X\mid Y=y}\).
\end{definition}

\subsection{Divergences and notation}
\label{sec:divergences-notation-prelim}

For probability measures \(P\ll Q\) and \(s>1\), define the R\'enyi
divergence\footnote{Our notation \(D_s(P\|Q)\)
corresponds to \(\mathcal R_s(P\|Q)\) in \cite[Definition~2.2.23]{Chewi26Book}, with the same definition
and normalization; we use $D_s$ here to avoid possible confusion with the reflection operator $R_h$ defined in \eqref{eq:specular}.} of order \(s\) and the \(\chi^2\)-divergence by
\begin{equation}
 D_s(P\|Q)\deq\frac1{s-1}\log\int
 \Bigl(\frac{\dd P}{\dd Q}\Bigr)^s\,\dd Q\,,
 \qquad
 \chi^2(P\|Q)\deq e^{D_2(P\|Q)}-1\,.
 \label{eq:renyi-chi}
\end{equation}
We set \(D_s(P\|Q)=+\infty\) when \(P\not\ll Q\).  We refer to a
bound \(D_2(P\|Q)\le\Delta\) as an order-2 R\'enyi
warm start bound.  Our convention for total variation is
\(
 \TV(P,Q)\deq\sup_A|P(A)-Q(A)|
\).
We use the standard inequality
\begin{equation}
 \TV(P,Q)\le\frac12\sqrt{\chi^2(P\|Q)}\,.
 \label{eq:TV-chi-prelim}
\end{equation}
For any coupling \((X,Y)\) of \(P\) and \(Q\),
\(\TV(P,Q)\le\Prob\{X\ne Y\}\), and a maximal coupling of $X$ and $Y$ attains
equality.

If \(X_0\sim\mu_0\) and, conditionally on \(X_0\),
\(Y_0\sim\cN(X_0,\eta I)\), then, whenever \(\mu_0\ll\mu\),
\[
 \frac{\dd\mathcal L(X_0,Y_0)}{\dd\pi_\eta}(x,y)
 =\frac{\dd\mu_0}{\dd\mu}(x)\,.
\]
Consequently,
\begin{equation}
 D_2\bigl(\mathcal L(X_0,Y_0)\|\pi_\eta\bigr)
 =D_2(\mu_0\|\mu)\,.
 \label{eq:warm-lift-intro}
\end{equation}
The identity also holds in the extended sense when
\(\mu_0\not\ll\mu\), in which case both sides are infinite.

\paragraph{Notation.}
For a random variable \(W\), \(\mathcal L(W)\) denotes its law, and
for a probability measure \(P\) on a product space, \(P^X\) denotes
its \(X\)-marginal. Markov kernels are denoted by calligraphic
symbols. If \(\mathcal K\) is a Markov kernel,
\(P\mathcal K\) denotes the law obtained by applying \(\mathcal K\) to
\(P\), and \(\mathcal K^m\) denotes its \(m\)-fold iterate. Linear operators and state-space transformations are denoted by
sans-serif symbols, scalar functionals by script letters, and
finite-dimensional matrices by ordinary uppercase symbols.  We write
\([u]_+\deq\max\{u,0\}\).  The notation \(\|\cdot\|\) denotes the Euclidean norm for
vectors, the induced operator norm for matrices and bounded linear
operators, and the ambient \(L^2\)-norm for functions; inner products
are interpreted analogously from context. Unsubscripted letters
\(c,C>0\) denote universal constants whose values may change from line
to line, while constants denoted by \(K\), with or without a
descriptive subscript, are fixed within the result in which they are
introduced.  The notation \(\widetilde O(B)\) means \(B\) times a
fixed universal power of logarithms in the displayed problem
parameters; no polynomial dependence is hidden.

\cref{tab:notation} collects the notation used most
frequently in the remainder of the paper.
\begingroup
\footnotesize
\renewcommand{\arraystretch}{1.18}
\setlength{\LTcapwidth}{\textwidth}
\begin{longtable}{@{}
 >{\raggedright\arraybackslash}p{0.14\textwidth}
 >{\raggedright\arraybackslash}p{0.31\textwidth}
 >{\raggedright\arraybackslash}p{0.14\textwidth}
 >{\raggedright\arraybackslash}p{0.31\textwidth}@{}}
\caption{Frequently used notation, grouped by theme.}
\label{tab:notation} \\
\toprule
\textbf{Symbol} & \textbf{Meaning} & \textbf{Symbol} & \textbf{Meaning} \\
\midrule
\endfirsthead
\caption[]{Frequently used notation, grouped by theme (continued).} \\
\toprule
\textbf{Symbol} & \textbf{Meaning} & \textbf{Symbol} & \textbf{Meaning} \\
\midrule
\endhead
\midrule
\multicolumn{4}{r@{}}{\emph{Continued on next page}} \\
\endfoot
\bottomrule
\endlastfoot
\multicolumn{4}{@{}l}{\emph{Target and problem parameters}} \\*
\(d\) & Ambient dimension. &
\(V\) & Potential \(V:\R^d\to\R\). \\
\(\mu\) & Target law \(\mu(\dd x)\propto e^{-V(x)}\,\dd x\). &
\(\alpha,\beta\) & Strong-convexity and smoothness constants of
\(V\). \\
\(\kappa\) & Condition number \(\kappa=\beta/\alpha\). &
\(\varepsilon,\Delta\) & Target total variation accuracy and order-2
R\'enyi warm start budget. \\
\midrule
\multicolumn{4}{@{}l}{\emph{Proximal augmentation}} \\*
\(\eta\) & Proximal scale. &
\(\pi_\eta\) & Augmented target on \((X,Y)\). \\
\(\pi_\eta^X,\pi_\eta^Y\) & Marginal laws of \(\pi_\eta\); in
particular, \(\pi_\eta^X=\mu\). &
\(\pi_\eta^{X\mid Y=y}\) & Conditional target sampled by the restricted
Gaussian oracle. \\
\(U_\eta,V_y\) & \(Y\)-marginal potential and conditional potential,
respectively. &
\(\prox_{\eta V}\) & Proximal map of \(V\). \\
\midrule
\multicolumn{4}{@{}l}{\emph{Algorithmic quantities}} \\*
\(\varphi=\cN(0,I)\) & Standard Gaussian momentum law. &
\(p\) & Temporary momentum. \\
\(\widetilde X\) & Conditional reference point. &
\(\rho\) & Conditional resampling probability. \\
\(R_h\) & Specular reflection with normal \(h\); \(R_0=I\). &
\(N_{\nabla V}\) & Total number of gradient queries. \\
\midrule
\multicolumn{4}{@{}l}{\emph{Kernels, operators, and function spaces}} \\*
\(\mathcal H_y,\opH_y\) & Conditional half-turn kernel and its action on
observables for fixed \(y\). &
\(\opH\) & Half-turn operator lifted to the augmented state space. \\
\(\mathcal K,\opK\) & Full-chain transition kernel and its induced operator
on observables. &
\(\opU\) & Reflection \((x,y)\mapsto(x,2x-y)\) and its pullback on
observables. \\
\(\opPi_p\) & Averaging over the Gaussian momentum. &
\(\opP,\opP^\perp\) & Conditional expectation \(\opP f=\E[f\mid Y]\)
and its orthogonal complement. \\
\(\cH,\cH_{\mathsf P},\cH_\perp\) & The space \(L^2(\pi_\eta)\) and its
macroscopic and microscopic subspaces. &
\(\mathsf T_{\perp\mathsf P}\) & Operator block
\(\opP^\perp\mathsf T\opP\); subscripts give output then input. \\
\pagebreak
\multicolumn{4}{@{}l}{\emph{Laws and divergences}} \\*
\(\mathcal L(W),P^X\) & Law of \(W\) and \(X\)-marginal of \(P\). &
\(P\mathcal K,\mathcal K^m\) & Action of a kernel on a law and its
\(m\)-fold iterate. \\
\(D_s,\chi^2,\TV\) & R\'enyi, chi-squared, and total variation
divergences. &
\(\E,\Prob,\Var,\Cov\) & Expectation, probability, variance, and
covariance. \\
\midrule
\multicolumn{4}{@{}l}{\emph{Constants and conventions}} \\*
\([u]_+\) & Positive part \(\max\{u,0\}\). &
\(c,C\) & Universal constants; values may change from line to line. \\
\(K\) & Result-specific constants, fixed where introduced. &
\(\widetilde O(B)\) & \(B\) times a fixed power of logarithms in the
problem parameters. \\
\end{longtable}
\endgroup

\section{The Proximal Bouncy Particle Sampler}
\label{sec:algorithm}

\subsection{Overview of the construction}
\label{sec:algorithm-overview}

We first describe the Proximal BPS at a high level. \Cref{fig:algorithm} gives a geometric overview.
\begin{figure}[ht]
\centering
\includegraphics[width=0.82\textwidth]{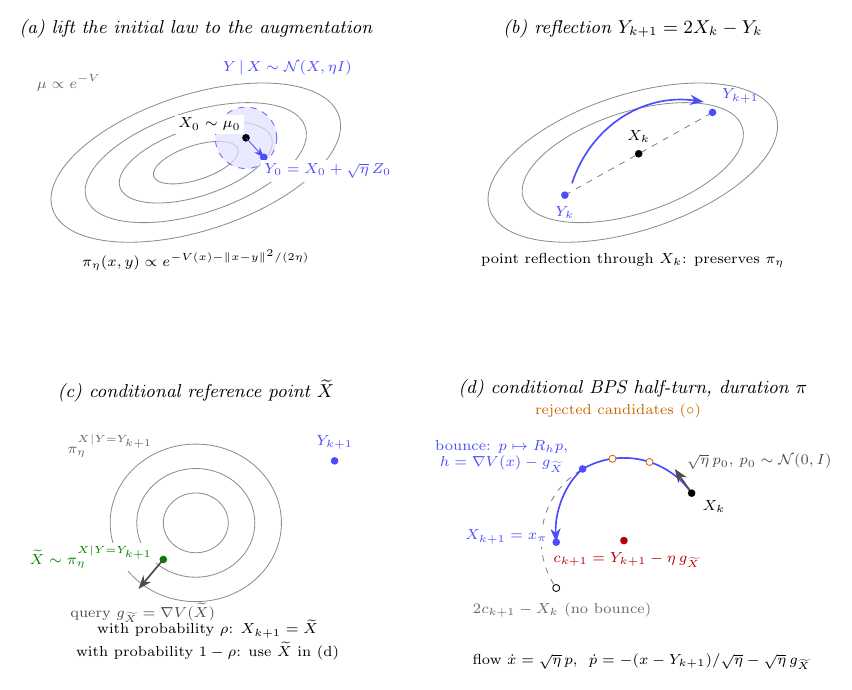}
\caption{Geometric overview of Proximal BPS.
(a) The initial law is lifted to the augmented distribution.  At each
iteration, (b) the auxiliary point is reflected through the current
position, and (c) a conditional reference point is drawn.  With
probability \(\rho\), that point becomes the next position; with
probability \(1-\rho\), it determines the equilibrium position
\(c=Y^+-\eta\nabla V(\widetilde X)\) of the deterministic harmonic flow
for the time-\(\pi\) BPS half-turn process in (d), which starts from the current
position.}
\label{fig:algorithm}
\end{figure}

\paragraph{Baseline Gibbs update and reflection of the auxiliary variable.}

The usual proximal sampler for the augmentation
\eqref{eq:augmentation} is the two-block Gibbs update
\[
 Y^+\sim\cN(x,\eta I)\,,
 \qquad
 X^+\sim\pi_\eta^{X\mid Y=Y^+}\,.
\]
Each update discards the previous value of the coordinate it changes:
the first adds fresh Gaussian noise to \(x\), while the second makes an
independent conditional draw using a restricted Gaussian oracle (RGO).
The latter conditional law is local and well-conditioned by
\eqref{eq:conditional-curvature}, which makes the RGO inexpensive for
small \(\eta\).  The same localization, however, generally requires more
iterations, so \(\eta\) balances the cost of each conditional
update against the number of iterations.  This construction and its RGO
implementation are studied in
\cite{fors,chen-chewi-salim-wibisono,fan-yuan-chen,lee-shen-tian}.

Proximal BPS preserves the same augmented target but replaces this
two-block Gibbs transition with one that retains dependence on the
preceding augmented state.  This memory plays the role of momentum and
underlies the kinetic acceleration established below.  For the
\(Y\)-update, the fresh Gaussian draw is replaced by the deterministic
reflection (see Figure~\ref{fig:algorithm}(b))
\[
 y^+=2x-y\,,
 \qquad
 y^+-x=-(y-x)\,.
\]
This map preserves \(\pi_\eta\), because it negates the centered Gaussian
residual \(Y-X\mid X\) without changing its conditional law. Consequently, the residual from the current augmented
state is carried into the next position update rather than discarded and
replaced by fresh Gaussian noise.  The \(X\)-update described next also
retains dependence on the current position along the dominant branch,
which runs the half-turn process; only occasional conditional resampling
discards this memory.

\paragraph{Conditional reference point and occasional resampling.}

After the reflection sets \(Y=y^+\), Proximal BPS updates \(X\) using a
mixture of two kernels that preserve
\(\pi_\eta^{X\mid Y=y^+}\): independent conditional resampling and the
kernel induced by the conditional half-turn process defined below. Draw
\[
 \widetilde X\sim\pi_\eta^{X\mid Y=y^+}\,,
 \qquad
 X^+=
 \begin{cases}
  \widetilde X\,,
    & \text{with probability }\rho
      \quad\text{(conditional resampling)}\,,\\
  x_\pi\,,
    & \text{with probability }1-\rho
      \quad\text{(half-turn process, defined below)}\,.
 \end{cases}
\]
Choosing \(\rho\)
small makes the half-turn process the dominant option,
while conditional resampling prevents periodicity.
When conditional resampling is selected, the RGO output \(\widetilde X\)
is the next position. On the dominant branch, the process uses
\(\widetilde X\) as an auxiliary reference point, and its output \(x_\pi\)
remains dependent on the current position.
\Cref{thm:ideal-convergence} shows that composing this update with the
reflection \(y^+=2x-y\) accelerates convergence and thus reduces the
iteration complexity, as the
updates carry information across iterations.

\paragraph{The conditional half-turn process.}

This process updates \(X\) through a kernel that preserves
\(\pi_\eta^{X\mid Y=y^+}\) while retaining dependence on the current
position, rather than replacing \(x\) by an independent conditional
draw.  For this purpose, we construct a stochastic process on phase
space \((x,p)\) with invariant law
\(\pi_\eta^{X\mid Y=y^+}\otimes\cN(0,I)\).  Each run starts from
\(x_0=x\) and an independent \(p_0\sim\cN(0,I)\).  This momentum is
temporary: it is discarded at the end of the process, whereas the memory
across iterations remains in \((X,Y)\).

At fixed \(Y=y^+\), vanilla BPS would use straight-line motion and bounce against the full gradient of
\(V(\cdot)+\|{\cdot}-y^+\|^2/(2\eta)\).  Although exact, this does not exploit the fact that we
know the quadratic coupling potential exactly, which dominates \(V\) when \(\eta\) is
small.  We thus incorporate this quadratic term into the deterministic
motion between bounce events, which account for \(V\).

A direct version of this idea would still bounce against the full
gradient \(\nabla V(x)\), whose magnitude need not be small even inside
the conditional region localized near $y^+$.
This is where the conditional reference point \(\widetilde X\) is used.
The fixed gradient \(\nabla V(\widetilde X)\) is incorporated into the
deterministic motion, whose equilibrium position is
\(c\deq y^+-\eta\nabla V(\widetilde X)\).  Bounces handle only
\(\nabla V(x)-\nabla V(\widetilde X)\), which makes the bounce rate small
due to smoothness of the potential.

Ignoring bounces, the harmonic motion traces a periodic orbit around
\(c\) and returns to its initial point after time \(2\pi\).  By a
\emph{half-turn} we mean stopping halfway around this orbit, at time
\(\pi\).  See Figure~\ref{fig:algorithm}(d). This midpoint is chosen to preserve memory.  The temporary
momentum is needed to define invariant phase space dynamics, but it
should not replace the information carried by the current position.
If we stop after a quarter-period, the new position
would be determined by the freshly drawn momentum rather than by
\(x_0\).  Continuing through a full period would undo the move and
return to \(x_0\).  The half-period is the first non-trivial stopping time
that avoids both failures: the fresh momentum no longer contributes to
the position, and the displacement from \(c\) is retained with its
direction reversed.  Thus the no-bounce map is
\[
 (x_\pi,p_\pi)=(2c-x_0,-p_0)\,.
\]
The BPS bounces correct for \(V\) along the trajectory, but the process
keeps this time-\(\pi\) horizon so that its dominant harmonic motion
retains the same memory-preserving geometry.

\medskip

The rest of this section formalizes the construction.  We reserve
\(t\in[0,\pi]\) for the phase space process \((x_t,p_t)\) during one run
of the half-turn process
and \(k\) for the Markov chain \((X_k,Y_k)\).
\Cref{sec:ideal-half-turn} constructs the conditional half-turn process
and proves reversibility of its induced kernel;
\cref{sec:ideal-proximal-bps} defines the full ideal
transition. The practical gradient-only
implementation is discussed in \cref{sec:implementation}.

\subsection{The ideal conditional half-turn process}
\label{sec:ideal-half-turn}

Fix \(Y=y\).  Recall from \eqref{eq:conditional-prox-law} that the position
\(X\) to be sampled has conditional law
\[
 \pi_\eta^{X\mid Y=y}(\dd x)
 \propto
 e^{-V_y(x)}\,\dd x\,,
 \qquad
 V_y(x)\deq V(x)+\frac{\|x-y\|^2}{2\eta}\,.
\]
We construct the half-turn process so that its induced transition
preserves this law.  The
construction starts from vanilla BPS for the conditional target, absorbs
the quadratic coupling potential into the deterministic motion, and then
introduces
an independent conditional reference point to localize the remaining
bounce rate.  We begin by introducing a Gaussian momentum
\(p\sim\cN(0,I)\).  The phase space law to be preserved is
\begin{equation}
 \nu_{\eta,y}(\dd x\,\dd p)
 \deq\pi_\eta^{X\mid Y=y}(\dd x)\,\cN(0,I)(\dd p)
 \propto
 \exp \Bigl\{-V_y(x)-\frac{\|p\|^2}{2}\Bigr\}\,\dd x\,\dd p\,.
 \label{eq:conditional-position-momentum-measure}
\end{equation}
Recall from \eqref{eq:specular} that \(R_h\) denotes specular
reflection across the hyperplane orthogonal to \(h\).

A natural first attempt is to apply vanilla BPS directly to this
conditional distribution.  Its infinitesimal generator acts on a smooth
test function \(f\) as
\begin{equation}
 \begin{split}
 (\opL_y^{\rm vanilla}f)(x,p)
 &=\sqrt\eta\,p^\top\nabla_xf(x,p)\\
 &\quad+\sqrt\eta\,[p^\top\nabla V_y(x)]_+
 \bigl\{f(x,R_{\nabla V_y(x)}p)-f(x,p)\bigr\}\,.
 \end{split}
 \label{eq:vanilla-conditional-bps-generator}
\end{equation}
The common factor \(\sqrt\eta\) only fixes the unit of time; it does not
change the invariant law.  We choose this unit so that the harmonic motion
introduced below has period \(2\pi\).
The first term gives straight-line motion, while the second reflects the
momentum against the full conditional gradient.  This process is exact,
but it can be expensive: even when \(V=0\), the quadratic coupling potential
\(\|{\cdot}-y\|^2/(2\eta)\) produces a stationary mean bounce rate of order
\(\sqrt d\).

To reduce the number of bounce events, we put the quadratic part of the
conditional gradient into the deterministic dynamics.  This gives the
generator
\begin{equation}
 \begin{split}
 (\opL_y^{\rm harm}f)(x,p)
 &=\sqrt\eta\,p^\top\nabla_xf(x,p)
   -\frac{(x-y)^\top}{\sqrt\eta}\nabla_pf(x,p)\\
 &\quad+\sqrt\eta\,[p^\top\nabla V(x)]_+
 \bigl\{f(x,R_{\nabla V(x)}p)-f(x,p)\bigr\}\,.
 \end{split}
 \label{eq:direct-harmonic-bps-generator}
\end{equation}
The first line is the Hamiltonian flow associated with the quadratic
coupling potential: between bounces, it rotates \((x-y)/\sqrt\eta\) and
\(p\).  This use of a quadratic reference flow, with the remaining force
handled by bounce events, is analogous to the Boomerang
sampler~\cite{boomerang}.
The direct harmonic generator \eqref{eq:direct-harmonic-bps-generator}
still bounces against the full gradient \(\nabla V\), so its bounce rate
can remain large.

To further reduce the bounce rate and the resulting computational
cost, we introduce a conditional reference point and use the smoothness of
\(V\) to control the local gradient difference.  Specifically, draw
\(\widetilde X\sim\pi_\eta^{X\mid Y=y}\) independently and write the exact
decomposition
\[
 \nabla V(x)
 =\nabla V(\widetilde X)
  +\{\nabla V(x)-\nabla V(\widetilde X)\}\,.
\]
Conditionally on \(\widetilde X=\widetilde x\), we incorporate the fixed
reference term into the deterministic dynamics and retain only the gradient
difference in the BPS jump term.  Define the equilibrium position of this
deterministic harmonic flow and the residual gradient by
\begin{equation}
 c_{y,\widetilde x}\deq y-\eta\nabla V(\widetilde x)\,,
 \qquad
 h_{\widetilde x}(x)\deq\nabla V(x)-\nabla V(\widetilde x)\,.
 \label{eq:original-coordinate-equilibrium-residual}
\end{equation}
The resulting generator is
\begin{equation}
 \begin{split}
 \opL_{y,\widetilde x}f(x,p)
 &\deq\sqrt\eta\,p^\top\nabla_xf(x,p)
 -\frac{(x-c_{y,\widetilde x})^\top}{\sqrt\eta}\nabla_pf(x,p)\\
 &\quad+\sqrt\eta\,[p^\top h_{\widetilde x}(x)]_+
 \bigl\{f(x,R_{h_{\widetilde x}(x)}p)-f(x,p)\bigr\}\,.
 \end{split}
 \label{eq:fiber-generator}
\end{equation}
The harmonic and residual terms sum to the full conditional force
\(\sqrt\eta\,\nabla V_y(x)\).  Thus \(\widetilde x\) changes only how this
force is divided between deterministic motion and bounces.  By the
\(\beta\)-smoothness of \(V\),
\[
 \|h_{\widetilde x}(x)\| = \|\nabla V(x)-\nabla V(\widetilde x)\|
 \le\beta\,\|x-\widetilde x\|\,.
\]
Drawing the reference point from the same localized conditional law keeps
the residual gradient small and reduces the expected number of bounces per
run of the half-turn process to order \(\beta\eta\sqrt d\); see
\cref{prop:ideal-bounce-count}.

Having specified the dynamics, it remains to choose the integration
horizon.  For fixed \((y,\widetilde x)\), the deterministic part of
\eqref{eq:fiber-generator} is harmonic motion about its equilibrium
position \(c_{y,\widetilde x}\).  In the absence of bounces, its position at
time \(t\) satisfies
\begin{equation}
 x_t-c_{y,\widetilde x}
 =(x_0-c_{y,\widetilde x})\cos t
  +\sqrt\eta\,p_0\sin t\,.
 \label{eq:no-bounce-harmonic-position}
\end{equation}
As discussed earlier, we choose the integration horizon \(t=\pi\), corresponding to one
half-period.  This is the first non-zero harmonic time at which the
position \(x_\pi\) no longer depends on the initial momentum \(p_0\).
Without bounces, the resulting map is
\begin{equation}
 (x_\pi,p_\pi)
 =(2c_{y,\widetilde x}-x_0,-p_0)\,.
 \label{eq:no-bounce-half-period-map}
\end{equation}
Thus, the position is reflected through the equilibrium position
\(c_{y,\widetilde x}\) and retains its dependence on \(x_0\).  Bounce
events correct for the residual force along the trajectory, while the
fixed horizon \(t=\pi\) retains the memory-preserving structure of the
half-turn process.

The construction is now complete.  The following algorithm gives its
pathwise prescription: it draws the reference point and momentum, evolves
the process with generator \eqref{eq:fiber-generator} for time \(\pi\), and
returns only the final position.

\begin{algorithm}[H]
\caption{Simulation of the ideal conditional half-turn process}
\label{alg:ideal-half}
\begin{algorithmic}[1]
\Statex \textbf{Input:} state \((x,y)\)
\State Draw \(\widetilde X\sim\pi_\eta^{X\mid Y=y}\) and
       \(p_0\sim\cN(0,I)\) independently.
\State Query \(\nabla V(\widetilde X)\).
\State Set \(x_0\gets x\).
\State Evolve \((x_t,p_t)\) according to
\begin{equation}
 \dot x=\sqrt\eta\,p\,,
 \qquad
 \dot p=-\frac{x-y}{\sqrt\eta}
 -\sqrt\eta\,\nabla V(\widetilde X)\,,
 \label{eq:x-harmonic-flow}
\end{equation}
until \(t=\pi\).
\State During the evolution, generate bounce events at rate
\begin{equation}
 \lambda_{\widetilde X}(x,p)
 =\sqrt\eta\,\bigl[p^\top[\nabla V(x)-\nabla V(\widetilde X)]\bigr]_+\,.
 \label{eq:x-bounce-rate}
\end{equation}
\State At each bounce, set
       \(p_t\gets
       R_{\nabla V(x_t)-\nabla V(\widetilde X)}p_t\).
\State \Return \(x_\pi\)
\end{algorithmic}
\end{algorithm}

\begin{proposition}[Well-posedness and stationarity of the conditional phase space process]
\label{prop:fiber-construction}
For every fixed \(y,\widetilde x\in\R^d\) and initial phase space point
\((x_0,p_0)\), the event-driven dynamics in \cref{alg:ideal-half}, with
the reference point fixed at \(\widetilde x\), defines a unique
non-explosive, time-homogeneous Markov process.  The phase space measure
\(\nu_{\eta,y}\) is stationary for this process.  In particular, the
returned position \(x_\pi\) is well-defined almost surely.
\end{proposition}

The well-posedness and stationarity assertions are proved in
\cref{app:fixed-reference-process}.

The induced Markov kernel on positions also satisfies the following
reversibility property.

\begin{proposition}[Reversibility of the conditional half-turn kernel]
\label{prop:half-invariance}
For every \(y\in\R^d\), let \(\mathcal H_y\) be the Markov kernel on
\(\R^d\) induced by \cref{alg:ideal-half}; thus
\(\mathcal H_y(x,\cdot)\) is the law of the returned position \(x_\pi\)
for input \((x,y)\).  Then \(\mathcal H_y\) is reversible with respect to
\(\pi_\eta^{X\mid Y=y}\).
\end{proposition}

The proof is deferred to Appendix~\ref{app:proof-half-invariance}.

The following proposition quantifies the stationary bounce count of the
ideal half-turn process.

\begin{proposition}[Stationary bounce count]
\label{prop:ideal-bounce-count}
There exists a universal constant \(K_{\rm id}\ge1\) such that the
following holds.  Suppose that \((x_0,Y)\sim\pi_\eta\).  Conditional on
\(Y\), let \(\widetilde X\sim\pi_\eta^{X\mid Y}\) be independent of
\(x_0\), and let \(p_0\sim\cN(0,I)\) be independent of all position
variables.  If \(N_{\rm b}^{\rm id}\) denotes the number of bounce events
in the resulting ideal half-turn process of \cref{alg:ideal-half}, then
\begin{equation}
 \E N_{\rm b}^{\rm id}
 \le K_{\rm id}\beta\eta\sqrt d\,.
 \label{eq:ideal-bounce-count}
\end{equation}
\end{proposition}

The proof is deferred to
\hyperref[app:proof-ideal-bounce-count]{Appendix~\ref*{app:halfturn-process},
proof of \cref*{prop:ideal-bounce-count}}.

\subsection{Ideal Proximal BPS}
\label{sec:ideal-proximal-bps}

With the conditional half-turn kernel in hand, we can now define the ideal
Proximal BPS.  For a resampling probability \(\rho\in(0,1)\), specified in \cref{thm:ideal-convergence}, its
augmented-state transition kernel is
\begin{equation}
 \begin{split}
 \mathcal K\bigl((x,y),\dd x'\,\dd y'\bigr)
 &\deq
 \delta_{2x-y}(\dd y')
 \,\bigl[
  \rho\,\pi_\eta^{X\mid Y=y'}(\dd x')
  +(1-\rho)\,\mathcal H_{y'}(x,\dd x')
 \bigr]\,.
 \end{split}
 \label{eq:ideal-proximal-bps-kernel}
\end{equation}
The Dirac factor enforces \(y'=2x-y\); the two terms in brackets are
conditional resampling and the half-turn kernel.  Thus
\eqref{eq:ideal-proximal-bps-kernel} is the formal version of the update
displayed in \cref{sec:algorithm-overview}.  The following algorithm
gives an equivalent simulation.

\begin{algorithm}[H]
\caption{Ideal Proximal BPS}
\label{alg:ideal-main}
\begin{algorithmic}[1]
\Statex \textbf{Input:} \((X_0,Y_0)\), \(T\), and \(\rho\)
\For{\(k=0,\ldots,T-1\)}
  \State \(Y_{k+1}\gets2X_k-Y_k\)
  \State Draw \(B_k\sim\operatorname{Bernoulli}(\rho)\)
  \If{\(B_k=1\)}
    \State Draw
           \(X_{k+1}\sim\pi_\eta^{X\mid Y=Y_{k+1}}\) exactly
  \Else
    \State \(X_{k+1}\gets
           \Call{IdealHalfTurn}{X_k,Y_{k+1}}\)
           using \cref{alg:ideal-half}
  \EndIf
\EndFor
\State \Return \(X_T\)
\end{algorithmic}
\end{algorithm}

\begin{proposition}[Invariance of the ideal Proximal BPS]
\label{prop:ideal-invariance}
The transition kernel \(\mathcal K\) in
\eqref{eq:ideal-proximal-bps-kernel} preserves the augmented law \(\pi_\eta\).
\end{proposition}

\begin{proof}
The reflection preserves \(\pi_\eta\) by \cref{prop:augmentation}.
At fixed \(y'\), the first term in brackets in
\eqref{eq:ideal-proximal-bps-kernel} preserves
\(\pi_\eta^{X\mid Y=y'}\) by definition, and the second does so by
\cref{prop:half-invariance}.  The conditional mixture, and hence its
composition with the reflection, preserves \(\pi_\eta\).
\end{proof}

The next theorem gives the quantitative convergence guarantee for
ideal Proximal BPS and also underlies the implementable guarantee
in \cref{sec:implementation}.  Its proof is in
Appendices~\ref{sec:micro-macro} and~\ref{app:key-estimates}.

\begin{theorem}[$\chi^2$ convergence]
\label{thm:ideal-convergence}
There exist universal constants
\(c_0,\rho_\star\in(0,1)\) and \(C\ge1\) such that
the following holds.  Assume \eqref{eq:SC-smooth}, and let
\(0<\eta\le c_0/\beta\).
Set
\begin{equation}
 \ell_\eta\deq\log\frac{e}{\alpha\eta}\,,
 \qquad
 \rho\deq\rho_\star\sqrt{\alpha\eta\,\ell_\eta}\in(0,1/2]\,,
 \label{eq:rho-choice}
\end{equation}
and let \(\mathcal K\) be the transition kernel in
\eqref{eq:ideal-proximal-bps-kernel} with this choice of \(\rho\).
Then, for all integers $N$ satisfying
\begin{equation}
 N\geq\frac{C}{\sqrt{\alpha\eta}}
 \sqrt{\log\frac{e}{\alpha\eta}}\,,
 \label{eq:N-main-bound}
\end{equation}
all probability measures \(\nu_0\) on
\(\R^d\times\R^d\) satisfying
\(\chi^2(\nu_0\|\pi_\eta)<\infty\), and all integers \(J\ge0\),
\begin{equation}
 \chi^2\bigl(\nu_0 \mathcal K^{JN}\|\pi_\eta\bigr)
 \le
 e^{-J}
 \chi^2(\nu_0\|\pi_\eta)\,\,.
 \label{eq:block-chi}
\end{equation}
Consequently, the \(X\)-marginal satisfies
\begin{equation}
 \TV\bigl((\nu_0 \mathcal K^{JN})^X,\mu\bigr)
 \le\frac12\sqrt{\chi^2(\nu_0\|\pi_\eta)}
 \exp \Bigl\{-\frac{J}{2}\Bigr\}\,.
 \label{eq:block-TV}
\end{equation}
\end{theorem}

\Cref{thm:ideal-convergence} is a hypocoercive estimate.  To see this,
let \(f\in L^2(\pi_\eta)\) have mean zero and write
\[
 \bar f(y)\deq\E_{\pi_\eta}[f(X,Y)\mid Y=y]\,,
 \qquad
 \widetilde f(x,y)\deq f(x,y)-\bar f(y)\,.
\]
Then, \(\E_{\pi_\eta}[\widetilde f(X,Y)\mid Y]=0\) and
\(\|f\|_2^2=\|\bar f\|_2^2+\|\widetilde f\|_2^2\), where
\(\|\cdot\|_2\) denotes the norm in \(L^2(\pi_\eta)\).  At fixed
\(Y\), conditional resampling removes \(\widetilde f\), whereas the
conditional half-turn leaves \(\bar f\) unchanged and does not increase
\(\|\widetilde f\|_2\).  Since reflection preserves the \(L^2\) norm,
the one-step Markov operator \(\opK\) satisfies
\[
 \|\opK f\|_2^2-\|f\|_2^2
 \le-\rho\,\|\widetilde f\|_2^2\,.
\]
The ordinary \(L^2\) estimate therefore gives no decay when
\(f=\bar f\).

It remains to understand observables that depend only on \(Y\).  For
\(f=\bar f\), reflection changes \(\bar f(Y)\) into
\(\bar f(2X-Y)\), which can be decomposed as
\[
 \bar f(2X-Y)
 =\E_{\pi_\eta}[\bar f(2X-Y)\mid Y]+r_f(X,Y)\,,
 \qquad
 \E_{\pi_\eta}[r_f(X,Y)\mid Y]=0\,.
\]
The first term still depends only on \(Y\), whereas \(r_f\) averages to
zero when \(X\) is resampled conditionally on \(Y\).  Since \(f\) is
centered, so is \(\bar f\), and \cref{lem:macroscopic-coercivity} shows
that the squared norm of the first term is at most
\((1-c\alpha\eta)\,\|\bar f\|_2^2\) for a universal constant \(c>0\).
Reflection nevertheless preserves the total norm: the decrease
of the first component is balanced by the contribution from \(r_f\).
Thus the usual \(L^2\) norm sees only a redistribution between the two
components, even though conditional resampling contracts the second one.
This interaction is the hypocoercive contraction 
mechanism~\cite{dms,lu-wang}.

To capture this mechanism, we add a corrector to the ordinary
squared \(L^2\) norm and use the modified Lyapunov function
\[
 \mathscr L_\omega(f)\deq\|f\|_2^2+\omega\mathscr C(f)\,.
\]
The corrector is chosen so that its one-step change contributes a negative
term proportional to \(\|\bar f\|_2^2\), together with terms controlled by
\(\|\widetilde f\|_2^2\).  Combined with the direct \(L^2\) estimate
above, this gives control of both components.  The model-specific input
for the half-turn is \cref{prop:global-halfturn-cutoff}, which bounds its
contribution to the change of the corrector.
This is a discrete version of the modified-\(L^2\) method of Dolbeault,
Mouhot, and Schmeiser~\cite{dms} (see also the sharpened formulation of Fan, Li, and
Lu~\cite{fan-li-lu}).

In the appendices, we denote by \(\opP\) the
conditional expectation given \(Y\) and set \(\opP^\perp\deq I-\opP\), and let
\(\opU\) denote pullback by the reflection \(Y\mapsto2X-Y\).  Then
\[
 f_{\mathsf P}\deq\opP f=\bar f\,,
 \qquad
 f_\perp\deq\opP^\perp f=\widetilde f\,,
 \qquad
 r_f=\opP^\perp\opU\opP f\,.
\]
The proof steps are organized as follows.  Appendix~\ref{app:operator-framework}
sets up the decomposition above and derives the ordinary \(L^2\) estimate.
Appendix~\ref{app:reflection-halfturn-estimates} states the two estimates
used to recover control of \(f_{\mathsf P}\).  The reflection estimate is proved
in Appendix~\ref{app:proof-coercivity}, and the half-turn estimate in
Appendices~\ref{sec:affine-halfturn}--\ref{app:global-halfturn-cutoff}.
Finally, Appendix~\ref{sec:block} constructs the modified energy, proves
its contraction, chooses \(\rho\) and the block length, and converts the
result into the stated \(\chi^2\) and total variation bounds.

\begin{remark}
    The use of $L^2$ analysis is partially motivated by~\cite{MonWan26EntropicPDMP}, which gave obstructions to entropic convergence for piecewise-deterministic Markov processes such as the BPS\@.
\end{remark}

\section{Efficient implementation of Proximal BPS}
\label{sec:implementation}

This section gives a gradient-only implementation of the ideal Proximal BPS
from \cref{sec:ideal-proximal-bps}.  It combines two ingredients.  The
first is sampling from the conditional law
\(\pi_\eta^{X\mid Y=y}\).  For this component, we import the first-order
rejection sampler and its RGO guarantee from~\cite{fors}; we specialize
that guarantee to the present conditional law, construct the required
approximate proximal point, and verify the prescribed first-order
residual bound.  The second
ingredient is the conditional half-turn process, whose bounce events we
simulate with a capped rate.  We then combine the imported RGO
approximation with the capped rate approximation of this process and bound
both the resulting total variation error and the number of gradient
queries.

\subsection{Implementing the restricted Gaussian oracle}
\label{sec:implementing-rgo}

We import both the first-order rejection algorithm and its analysis
from~\cite{fors}.  The theorem below records the specialization of their
RGO guarantee needed here.

\begin{theorem}[Restricted Gaussian oracle guarantee]
\label{thm:rgo-guarantee}
There is a universal constant \(K_{\rm RGO}\ge1\) such that the
following holds.
Let \(\eta>0\), \(y\in\R^d\), \(\delta\in(0,1/2)\), and suppose that
\(\widehat x_+\in\R^d\) satisfies
\begin{equation}
 \|y-\eta\nabla V(\widehat x_+)-\widehat x_+\|
 \le\sqrt{d\eta}\,.
 \label{eq:approximate-prox-condition}
\end{equation}
Then the first-order rejection sampler of~\cite[Theorems 3.1 and
D.1]{fors} can be instantiated so that its output law \(\widehat\pi\)
satisfies
\begin{equation}
 \chi^2 \bigl(\pi_\eta^{X\mid Y=y}\bigm\|\widehat\pi\bigr)
 \le\delta^2\,,\qquad
 \chi^2 \bigl(\widehat\pi\bigm\|\pi_\eta^{X\mid Y=y}\bigr)
 \le\delta^2\,,
 \label{eq:RGO-chi-accuracy}
\end{equation}
using at most \(K_{\rm RGO}\) expected queries to \(\nabla V\),
excluding the construction of \(\widehat x_+\), provided that
\begin{equation}
 \eta^{-1}\ge
 K_{\rm RGO}\beta\,
 \{\sqrt{d\log(1/\delta)}+\log(1/\delta)\}\,.
 \label{eq:rgo-scale-condition}
\end{equation}
\end{theorem}

The quantity in \eqref{eq:approximate-prox-condition} is exactly
\(\eta\,\|\nabla V_y(\widehat x_+)\|\), where $V_y=V+\frac{\|\cdot-y\|^2}{2\eta}$.  When
\(\eta\le1/\beta\), as in the application below,
\eqref{eq:conditional-curvature} shows that \(V_y\) is
\((\alpha+\eta^{-1})\)-strongly convex and
\((\beta+\eta^{-1})\)-smooth, with condition number at most two.
For example, gradient descent with step size
\(2/(\alpha+\beta+2\eta^{-1})\) contracts the distance to the
minimizer by a universal factor at every iteration.  The strong
convexity and smoothness bounds then convert this distance contraction
into the same geometric decrease, up to a factor of two, for the
quantity \(\eta\,\|\nabla V_y\|\).  Consequently, there is a universal
constant \(K_{\rm prox}\ge1\) such that this method, initialized at
\(x\), finds a point \(\widehat x_+\) satisfying
\eqref{eq:approximate-prox-condition} using at most
\begin{equation}
 K_{\rm prox}\,\Bigl[1+\log \Bigl(
 1+\frac{\|y-\eta\nabla V(x)-x\|}{\sqrt{d\eta}}
 \Bigr)\Bigr]
 \label{eq:first-order-prox-cost}
\end{equation}
gradient queries~\cite{nesterov}.  A gradient query certifies the
stopping condition, so neither function values nor an exact proximal
oracle are needed.  We use this routine only through
\eqref{eq:first-order-prox-cost}; all of its queries, as well as those
of the rejection sampler, are included in \(N_{\nabla V}\).

\subsection{Simulating the half-turn process}
\label{sec:simulating-bounces}

For convenience, recall from \eqref{eq:x-bounce-rate} that the ideal
bounce rate along the half-turn process is
\[
 \lambda_t\deq\lambda_{\widetilde X}(x_t,p_t)
 =\sqrt\eta\,
 \bigl[p_t^\top[\nabla V(x_t)-\nabla V(\widetilde X)]\bigr]_+\,.
\]
This rate is not uniformly bounded along every path.
Given a deterministic cap
\(\bar\lambda>0\), we therefore replace it by the truncated rate
\(\min\{\lambda_t,\bar\lambda\}\) and simulate the resulting process
by the homogeneous Poisson thinning construction reviewed in
\cref{sec:standard-bps-prelim}~\cite{lewis-shedler}.  Let \(X_\pi^{\rm id}\) and
\(X_\pi^{\rm cap}\) denote the time-\(\pi\) position outputs of the ideal
and capped half-turn processes, respectively,
and let \(N_{\nabla V}^{\rm cap}\) count the gradient queries used by
the capped simulation, assuming that \(\nabla V(\widetilde X)\) has
already been queried.  For the ideal-versus-capped comparison below,
the two processes use the same reference point \(\widetilde X\), drawn
exactly from \(\pi_\eta^{X\mid Y}\) independently of the current
position conditional on \(Y\).  This exact draw assumption only
isolates the capped rate error; \cref{sec:implementable-main} combines
the analysis with the approximate RGO of \cref{sec:implementing-rgo}
by coupling.

Both the ideal and capped processes evolve in continuous time without
time discretization; the latter differs only through its truncated
bounce rate.
For \(c_{y,\widetilde x}=y-\eta\nabla V(\widetilde x)\), the harmonic
flow over an interval of length \(u\) is given exactly by
\begin{equation}
 \begin{split}
 x_{t+u}-c_{y,\widetilde x}
 &=(x_t-c_{y,\widetilde x})\cos u
   +\sqrt\eta\,p_t\sin u\,,\\
 p_{t+u}
 &=-\eta^{-1/2}(x_t-c_{y,\widetilde x})\sin u
   +p_t\cos u\,.
 \end{split}
 \label{eq:explicit-harmonic-flow}
\end{equation}

\begin{algorithm}[H]
\caption{Capped-rate simulation of the conditional half-turn process}
\label{alg:capped-half}
\begin{algorithmic}[1]
\Statex \textbf{Input:} \((x,y)\), reference point \(\widetilde x\)
        with cached \(\nabla V(\widetilde x)\), and rate cap
        \(\bar\lambda\)
\State Draw \(p_0\sim\cN(0,I)\)
\State Set \(x_0\gets x\), \(t\gets0\), and
       \(c_{y,\widetilde x}\gets y-\eta\nabla V(\widetilde x)\)
\While{\(t<\pi\)}
  \State Draw a candidate waiting time
         \(u\sim\operatorname{Exp}(\bar\lambda)\)
  \State Set \(\bar u\gets\min\{u,\pi-t\}\)
  \State Propagate \((x_t,p_t)\) for time \(\bar u\) using
         \eqref{eq:explicit-harmonic-flow}
  \State \(t\gets t+\bar u\)
  \If{\(t<\pi\)}
    \State Query \(\nabla V(x_t)\) and set
           \(\lambda_t\gets\lambda_{\widetilde x}(x_t,p_t)\)
    \State Draw \(U\sim\operatorname{Unif}[0,1]\)
    \If{\(U\le\min\{1,\lambda_t/\bar\lambda\}\)}
      \State \(p_t\gets
             R_{\nabla V(x_t)-\nabla V(\widetilde x)}p_t\)
    \EndIf
  \EndIf
\EndWhile
\State \Return \(x_\pi\)
\end{algorithmic}
\end{algorithm}

The cap \(\bar\lambda\) is not required to dominate the ideal rate
\(\lambda_t\).  Thus, \cref{alg:capped-half} exactly simulates the
half-turn process with rate \(\min\{\lambda_t,\bar\lambda\}\), which serves as
an approximation to the ideal half-turn process.  The next proposition bounds
the coupling error introduced by truncation and the gradient queries
used by the capped simulation.

\begin{proposition}[Capped-rate approximation and query cost]
\label{prop:capped-rate}
There exist universal constants \(c_{\rm evt}\in(0,1)\) and
\(K_{\rm evt}\ge1\) such that the following holds whenever
\(\beta\eta\sqrt d\le c_{\rm evt}\).
\begin{enumerate}[label=\textup{(\roman*)},leftmargin=2.2em]
\item Let \(\xi\in(0,1/2)\), \(\Delta\ge1\), and set
\begin{equation}
 \bar\lambda_{\xi,\Delta}\deq
 K_{\rm evt}\beta\eta\,\bigl(\sqrt{d\ell_{\xi,\Delta}}+
 \ell_{\xi,\Delta}\bigr)\,,
 \qquad\text{where}\qquad
 \ell_{\xi,\Delta}\deq
 K_{\rm evt}\,\Bigl[\Delta+\log
 \frac1\xi\Bigr]\,.
 \label{eq:clock-cap}
\end{equation}
For every probability measure \(\nu_0\ll\pi_\eta\) satisfying
\(D_2(\nu_0\|\pi_\eta)\le\Delta\), consider the ideal half-turn process and
the capped half-turn process with rate cap
\(\bar\lambda=\bar\lambda_{\xi,\Delta}\), both started from
\((X_0,Y)\sim\nu_0\).  They can be coupled so that
\[
 \Prob_{\nu_0} \bigl\{X_\pi^{\rm id}\ne X_\pi^{\rm cap}\bigr\}
 \le\xi\,.
\]
\item For the same cap, every initial state, and every reference point,
the expected number of gradient queries used by the capped simulation
satisfies
\begin{equation}
 \E N_{\nabla V}^{\rm cap}
 =\pi\bar\lambda_{\xi,\Delta}
 =\pi K_{\rm evt}\beta\eta\,\bigl(
 \sqrt{d\ell_{\xi,\Delta}}+\ell_{\xi,\Delta}\bigr)\,.
 \label{eq:half-query-cost}
\end{equation}
\end{enumerate}
\end{proposition}

The proof is deferred to
\hyperref[app:proof-capped-rate]{Appendix~\ref*{app:halfturn-process},
proof of \cref*{prop:capped-rate}}.

\subsection{Implementable Proximal BPS and total complexity}
\label{sec:implementable-main}

\Cref{alg:implementable-main} assembles the preceding components into
the complete gradient-only Proximal BPS algorithm.  In each transition,
the approximate RGO produces an
approximate conditional draw \(\widetilde X\), which is reused as in the ideal
construction: it becomes the next position on the resampling branch and
the reference point for the capped half-turn process in
\cref{alg:capped-half}.
Here,
\textsc{FirstOrderProx}\((x,y,Q)\) denotes a first-order solver for
\(V_y\), initialized at \(x\), that either returns a point satisfying
\eqref{eq:approximate-prox-condition} within \(Q\) gradient queries or
reports failure.

A reported failure means only that the fixed query budget
was exhausted before the residual condition was certified; it can occur
because the initial residual in \eqref{eq:first-order-prox-cost} is
unbounded over the state space.  The sampler then returns its current
position, a fallback that preserves the
proximal solver's query cap while requiring no additional queries.
The analysis treats the entire failure event as a coupling error, the fallback is not required to satisfy any accuracy constraints.  The call
\textsc{ApproximateRGO}\((y,\widehat x_+,a)\) denotes the sampler in \cref{thm:rgo-guarantee}.

The parameter choices following \cref{alg:implementable-main}
jointly control the proximal solver, RGO, and capped-rate errors.
\Cref{thm:main-general} then gives the resulting accuracy and expected
query complexity.

\begin{algorithm}[H]
\caption{Implementable Proximal BPS}
\label{alg:implementable-main}
\begin{algorithmic}[1]
\Statex \textbf{Input:} \(X_0\sim\mu_0\), where
        \(D_2(\mu_0\|\mu)\le\Delta\)
\Statex \textbf{Parameters:}
       \(\eta,T,a,\rho,\bar\lambda,Q_{\rm prox}\)
\State Draw \(Z_0\sim\cN(0,I)\) independently and set
       \(Y_0\gets X_0+\sqrt\eta Z_0\)
\For{\(k=0,\ldots,T-1\)}
  \State \(Y_{k+1}\gets2X_k-Y_k\)
  \State \(\widehat x_+\gets
         \Call{FirstOrderProx}{X_k,Y_{k+1},Q_{\rm prox}}\)
  \If{\textsc{FirstOrderProx} reports failure}
    \State \Return \(X_k\)
  \EndIf
  \State \(\widetilde X\gets
         \Call{ApproximateRGO}{Y_{k+1},\widehat x_+,a}\)
  \State Draw \(B_k\sim\operatorname{Bernoulli}(\rho)\) independently
  \If{\(B_k=1\)}
    \State \(X_{k+1}\gets\widetilde X\)
  \Else
    \State Query \(\nabla V(\widetilde X)\)
    \State \(X_{k+1}\gets
           \Call{CappedHalfTurn}
           {X_k,Y_{k+1},\widetilde X,\bar\lambda}\)
           using \cref{alg:capped-half}
  \EndIf
\EndFor
\State \Return \(X_T\)
\end{algorithmic}
\end{algorithm}

\paragraph{Choice of parameters.}
Given \(\Delta\ge1\), \(\varepsilon\in(0,1/4)\),
\(0<\eta\le1/\beta\), and a tuning constant \(K\ge1\), define
\begin{equation}
 \ell\deq K\,\Bigl[\Delta+
 \log \Bigl(\frac{Kd\kappa}
 {\varepsilon\alpha\eta}\Bigr)\Bigr]\,,
 \qquad
 \bar\lambda_{\rm impl}\deq K\beta\eta\,(\sqrt{d\ell}+\ell)\,.
 \label{eq:implementation-log-level}
\end{equation}
Set \(\rho\) as in \eqref{eq:rho-choice}, let \(N\) be the smallest
integer satisfying \eqref{eq:N-main-bound}, and define
\begin{equation}
 \begin{aligned}
 T&\deq \Bigl\lceil K\,\Bigl(\Delta+
       \log\frac4\varepsilon\Bigr)\Bigr\rceil\,N\,,
 \qquad
 a\deq\frac{\varepsilon}{6T}\,,
 \qquad Q_{\rm prox}\deq\bigl\lceil K\log(\kappa\ell)\bigr\rceil\,.
 \end{aligned}
 \label{eq:implementation-parameters}
\end{equation}

\begin{theorem}[Sampler from an order-2 R\'enyi warm start]
\label{thm:main-general}
There exist universal constants \(c\in(0,1)\) and \(K\ge1\) such
that the following holds.  Assume \eqref{eq:SC-smooth}.  Let
\(\Delta\ge1\), \(\varepsilon\in(0,1/4)\), and let \(\mu_0\) be a
probability measure on \(\R^d\) satisfying
\(D_2(\mu_0\|\mu)\le\Delta\), and choose \(0<\eta\le1/\beta\).
Assume that
\begin{equation}
 \beta\eta\,\bigl(\sqrt{d\ell}+\ell\bigr)\le c\,.
 \label{eq:main-scale-condition}
\end{equation}
Run \cref{alg:implementable-main} from \(X_0\sim\mu_0\) using
\eqref{eq:implementation-log-level}--\eqref{eq:implementation-parameters}.  If
\(\widehat X\) is its output and \(N_{\nabla V}\) is its total number
of gradient queries, then
\begin{equation}
 \TV\bigl(\mathcal L(\widehat X),\mu\bigr)\le\varepsilon\,.
 \label{eq:main-TV}
\end{equation}
Moreover, its expected query count satisfies
\begin{equation}
 \E N_{\nabla V}
 \le
 \frac{K}{\sqrt{\alpha\eta}}
 \sqrt{\log\frac{e}{\alpha\eta}}\,
 \Bigl(\Delta+\log\frac4\varepsilon\Bigr)\,
 \bigl[Q_{\rm prox}+\beta\eta\,\bigl(
 \sqrt{d\ell}+\ell\bigr)\bigr]\,.
 \label{eq:main-grad-general}
\end{equation}
\end{theorem}

\begin{proof}
Set \(\nu_0\deq\mathcal L(X_0,Y_0)\).  By
\eqref{eq:warm-lift-intro} and \eqref{eq:renyi-chi},
\[
 D_2(\nu_0\|\pi_\eta)\le\Delta\,,
 \qquad
 \chi^2(\nu_0\|\pi_\eta)\le e^\Delta\,.
\]
Since \(\ell\ge1\), \eqref{eq:main-scale-condition} gives
\(\beta\eta\sqrt d\le c\).  Taking \(c\le c_0\) and using
\eqref{eq:block-TV},
\[
 \TV\bigl((\nu_0 \mathcal K^T)^X,\mu\bigr)
 \le\frac12\exp \Bigl\{
 \frac12\,\Bigl[\Delta-
 \Bigl\lceil K\,\Bigl(\Delta+
 \log\frac4\varepsilon\Bigr)\Bigr\rceil\Bigr]\Bigr\}
 \le\frac\varepsilon2\,,
\]
after increasing \(K\).  Minimality of \(N\), together with
\(\alpha\eta\le1\), gives
\begin{equation}
 \begin{aligned}
 N&\le 1+\frac{C}{\sqrt{\alpha\eta}}
       \sqrt{\log\frac{e}{\alpha\eta}}
 \le\frac{C+1}{\sqrt{\alpha\eta}}
       \sqrt{\log\frac{e}{\alpha\eta}}\,,\\
 T&\le\frac{C}{\sqrt{\alpha\eta}}
       \sqrt{\log\frac{e}{\alpha\eta}}\,
      \Bigl(\Delta+\log\frac4\varepsilon\Bigr)\,.
 \end{aligned}
 \label{eq:T-expanded}
\end{equation}

Let \(\nu_k\deq\nu_0\mathcal K^k\), and let \(\nu_k^+\) be its image
under the reflection at the start of transition \(k\).  Since
\(\mathcal K\) preserves \(\pi_\eta\) and the reflection is a
\(\pi_\eta\)-preserving bijection, data processing gives
\[
 D_2(\nu_k^+\|\pi_\eta)
 =D_2(\nu_k\|\pi_\eta)
 \le D_2(\nu_0\|\pi_\eta)\le\Delta\,,
 \qquad 0\le k<T\,.
\]
Since \(a^{-1}=6T/\varepsilon\), \eqref{eq:T-expanded} and
\eqref{eq:clock-cap} yield, for a sufficiently large universal \(K\),
\begin{equation}
 0<a<\frac12\,,
 \qquad
 \log\frac1a\vee\ell_{a,\Delta}
 \le C\,\Bigl[\Delta+
       \log \Bigl(\frac{Kd\kappa}
       {\varepsilon\alpha\eta}\Bigr)\Bigr]
 \le\ell\,.
 \label{eq:implementation-log-dominance}
\end{equation}
Taking
\(c\le\min\{c_0,c_{\rm evt},K_{\rm RGO}^{-1}\}\) and
\(K\ge K_{\rm evt}\), we have
\[
 \begin{split}
 &K_{\rm RGO}\beta\eta\,
 \bigl\{\sqrt{d\log(1/a)}+\log(1/a)\bigr\}
 \le K_{\rm RGO}c\le1\,,\\
 &\bar\lambda_{\rm impl}
 =K\beta\eta\,(\sqrt{d\ell}+\ell)
 \ge K_{\rm evt}\beta\eta\,
       (\sqrt{d\ell_{a,\Delta}}+\ell_{a,\Delta})
 =\bar\lambda_{a,\Delta}\,.
 \end{split}
\]
Thus, \cref{thm:rgo-guarantee} applies with \(\delta=a\), while
\cref{prop:capped-rate} and monotonicity in the cap apply along every
ideal transition.  If
\(p_k^{\rm RGO}\) and \(p_k^{\rm cap}\) denote the mismatch
probabilities in the corresponding couplings, then
\[
 p_k^{\rm RGO}\le a\,,
 \qquad
 p_k^{\rm cap}\le a\,,
 \qquad
 \E N_{\nabla V}^{\rm RGO}\le K_{\rm RGO}\,,
 \qquad 0\le k<T\,.
\]

For the proximal solver, let \(x_\star\) minimize \(V\).  Under
\(\pi_\eta\), write
\(Y=X+\sqrt\eta Z\) with \(X\sim\mu\),
\(Z\sim\cN(0,I)\), and \(X\perp Z\), and set
\begin{equation}
 R(X,Y)\deq
 \frac{\|(2X-Y)-\eta\nabla V(X)-X\|}{\sqrt{d\eta}}
 =\frac{\|\sqrt\eta Z+\eta\nabla V(X)\|}{\sqrt{d\eta}}\,.
 \label{eq:prox-initial-residual}
\end{equation}
Integration by parts gives
\(
 \alpha\,\E\|X-x_\star\|^2
 \le\E\langle\nabla V(X),X-x_\star\rangle=d
\).
Using \(\|\nabla V(X)\|\le\beta\,\|X-x_\star\|\),
\(\beta\eta\le1\), \eqref{eq:concentration-tool}, and the Gaussian
norm tail, we obtain for \(u\ge1\)
\begin{equation}
 \begin{aligned}
 &\Prob_{\pi_\eta} \bigl\{
 R(X,Y)>2\sqrt\kappa\,\bigl(1+\sqrt{2u/d}\bigr)\bigr\}\\
 &\quad\le
 \Prob \bigl\{\|X-x_\star\|>
       \sqrt{d/\alpha}+\sqrt{2u/\alpha}\bigr\}
 +\Prob \bigl\{\|Z\|>\sqrt d+\sqrt{2u}\bigr\}
 \le2e^{-u}\,.
 \end{aligned}
 \label{eq:prox-initial-residual-tail}
\end{equation}
Set \(u_0\deq\Delta+2\log(1/a)+\log2\).  After increasing \(K\),
\eqref{eq:implementation-log-dominance} and the definition of
\(Q_{\rm prox}\) give
\[
 u_0\le C\ell\,,
 \qquad
 Q_{\rm prox}\ge K_{\rm prox}\,
 \bigl[1+\log \bigl\{1+2\sqrt\kappa\,
      \bigl(1+\sqrt{2u_0/d}\bigr)\bigr\}\bigr]\,.
\]
If \(\mathcal F\) is the solver failure event, then
\eqref{eq:first-order-prox-cost} and
\eqref{eq:prox-initial-residual-tail} give, under stationarity,
\[
 \pi_\eta(\mathcal F)
 \le2e^{-u_0}=e^{-\Delta}a^2\,.
\]
Cauchy--Schwarz then yields
\[
 \nu_k^+(\mathcal F)
 \le e^{D_2(\nu_k^+\|\pi_\eta)/2}\,
       \pi_\eta(\mathcal F)^{1/2}
 \le a\,.
\]

Construct the preceding couplings along the ideal trajectory, using the
same Bernoulli branch variables in both chains.  If \(X_T^{\rm id}\)
is the ideal output, a union bound gives
\begin{equation}
 \Prob\{\widehat X\ne X_T^{\rm id}\}
 \le\sum_{k=0}^{T-1}
 \bigl\{\nu_k^+(\mathcal F)
       +p_k^{\rm RGO}+p_k^{\rm cap}\bigr\}
 \le3Ta=\frac\varepsilon2\,.
 \label{eq:implementation-coupling-error}
\end{equation}
Consequently,
\[
 \TV\bigl(\mathcal L(\widehat X),\mu\bigr)
 \le\Prob\{\widehat X\ne X_T^{\rm id}\}
    +\TV\bigl(\mathcal L(X_T^{\rm id}),\mu\bigr)
 \le\varepsilon\,,
\]
which proves \eqref{eq:main-TV}.

For the query count, early termination can only help, and a half-turn
call uses one reference gradient query and an expected
\(\pi\bar\lambda_{\rm impl}\) queries at candidate times.  Therefore
\begin{equation}
 \E N_{\nabla V}
 \le T\,\bigl[Q_{\rm prox}+K_{\rm RGO}
 +(1-\rho)\,(1+\pi\bar\lambda_{\rm impl})\bigr]
 \le KT\,\bigl[Q_{\rm prox}
 +\beta\eta\,(\sqrt{d\ell}+\ell)\bigr]\,.
 \label{eq:implementation-grad-cost}
\end{equation}
The last inequality uses \(Q_{\rm prox}\ge1\).  Combining it with
\eqref{eq:T-expanded} proves
\eqref{eq:main-grad-general}.
\end{proof}

\begin{corollary}[Optimized gradient complexity]
\label{cor:tuned}
There exist universal constants \(c_{\rm opt}\in(0,1)\) and
\(K_{\rm opt}\ge1\) such that the following holds.  Suppose that
\eqref{eq:SC-smooth} holds, fix \(\Delta\ge1\) and
\(\varepsilon\in(0,1/4)\), and let \(\mu_0\) be a probability measure
on \(\R^d\) satisfying \(D_2(\mu_0\|\mu)\le\Delta\).  Set
\begin{equation}
 L\deq\Delta+\log\frac{K_{\rm opt}d\kappa}{\varepsilon}\,,
 \qquad
 \eta\deq\frac{c_{\rm opt}}{\beta\,(\sqrt{dL}+L)}\,,
 \label{eq:tuned-eta}
\end{equation}
and run \cref{alg:implementable-main} from \(X_0\sim\mu_0\) using
\eqref{eq:implementation-log-level}--\eqref{eq:implementation-parameters}.  Its output \(\widehat X\) satisfies
\[
 \TV\bigl(\mathcal L(\widehat X),\mu\bigr)\le\varepsilon\,.
\]
Moreover, if \(N_{\nabla V}\) is the total number of gradient queries,
then
\begin{equation}
 \E N_{\nabla V}
 \le
 K_{\rm opt}\sqrt\kappa\,(d+L)^{1/4}\,L^{3/4}\,
 \Bigl(\Delta+\log\frac4\varepsilon\Bigr)
 \log(\kappa L)\,.
 \label{eq:tuned-complexity}
\end{equation}
In particular, when \(\Delta=O(1)\), the expected gradient
complexity is \(\widetilde O(\sqrt\kappa\,d^{1/4})\), where
\(\widetilde O\) suppresses logarithmic factors in
\(d\), \(\kappa\), and \(1/\varepsilon\).
\end{corollary}

\begin{proof}
We first verify that the tuned step size satisfies the hypotheses of
\cref{thm:main-general}, and then simplify its query bound.  Since
\(d\ge1\), \(\kappa\ge1\), and \(L\ge1\), the choice in
\eqref{eq:tuned-eta} satisfies \(0<\eta\le1/\beta\).  Moreover,
\[
 \alpha\eta
 =\frac{c_{\rm opt}}{\kappa\,(\sqrt{dL}+L)}\,.
\]
The definitions of \(L\) and \(\ell\) therefore imply, with
\(K_{\rm opt}\) large enough, that
\begin{equation}
 \log\frac{e}{\alpha\eta}\le K_{\rm opt}L\,,
 \qquad
 \ell\le K_{\rm opt}L\,,
 \qquad
 Q_{\rm prox}\le K_{\rm opt}\log(\kappa L)\,.
 \label{eq:tuned-log-check}
\end{equation}
The bound on \(\ell\) and the formula for \(\eta\) give
\begin{equation}
 \beta\eta\,(\sqrt{d\ell}+\ell)
 \le\frac{K_{\rm opt}c_{\rm opt}}{\sqrt{dL}+L}\,
       (\sqrt{dL}+L)
 =K_{\rm opt}c_{\rm opt}\,.
 \label{eq:tuned-scale-check}
\end{equation}
Choosing \(c_{\rm opt}\) small enough makes the last expression no larger
than the constant \(c\) in \cref{thm:main-general}.  This verifies
\eqref{eq:main-scale-condition}, so the total variation conclusion
follows from \cref{thm:main-general}.

It remains to simplify the query bound.  The choice of \(\eta\) gives
\begin{equation}
 \frac1{\sqrt{\alpha\eta}}
 =\sqrt{\frac\kappa{c_{\rm opt}}}\,
   (\sqrt{dL}+L)^{1/2}
 \le C\sqrt\kappa\,(dL+L^2)^{1/4}\,.
 \label{eq:tuned-sqrt-kappa-scale}
\end{equation}
Combining this estimate with the first inequality in
\eqref{eq:tuned-log-check} yields
\[
 \frac1{\sqrt{\alpha\eta}}
 \sqrt{\log\frac e{\alpha\eta}}
 \le C\sqrt\kappa\,(dL+L^2)^{1/4}\sqrt L\,.
\]
The remaining bracket in \eqref{eq:main-grad-general} is, by
\eqref{eq:tuned-log-check} and \eqref{eq:tuned-scale-check}, bounded by
a universal multiple of \(\log(\kappa L)\).  Substitution into
\eqref{eq:main-grad-general}, with
\(K_{\rm opt}\) large enough, proves \eqref{eq:tuned-complexity}.
\end{proof}

\clearpage
\appendix

\section{Conditional half-turn process: construction, invariance, and implementation}
\label{app:halfturn-process}
We first construct the half-turn process up to an arbitrary time $t$ at fixed \((y,\widetilde x)\) in \cref{app:fixed-reference-process}.
We use an integrated hazard construction~\cite[Section~2]{davis-pdmp}
and prove reversibility by a path reversal argument. We then average over the initial momentum and
reference point in \cref{app:averaged-halfturn} to obtain the ideal
half-turn kernel and its augmented-space lift.  Finally,
\cref{app:bounce-capped} treats the bounce count at stationarity, the
tail behaviour of the event rates, and the implementation via a capped rate.

\subsection{The fixed-reference phase space process}
\label{app:fixed-reference-process}

\phantomsection
\label{app:fiber-construction}

\begin{proof}[Proof of the well-posedness assertion in
\cref{prop:fiber-construction}]
Fix \(y,\widetilde x\in\R^d\).  In the notation of
\eqref{eq:original-coordinate-equilibrium-residual}, the bounce rate
\eqref{eq:x-bounce-rate} reads
\begin{equation*}
 \lambda_{\widetilde x}(x,p)
 =\sqrt\eta\,[p^\top h_{\widetilde x}(x)]_+\,.
\end{equation*}
The deterministic harmonic flow, namely the solution
\eqref{eq:explicit-harmonic-flow} of \eqref{eq:x-harmonic-flow}, and
the bounce map are
\begin{align*}
 \Phi_t^{y,\widetilde x}(x,p)
 &\deq
 \bigl(c_{y,\widetilde x}+(x-c_{y,\widetilde x})\cos t
       +\sqrt\eta\,p\sin t,\;
       -\eta^{-1/2}(x-c_{y,\widetilde x})\sin t
       +p\cos t\bigr)\,,\\
 \opS_{\widetilde x}(x,p)
 &\deq(x,R_{h_{\widetilde x}(x)}p)\,,
 \qquad R_0\deq I\,.
\end{align*}
The bounce map is an involution.  Although there is a possible discontinuity at \(h_{\widetilde x}(x)=0\), this creates no ambiguity
because the event rate vanishes there.

Let \((E_n)_{n\ge1}\) be independent \(\operatorname{Exp}(1)\)
random variables.  Starting from
\(\zeta_0\deq (x_0,p_0)\) at time \(T_0 \deq 0\), define recursively
\begin{align}
 S_{n+1}
 &\deq
 \inf\Bigl\{u\ge0:
 \int_0^u\lambda_{\widetilde x}
 (\Phi_s^{y,\widetilde x}(\zeta_{T_n}))\,\dd s\ge E_{n+1}
 \Bigr\}\,,
 \label{eq:jump-time-construction}\\
 T_{n+1}&\deq T_n+S_{n+1}\,,\qquad
 \zeta_{T_{n+1}}
 \deq\opS_{\widetilde x}\bigl(
 \Phi_{S_{n+1}}^{y,\widetilde x}(\zeta_{T_n})\bigr)\,,
 \label{eq:post-jump-construction}
\end{align}
when \(S_{n+1}<\infty\).  If the set is empty, set
\(T_{n+1}=S_{n+1}=\infty\) and stop the recursion.  Between jumps, set
\(\zeta_{T_n+t}=\Phi_t^{y,\widetilde x}(\zeta_{T_n})\) for
\(0\le t<S_{n+1}\).

To prove non-explosion, define the shifted harmonic energy
\begin{equation*}
 H_{y,\widetilde x}(x,p)
 \deq\frac12\,\bigl\{\eta^{-1}\,\|x-c_{y,\widetilde x}\|^2
 +\|p\|^2\bigr\}\,.
\end{equation*}
The flow and each bounce preserve this energy.  Hence, with
\(\mathcal E\deq H_{y,\widetilde x}(x_0,p_0)\), every state produced by
the recursion satisfies
\begin{equation*}
 \|p\|\le\sqrt{2\mathcal E}\,,
 \qquad
 \|x-c_{y,\widetilde x}\|\le\sqrt{2\eta\mathcal E}\,.
\end{equation*}
Because \(\nabla V\) is \(\beta\)-Lipschitz,
\(\lambda_{\widetilde x}(x,p)
\le\bar\lambda_{\mathcal E}(y,\widetilde x)\) throughout the path, where
\begin{equation*}
 \bar\lambda_{\mathcal E}(y,\widetilde x)
 \deq
 \sqrt\eta\,\beta\sqrt{2\mathcal E}\,
 \bigl\{\sqrt{2\eta\mathcal E}
       +\|c_{y,\widetilde x}-\widetilde x\|\bigr\}\,.
\end{equation*}
Hence,
\[
 \int_0^u\lambda_{\widetilde x}
 (\Phi_s^{y,\widetilde x}(\zeta_{T_n}))\,\dd s
 \le \bar\lambda_{\mathcal E}(y,\widetilde x)\,u\,.
\]
If \(\bar\lambda_{\mathcal E}(y,\widetilde x)=0\), the event rate
vanishes and there are no jumps.  Otherwise,
\eqref{eq:jump-time-construction} implies, for every finite waiting time,
\[
 S_{n+1}\ge
 \frac{E_{n+1}}{\bar\lambda_{\mathcal E}(y,\widetilde x)}\,.
\]
If the recursion does not terminate, it follows that
\[
 T_n\ge
 \frac{\sum_{j=1}^nE_j}
      {\bar\lambda_{\mathcal E}(y,\widetilde x)}
 \longrightarrow\infty
\]
almost surely by the strong law of large numbers.  Thus, the jump times
cannot accumulate, and the recursive construction defines a unique
process for all \(t\ge0\).  Finally, the
memoryless property of the exponential clocks gives the time-homogeneous
Markov property.  This is precisely the integrated hazard construction
of Davis~\cite[Section~2]{davis-pdmp}.
\end{proof}

\phantomsection
\label{app:fiber-invariance}

With the process now constructed, we prove the stationarity assertion in
\cref{prop:fiber-construction} and record the stronger operator facts used
to analyze the averaged kernel.  Recall \(\nu_{\eta,y}\) from
\eqref{eq:conditional-position-momentum-measure}, and define momentum
reversal by
\begin{equation*}
 \opR(x,p)\deq(x,-p)\,,
\qquad
 (\opR F)(x,p)\deq F(\opR(x,p))\,.
\end{equation*}
As usual, we use the same symbol for \(\nu_{\eta,y}\) and its Lebesgue
density:
\begin{equation*}
 \nu_{\eta,y}(x,p)
 \propto
 \exp\Bigl\{-V_y(x)-\frac{\|p\|^2}{2}\Bigr\}\,.
\end{equation*}

\begin{proposition}[Path reversal, invariance, and the \(L^2\) semigroup]
\label{prop:fiber-semigroup-rigorous}
Let \((\mathsf T_t^{y,\widetilde x})_{t\ge0}\) be the Markov transition
operators of the process constructed in
\cref{prop:fiber-construction}.  For every \(t\ge0\),
\begin{equation}
 (\mathsf T_t^{y,\widetilde x})^*
 =\opR\mathsf T_t^{y,\widetilde x}\opR
 \quad\text{on }L^2(\nu_{\eta,y})\,.
 \label{eq:skew-detailed-balance-semigroup}
\end{equation}
Consequently:
\begin{enumerate}[label=\textup{(\roman*)},leftmargin=2.2em]
\item \(\nu_{\eta,y}\) is invariant;
\item \((\mathsf T_t^{y,\widetilde x})_{t\ge0}\) is a strongly continuous
contraction semigroup on \(L^2(\nu_{\eta,y})\);
\item if \(\opPi_p:L^2(\nu_{\eta,y})\to
L^2(\pi_\eta^{X\mid Y=y})\) denotes Gaussian momentum averaging, then
\(\opPi_p\mathsf T_t^{y,\widetilde x}\opPi_p^*\) is a
self-adjoint Markov contraction on \(L^2(\pi_\eta^{X\mid Y=y})\).
\end{enumerate}
\end{proposition}

\paragraph{Idea of the proof.}
The identity \eqref{eq:skew-detailed-balance-semigroup} is a
momentum-reversed form of detailed balance.  To see the mechanism
formally, write
\(\opL_{y,\widetilde x}=\mathcal A+\mathcal J\), where
\begin{align*}
 \mathcal A f(x,p)
 &\deq
 \sqrt\eta\,p^\top\nabla_x f(x,p)
 -\frac{(x-c_{y,\widetilde x})^\top}{\sqrt\eta}
  \nabla_p f(x,p),\\
 \mathcal J f(x,p)
 &\deq
 \lambda_{\widetilde x}(x,p)
 \bigl\{f(\opS_{\widetilde x}(x,p))-f(x,p)\bigr\}.
\end{align*}
Momentum reversal reverses the harmonic flow, so
\(\opR\mathcal A\opR=-\mathcal A\).  Moreover,
\[
 \mathcal A\log\nu_{\eta,y}(x,p)
 =-\sqrt\eta\,p^\top h_{\widetilde x}(x)
 =-\lambda_{\widetilde x}(x,p)+\lambda_{\widetilde x}(x,-p).
\]
Formal integration by parts gives
\[
 \mathcal A^*=-\mathcal A+\lambda_{\widetilde x}
  -\lambda_{\widetilde x}\circ\opR.
\]
The bounce map \(\opS_{\widetilde x}\) is a
\(\nu_{\eta,y}\)-preserving involution, commutes with \(\opR\), and
satisfies \(\lambda_{\widetilde x}\circ\opS_{\widetilde x}
=\lambda_{\widetilde x}\circ\opR\).  A formal change of variables
therefore gives
\[
 \mathcal J^*G
 =(\lambda_{\widetilde x}\circ\opR)
    (G\circ\opS_{\widetilde x})
  -\lambda_{\widetilde x}G.
\]
Combining the two identities yields
\[
 \opL_{y,\widetilde x}^*G
 =-\mathcal AG
  +(\lambda_{\widetilde x}\circ\opR)
      (G\circ\opS_{\widetilde x}-G)
 =(\opR\opL_{y,\widetilde x}\opR)G.
\]
Formally exponentiating this identity gives
\((\mathsf T_t^{y,\widetilde x})^*
 =\opR\mathsf T_t^{y,\widetilde x}\opR\); Gaussian momentum averaging
then removes \(\opR\) and produces an ordinary self-adjoint position
kernel.

This generator calculation is only heuristic as stated, as the standard arguments are not
immediately applicable because
\(\opS_{\widetilde x}\) may be discontinuous where
\(h_{\widetilde x}(x)=0\).  The pathwise argument below proves the same
identity without requiring differentiability or continuity of the
bounce map.

\begin{proof}
Along
\(\zeta_s=(x_s,p_s)=\Phi_s^{y,\widetilde x}(\zeta_0)\),
\begin{equation}
 \frac{\dd}{\dd s}\log\nu_{\eta,y}(\zeta_s)
 =-\sqrt\eta\,p_s^\top h_{\widetilde x}(x_s)
 =-\bigl\{\lambda_{\widetilde x}(\zeta_s)
      -\lambda_{\widetilde x}(\opR\zeta_s)\bigr\}\,.
 \label{eq:density-rate-balance}
\end{equation}
Moreover, by definition of $\opR$ and $\Phi_t$,
\begin{equation*}
 \opR\Phi_t^{y,\widetilde x}\opR
 =\Phi_{-t}^{y,\widetilde x}\,,
 \qquad
 \opR\opS_{\widetilde x}
 =\opS_{\widetilde x}\opR\,.
\end{equation*}
If \(\zeta^+=\opS_{\widetilde x}\zeta^-\) is a non-trivial jump, then
\(h_{\widetilde x}(x)\ne0\),
\( (p^-)^\top  h_{\widetilde x}(x) >0\), and
\begin{equation}
 \lambda_{\widetilde x}(\zeta^-)
 =\lambda_{\widetilde x}(\opR\zeta^+)\,.
 \label{eq:jump-rate-reversal}
\end{equation}
The flow and \(\opR\) preserve Lebesgue measure, as does the possibly
discontinuous bounce map: by Fubini's theorem and orthogonality, for
every integrable \(F\),
\begin{equation}
 \int F(\opS_{\widetilde x}(x,p))\,\dd x\,\dd p
 =\int F(x,p)\,\dd x\,\dd p\,.
 \label{eq:jump-map-volume}
\end{equation}

To pass from these identities to complete trajectories, let
\[
 \Delta_n(t)\deq\{(t_1,\ldots,t_n):0<t_1<\cdots<t_n<t\}\,.
\]
Given \(\zeta_0\) and \(\mathbf t\in\Delta_n(t)\), let
\(\zeta_s^{\zeta_0,\mathbf t}\) be the path obtained by following
\(\Phi^{y,\widetilde x}\) and applying \(\opS_{\widetilde x}\) at the
prescribed times.  Iterating the exponential clock construction gives,
for every bounded measurable functional \(\Psi\) of the path up to time
\(t\),
\begin{align}
 \E_{\zeta_0}\Psi(\zeta_{[0,t]})
 &=
 \sum_{n=0}^{\infty}
 \int_{\Delta_n(t)}
 \Psi(\zeta_{[0,t]}^{\zeta_0,\mathbf t})
 \exp\Bigl\{-\!\int_0^t
 \lambda_{\widetilde x}
 (\zeta_s^{\zeta_0,\mathbf t})\,\dd s\Bigr\}
 \prod_{j=1}^n
 \lambda_{\widetilde x}
 (\zeta_{t_j-}^{\zeta_0,\mathbf t})
 \,\dd\mathbf t\,.
 \label{eq:complete-path-density}
\end{align}
The series is absolutely convergent because
\(\lvert\Delta_n(t)\rvert=t^n/n!\) and the rate
is bounded by \(\bar\lambda_{\mathcal E}(y,\widetilde x)\).
Integrate \eqref{eq:complete-path-density} against
\(\nu_{\eta,y}(\dd \zeta_0)\) and apply
\begin{equation}
 (\zeta_s)_{0\le s\le t}
 \longmapsto(\opR\zeta_{t-s})_{0\le s\le t}\,,
 \qquad
 (t_1,\ldots,t_n)\longmapsto(t-t_n,\ldots,t-t_1)\,.
 \label{eq:path-reversal-map}
\end{equation}
For completeness, we justify this change of variables despite the
possible discontinuity of the bounce map.  For fixed
\(\mathbf t\in\Delta_n(t)\), the endpoint map
\(\zeta_0\mapsto\zeta_t^{\zeta_0,\mathbf t}\) is a finite composition
of the Borel involution \(\opS_{\widetilde x}\) and the affine
harmonic flows \(\Phi_u^{y,\widetilde x}\).  By
\eqref{eq:jump-map-volume} and the explicit flow formula, each factor is
a measure-preserving Borel bijection.  Hence the endpoint map, and
therefore
\(\zeta_0\mapsto\opR\zeta_t^{\zeta_0,\mathbf t}\), is a Borel
measure-preserving bijection.  The identities preceding
\eqref{eq:complete-path-density} show that its inverse is obtained by
using the reversed ordered times in \eqref{eq:path-reversal-map}.  The
map on \(\Delta_n(t)\) is linear with absolute determinant one.
Thus the full change of variables is measure preserving; no Jacobian is needed.

Integrating \eqref{eq:density-rate-balance} over the deterministic arcs
and using that a reflection preserves \(\nu_{\eta,y}\) gives
\begin{equation*}
 \nu_{\eta,y}(\zeta_0)\,
 e^{-\int_0^t\lambda_{\widetilde x}(\zeta_s)\,\dd s}
 =
 \nu_{\eta,y}(\zeta_t)\,
 e^{-\int_0^t\lambda_{\widetilde x}(\opR\zeta_s)\,\dd s}\,.
\end{equation*}
This identity matches the initial density and survival factor of a
forward path with those of its reversal, while
\eqref{eq:jump-rate-reversal} matches every jump-rate factor.  Hence,
each \(n\)-jump term in \eqref{eq:complete-path-density} is unchanged by
\eqref{eq:path-reversal-map}.  Tonelli's theorem justifies the termwise
change of variables for non-negative bounded functions; decomposing
signed functions into positive and negative parts then gives, for
bounded measurable \(F,G\),
\begin{equation*}
 \int F\,\mathsf T_t^{y,\widetilde x}G\,\dd\nu_{\eta,y}
 =
 \int (\opR G)\,
       \mathsf T_t^{y,\widetilde x}(\opR F)\,\dd\nu_{\eta,y}\,.
\end{equation*}
Taking \(F=1\) proves invariance, and the identity extends by density to
\eqref{eq:skew-detailed-balance-semigroup} on \(L^2(\nu_{\eta,y})\).

Invariance and Jensen's inequality give
\(\|\mathsf T_t^{y,\widetilde x}F\|_2\le\|F\|_2\).  It remains to prove
strong continuity.  For \(F\in C_{\rm c}(\R^{2d})\), the deterministic flow
converges to the identity as $t\searrow 0$, and the probability of a jump before time
\(t\) is
\[
 1-\exp\Bigl\{-\int_0^t
 \lambda_{\widetilde x}
 (\Phi_s^{y,\widetilde x}\zeta)\,\dd s\Bigr\}\longrightarrow0\,.
\]
Dominated convergence therefore gives
\(\mathsf T_t^{y,\widetilde x}F\to F\) in \(L^2\).  Density of
\(C_{\rm c}(\R^{2d})\) and contractivity extend this to all \(F\in L^2\).

Finally, Gaussian momentum averaging satisfies
\(\opR\opPi_p^*=\opPi_p^*\) and
\(\opPi_p\opR=\opPi_p\).  Combining these identities with
\eqref{eq:skew-detailed-balance-semigroup} gives
\[
 (\opPi_p\mathsf T_t^{y,\widetilde x}\opPi_p^*)^*
 =\opPi_p(\mathsf T_t^{y,\widetilde x})^*\opPi_p^*
 =\opPi_p\opR\mathsf T_t^{y,\widetilde x}\opR\opPi_p^*
 =\opPi_p\mathsf T_t^{y,\widetilde x}\opPi_p^*\,.
\]
This proves self-adjointness; the Markov and contraction properties are
inherited from \(\mathsf T_t^{y,\widetilde x}\).
\end{proof}

\subsection{The averaged ideal half-turn kernel}
\label{app:averaged-halfturn}

\phantomsection
\label{app:proof-half-invariance}
\begin{proof}[Proof of \cref{prop:half-invariance}]
Fix \(y\).  For each reference point \(\widetilde x\),
\cref{prop:fiber-semigroup-rigorous} shows that the fixed-reference
position operator
\(\opPi_p\mathsf T_\pi^{y,\widetilde x}\opPi_p^*\) is a self-adjoint
Markov contraction on \(L^2(\pi_\eta^{X\mid Y=y})\).  The reference
point in \cref{alg:ideal-half} is drawn independently from
\(\pi_\eta^{X\mid Y=y}\).  Consequently, the conditional half-turn
operator is the average
\[
 (\opH_y f)(x)
 =
 \int_{\R^d}
 \bigl(\opPi_p\mathsf T_\pi^{y,\widetilde x}\opPi_p^*f\bigr)(x)\,
 \pi_\eta^{X\mid Y=y}(\dd\widetilde x)\,.
\]
The averaging measure does not depend on \(x\), so \(\opH_y\) is again a
self-adjoint Markov contraction.  In kernel form, this is the
detailed-balance identity
\begin{equation*}
 \pi_\eta^{X\mid Y=y}(\dd x)\,\mathcal H_y(x,\dd x')
 =\pi_\eta^{X\mid Y=y}(\dd x')\,\mathcal H_y(x',\dd x)\,,
\end{equation*}
which proves the proposition.
\end{proof}

The preceding reversibility statement holds for each fixed \(y\).  To use
these kernels as a single operator on the augmented space, we next verify
their joint measurability and then integrate the conditional operator
identities over \(y\).

\begin{proposition}[Joint measurability and augmented space half-turn operator]
\label{prop:joint-halfturn-operator}
The family of conditional half-turn kernels
\((\mathcal H_y)_{y\in\R^d}\) admits a version, again denoted by
\(\mathcal H_y(x,\dd x')\), such that
\((y,x)\mapsto\mathcal H_y(x,A)\) is Borel measurable for every
Borel set \(A\subset\R^d\).  Consequently,
\begin{equation}
 (\opH f)(x,y)
 \deq\int_{\R^d}f(x',y)\,\mathcal H_y(x,\dd x')
 \label{eq:joint-halfturn-operator}
\end{equation}
defines a self-adjoint Markov contraction on \(L^2(\pi_\eta)\).
\end{proposition}
\begin{proof}
Since \(\nabla V\) is continuous, the harmonic flow, event rate, and
reflection map are jointly Borel; the convention \(R_0=I\) ensures
this even where the reflection normal vanishes.  The integrated hazard
is jointly Borel, with continuous, nondecreasing dependence on time, so its
threshold-crossing times are Borel, as can be checked using rational
times.  Thus the recursion
\eqref{eq:jump-time-construction}--\eqref{eq:post-jump-construction}
makes each jump time and post-jump state Borel in
\((y,x_0,p_0,\widetilde x,(E_j)_{j\ge1})\).
By non-explosion, the terminal state at time \(\pi\) is the almost-sure
limit of the Borel terminal states obtained by allowing at most \(n\)
jumps.  Setting the limit to zero wherever it does not exist gives a
Borel version.

The conditional density in \eqref{eq:conditional-prox-law} is jointly
Borel in \((y,\widetilde x)\).  Integrating \(\1_{\{x_\pi\in A\}}\)
against this reference law and the laws of the Gaussian momentum and
exponential clocks therefore makes \((y,x)\mapsto\mathcal H_y(x,A)\)
Borel for every Borel set \(A\).

Finally, each \(\opH_y\) is a self-adjoint Markov contraction by
\cref{prop:half-invariance}.  Integrating the conditional contraction
bound gives
\begin{align*}
 \|\opH f\|_{L^2(\pi_\eta)}^2
 &\le\int\bigl\|f(\cdot,y)\bigr\|_{L^2(\pi_\eta^{X\mid Y=y})}^2\,
       \pi_\eta^Y(\dd y)
 =\|f\|_{L^2(\pi_\eta)}^2\,.
\end{align*}
Conditional self-adjointness and Fubini likewise give
\(\langle f,\opH g\rangle=\langle\opH f,g\rangle\) for bounded
measurable \(f,g\), and hence for all \(L^2(\pi_\eta)\) by density.
Positivity and \(\opH1=1\) follow from the corresponding properties of
each \(\opH_y\).
\end{proof}

\subsection{Bounce complexity and capped-rate implementation}
\label{app:bounce-capped}

\phantomsection
\label{app:proof-ideal-bounce-count}

\cref{prop:ideal-bounce-count} follows immediately from the following lemma. More specifically, for every fixed
\(y,\widetilde x\), \cref{prop:fiber-semigroup-rigorous} shows that the
phase space process preserves \(\nu_{\eta,y}\).  Hence, under the
initialization in the proposition and conditional on \(Y=y\), the triple
\((x_t,p_t,\widetilde X)\) has law
\(\pi_\eta^{X\mid Y=y}\otimes\cN(0,I)\otimes\pi_\eta^{X\mid Y=y}\) at
every time \(t\in[0,\pi]\).

\begin{lemma}[Tail bound for the event rate]
\label{lem:event-tail}
Let \(X,\widetilde X\stackrel{\rm iid}{\sim}\pi_\eta^{X\mid Y=y}\) and
\(P\sim\cN(0,I)\) be independent, and set $\lambda\deq\sqrt\eta
 \,[P^\top(\nabla V(X)-\nabla V(\widetilde X))]_+$.
For every \(u\ge0\),
\begin{equation}
 \Prob\bigl\{\lambda>C\beta\eta\,(\sqrt{du}+u)\bigr\}
 \le Ce^{-u}\,.
 \label{eq:event-tail}
\end{equation}
Consequently, for \(\ell\ge0\) and
\(\bar\lambda_\ell=C\beta\eta\,(\sqrt{d\ell}+\ell)\),
\begin{equation}
 \E(\lambda-\bar\lambda_\ell)_+
 \le C\beta\eta\sqrt d\,e^{-c\ell}\,.
 \label{eq:event-overflow}
\end{equation}
\end{lemma}

\begin{proof}
The measure \(\pi_\eta^{X\mid Y=y}\) is
\((\alpha+\eta^{-1})\)-strongly log-concave, and therefore
\eqref{eq:concentration-tool} gives
\[
 \Prob\bigl\{\|X-\E X\|\ge C\sqrt{\eta\,(d+u)}\bigr\}
 \le e^{-u}
\]
for every \(u\ge0\).
Therefore,
\begin{equation}
 \Prob\bigl\{
 \sqrt\eta\,\|\nabla V(X)-\nabla V(\widetilde X)\|
 >2C\beta\eta\sqrt{d+u}
 \bigr\}\le 2e^{-u}\,.
 \label{eq:event-normal-tail}
\end{equation}
Conditional on
\(h=\sqrt\eta\,\{\nabla V(X)-\nabla V(\widetilde X)\}\), the scalar
\(P^\top h\sim \cN(0,\|h\|^2)\).  The Gaussian tail and a union bound
with \eqref{eq:event-normal-tail} therefore show that, for $u\ge 1$ and outside an event
of probability at most \(2e^{-u}\),
\[
    [P^\top h]_+
 \le C\sqrt u\,\|h\|
 \le C\beta\eta\,(\sqrt{du}+u)\,,
\]
which proves \eqref{eq:event-tail} in this range.
The bound for \(0\le u<1\) follows after increasing \(C\) accordingly.
Finally, \eqref{eq:event-overflow} follows by integration.
\end{proof}

\phantomsection
\label{app:proof-capped-rate}
\begin{proof}[Proof of \cref{prop:capped-rate}]
For part~\textup{(i)}, first suppose that \((X_0,Y)\sim\pi_\eta\), and write
\(\Prob_{\pi_\eta}\) and \(\E_{\pi_\eta}\) for probability and
expectation under the resulting coupled construction, including the
reference point, initial momentum, and Poisson random measure shared by
the ideal and capped processes.  Let \(\mathcal D\) be the event that
the two paths separate.  The paths agree until the first event from
the excess rate \((\lambda_t-\bar\lambda_{\xi,\Delta})_+\), so the
probability of \(\mathcal D\) is bounded by the expected number of
such events:
\[
 \Prob_{\pi_\eta}(\mathcal D)
 \le
 \E_{\pi_\eta}\int_0^\pi
 (\lambda_t-\bar\lambda_{\xi,\Delta})_+\,\dd t\,.
\]
By stationarity, \eqref{eq:event-overflow} applies at every \(t\).  Since the
small-scale assumption implies
\(\beta\eta\sqrt d\le c_{\rm evt}\), we obtain
\[
 \Prob_{\pi_\eta}(\mathcal D)
 \le Ce^{-c\ell_{\xi,\Delta}}
 \le e^{-\Delta}\xi^2\,.
\]
The definition of \(\ell_{\xi,\Delta}\) in \eqref{eq:clock-cap} gives
the last inequality after increasing \(K_{\rm evt}\).

Now let \(\nu_0\) satisfy
\(D_2(\nu_0\|\pi_\eta)\le\Delta\), and set
\(h_0=\dd\nu_0/\dd\pi_\eta\).  Relative to the coupled path law
\(\Prob_{\pi_\eta}\), the construction started from \(\nu_0\) has
likelihood ratio \(h_0(X_0,Y)\): all reference points, momenta, and Poisson
variables are generated by the same conditional rules.  Hence,
Cauchy--Schwarz gives
\[
 \Prob_{\nu_0}(\mathcal D)
 =\E_{\pi_\eta}[h_0(X_0,Y)\1_{\mathcal D}]
 \le(\E_{\pi_\eta}h_0^2)^{1/2}\,
      \Prob_{\pi_\eta}(\mathcal D)^{1/2}
 \le e^{\Delta/2}\,(e^{-\Delta}\xi^2)^{1/2}=\xi\,.
\]
This proves part~\textup{(i)}.

For part~\textup{(ii)}, the number of arrivals of the
rate-\(\bar\lambda_{\xi,\Delta}\) proposal clock before time \(\pi\)
is Poisson with mean \(\pi\bar\lambda_{\xi,\Delta}\).  Since the
reference gradient is already cached, the capped simulation makes
exactly one gradient query at each arrival.  This proves
\eqref{eq:half-query-cost}.
\end{proof}

\section{Operator formulation and modified \texorpdfstring{\(L^2\)}{L2} contraction}
\label{sec:micro-macro}

This appendix proves \cref{thm:ideal-convergence} using the modified \(L^2\)
hypocoercivity framework of Dolbeault, Mouhot, and Schmeiser~\cite{dms} and Fan, Li, and Lu~\cite{fan-li-lu}.  We first set
up the operator framework and then derive the contraction from two
model-specific estimates proved in Appendix~\ref{app:key-estimates}.

\subsection{Operator framework and micro--macro decomposition}
\label{app:operator-framework}

Our target distribution is the joint law \(\pi_\eta\) defined in
\eqref{eq:augmentation}, and we set \(\cH\deq L^2(\pi_\eta)\).  Recall from
\eqref{eq:ideal-proximal-bps-kernel} that the ideal Proximal BPS has
transition kernel
\begin{equation*}
 \mathcal K\bigl((x,y),\dd x'\,\dd y'\bigr)
 =\delta_{2x-y}(\dd y')\,
 \bigl[
  \rho\,\pi_\eta^{X\mid Y=y'}(\dd x')
  +(1-\rho)\,\mathcal H_{y'}(x,\dd x')
 \bigr]\,,
\end{equation*}
where \(\mathcal H_{y'}\) is the conditional half-turn kernel from
\cref{prop:half-invariance}.  Since the algorithm is a discrete-time
chain, its evolution is described by the induced one-step Markov operator
on observables,
\begin{equation*}
 (\opK f)(x,y)
 \deq\int_{\R^{2d}}f(x',y')\,
   \mathcal K\bigl((x,y),\dd x'\,\dd y'\bigr)\,.
\end{equation*}
Thus \(\mathcal K\) acts on probability laws, whereas \(\opK\) is its
induced action on observables.  The first term in brackets
in \(\mathcal K\) is the conditional resampling branch.  After the
auxiliary coordinate is updated to \(y'\),
this branch replaces the \(x\)-coordinate by an independent draw from
\(\pi_\eta^{X\mid Y=y'}\).  The associated conditional expectation operator
is
\begin{equation}
 (\opP f)(x,y)
   \deq\int f(x',y)\,\pi_\eta^{X\mid Y=y}(\dd x')
     =\E[f(X,Y)\mid Y=y]\,.
 \label{eq:macroscopic-projection}
\end{equation}
In particular, \(\opP f\) depends only on the coordinate \(y\).  The range
and kernel of \(\opP\) are
\begin{equation}
 \begin{split}
 \cH_{\mathsf P}
   &\deq\ran(\opP)
     =\{f\in\cH:f(x,y)=g(y)
          \text{ for some }g\in L^2(\pi_\eta^Y)\}\,,\\
 \cH_\perp
   &\deq\ker(\opP)=\ran(\opP^\perp)
     =\{f\in\cH:\E[f(X,Y)\mid Y]=0\}\,.
 \end{split}
 \label{eq:micro-macro-spaces}
\end{equation}
Here \(\opP^\perp\deq I-\opP\), so
\(\cH=\cH_{\mathsf P}\oplus\cH_\perp\) and
\(f=\opP f+\opP^\perp f\).  In the common terminology of
hypocoercivity~\cite{dms,lu-wang}, \(\ran(\opP)\) is the
\emph{macroscopic} subspace and its orthogonal complement is the
\emph{microscopic} subspace.
Thus functions in \(\cH_{\mathsf P}\) depend only on the coordinate \(y\),
whereas functions in \(\cH_\perp\) have zero mean under
\(\pi_\eta^{X\mid Y=y}\) for \(\pi_\eta^Y\)-almost every \(y\).

Relative to the decomposition
\(\cH=\cH_{\mathsf P}\oplus\cH_\perp\), we write an operator
\(\mathsf T\) on \(\cH\) in block form as
\begin{equation}
 \mathsf T=
 \begin{pmatrix}
  \mathsf T_{\mathsf P\mathsf P} & \mathsf T_{\mathsf P\perp}\\
  \mathsf T_{\perp\mathsf P} & \mathsf T_{\perp\perp}
 \end{pmatrix}\,,
 \qquad
 \begin{aligned}
  \mathsf T_{\mathsf P\mathsf P}&\deq\opP\mathsf T\opP\,,&
  \mathsf T_{\mathsf P\perp}&\deq\opP\mathsf T\opP^\perp\,,\\
  \mathsf T_{\perp\mathsf P}&\deq\opP^\perp\mathsf T\opP\,,&
  \mathsf T_{\perp\perp}&\deq\opP^\perp\mathsf T\opP^\perp\,.
 \end{aligned}
 \label{eq:projection-block-convention}
\end{equation}
The first subscript specifies the target subspace and the second the source
subspace; thus \(\mathsf T_{ab}:\cH_b\to\cH_a\) for
\(a,b\in\{\mathsf P,\perp\}\).

We first apply this block notation to the reflection map \(\opU\) from
\eqref{eq:easy-reflection}.  We use the same symbol for its pullback on
observables:
\begin{equation}
 (\opU f)(x,y)\deq f(\opU(x,y))=f(x,2x-y)\,.
 \label{eq:reflection-operator}
\end{equation}
This pullback is self-adjoint and unitary, and hence has block matrix
\[
 \opU=
 \begin{pmatrix}
  \opUPP & \opUperpP^*\\
  \opUperpP & \opUperpperp
 \end{pmatrix}\,.
\]
Here \(\opUPP=\opP\opU\opP\) maps \(\cH_{\mathsf P}\) to itself, whereas
\(\opUperpP=\opP^\perp\opU\opP\) maps \(\cH_{\mathsf P}\) to \(\cH_\perp\).
Since
\(\opU\) is self-adjoint and unitary, \(\opUPP\) is a self-adjoint
contraction on \(\cH_{\mathsf P}\).  Writing \(\opU^2=I\) in blocks gives
\begin{equation}
 \opUperpP^*\opUperpP=I-\opUPP^2\,,
 \qquad
 \opUperpP^*\opUperpperp=-\opUPP\opUperpP^*\,.
 \label{eq:involution-identities}
\end{equation}

We next express the joint half-turn operator \(\opH\) in the same block
decomposition.  It is defined in \eqref{eq:joint-halfturn-operator} and is a
self-adjoint Markov contraction by
\cref{prop:joint-halfturn-operator}.  Since \(\opH\) acts only on the
\(x\)-coordinate, while every
function in \(\ran(\opP)\) depends only on \(y\), we have
\(\opH\opP=\opP\).  Self-adjointness also gives
\(\opP\opH=\opP\).  Set
\(\opHperpperp\deq\opP^\perp\opH\opP^\perp:\cH_\perp\to\cH_\perp\).
The identity \(\opP\opH=\opP\) makes \(\cH_\perp\) invariant under
\(\opH\), so \(\opHperpperp\) is its restriction to \(\cH_\perp\).  It is
therefore self-adjoint, and
\(\|\opHperpperp g\|=\|\opH g\|\le\|g\|\) for
\(g\in\cH_\perp\).  Consequently,
\begin{equation}
 \opH=\opP+\opHperpperp\,,
 \qquad
 -I\preceq\opHperpperp\preceq I
 \quad\text{on }\cH_\perp\,.
 \label{eq:halfturn-block-decomposition}
\end{equation}
Relative to \(\cH=\cH_{\mathsf P}\oplus\cH_\perp\), this decomposition
has the block representation
\[
 \opH=
 \begin{pmatrix}
  I_{\cH_{\mathsf P}} & 0\\
  0 & \opHperpperp
 \end{pmatrix}\,.
\]
The Markov operator \(\opK\) induced by the transition kernel \(\mathcal K\)
therefore factors as
\begin{equation}
 \opK
  =\opU\bigl[\rho\opP+(1-\rho)\opH\bigr]
  =\opU\bigl[\opP+(1-\rho)\opHperpperp\bigr]\,.
 \label{eq:ideal-markov-operator}
\end{equation}
Accordingly, its block representation is
\[
 \opK=
 \begin{pmatrix}
  \opUPP & \opUperpP^*\\
  \opUperpP & \opUperpperp
 \end{pmatrix}
 \begin{pmatrix}
  I_{\cH_{\mathsf P}} & 0\\
  0 & (1-\rho)\opHperpperp
 \end{pmatrix}\,.
\]
The diagonal factor in the block representation of \(\opK\) contracts
\(\cH_\perp\) but leaves \(\cH_{\mathsf P}\) unchanged.  For an observable
\(f\in\cH_{\mathsf P}\), contraction therefore relies on the
\(\cH_\perp\)-component \(\opUperpP f\) created by reflection.  The next
subsection estimates this component and the subsequent action of
\(\opHperpperp\).

\subsection{Estimates for the reflection and half-turn operators}
\label{app:reflection-halfturn-estimates}

We begin with the reflection operator \(\opU\).  Let
\(f\in\cH_{\mathsf P}\), so that
\(f(x,y)=f(y)\).  Although \(f\) initially depends only on the auxiliary
coordinate \(y\), reflection gives $(\opU f)(x,y)=f(2x-y)$, which generally depends on both coordinates.  Its macroscopic and microscopic
components are \(\opUPP f\) and \(\opUperpP f\), respectively.  The following
probabilistic representation identifies the squared norm of the microscopic
component created by reflection with an expected conditional variance, thereby
enabling quantitative estimate by the conditional Poincar\'e estimates used in
Appendix~\ref{app:key-estimates}.

Concretely, under \(\pi_\eta\), write
\begin{equation}
 Y_+\deq X+\sqrt\eta\, Z\,,
 \qquad
 Y_-\deq X-\sqrt\eta\, Z\,,
 \qquad
 Z\sim\cN(0,I)\,,\qquad X\sim\mu\,.
 \label{eq:Yplusminus}
\end{equation}
The pair \((Y_+,Y_-)\) is exchangeable.  We henceforth identify
\(\cH_{\mathsf P}\) with \(L^2(\pi_\eta^Y)\) and write \(f(y)\) for a
macroscopic observable.  Under this identification,
\begin{equation}
 (\opUPP f)(y)=\E[f(Y_-)\mid Y_+=y]\,,
 \qquad
 \|\opUperpP f\|^2=\E\Var(f(Y_-)\mid Y_+)\,.
 \label{eq:UPP-exchangeable}
\end{equation}

To quantify the microscopic component \(\opUperpP f\) created when reflection
acts on \(f\in\cH_{\mathsf P}\), define
\begin{equation}
 \opGammaP\deq(I-\opUPP^2)^{1/2}
 \quad\text{on }\cH_{\mathsf P}\,.
 \label{eq:GammaP-definition}
\end{equation}
By \eqref{eq:involution-identities},
\(\opUperpP^*\opUperpP=I-\opUPP^2\), and hence
\begin{equation}
 \opUperpP^*\opUperpP=\opGammaP^2\,,
 \qquad
 \|\opUperpP f\|=\|\opGammaP f\|\,,
 \quad f\in\cH_{\mathsf P}\,.
 \label{eq:reflection-Gamma-identity}
\end{equation}
Thus a lower bound for \(\opGammaP\) quantifies how much of a macroscopic
observable reflection transfers into the microscopic subspace.

\begin{lemma}[Uniform macroscopic coercivity]
\label{lem:macroscopic-coercivity}
Assume \(\beta\eta\le1\), and set
\begin{equation}
 \rho_{\rm mac}\deq\frac{1-\alpha\eta}{1+\alpha\eta}\,,
 \qquad
 \gamma_{\rm gap}\deq\sqrt{1-\rho_{\rm mac}^2}
 =\frac{2\sqrt{\alpha\eta}}{1+\alpha\eta}\,.
 \label{eq:macroscopic-gap-constants}
\end{equation}
Note that \(\gamma_{\rm gap}\geq \frac{1}{2}\sqrt{\alpha\eta}\).
Then the following hold.

(1) For every \(f\in L^2(\pi_\eta^Y)\),
\(\opUPP f\in H^1(\pi_\eta^Y)\) and
\begin{equation}
 4\eta\,\|\nabla\opUPP f\|^2
 \le\|f\|^2-\|\opUPP f\|^2
 =
 \|\opGammaP f\|^2=\|\opUperpP f\|^2\,.
 \label{eq:UPP-sharp-smoothing}
\end{equation}

(2) Consider the mean-zero macroscopic subspace $\cH_{\mathsf P,0}
 \deq\cH_{\mathsf P}\cap L_0^2(\pi_\eta)$.
Then, it holds that
\begin{equation}
    \|\opUPP\| \le \rho_{\rm mac} \quad\text{and}\quad \opUPP \succeq -\frac{1}{2}\,I
 \qquad\text{on }\cH_{\mathsf P,0}\,.
 \label{eq:UPP-two-sided-bound}
\end{equation}
Consequently,
\begin{equation}
 \opGammaP
 =(I-\opUPP^2)^{1/2}
 \succeq
 \gamma_{\rm gap} I
 \qquad\text{on }\cH_{\mathsf P,0}\,.
 \label{eq:macroscopic-coercivity}
\end{equation}

\end{lemma}

We prove \cref{lem:macroscopic-coercivity} in
Appendix~\ref{app:proof-coercivity}. The identity
\eqref{eq:UPP-exchangeable} expresses \(\|\opUperpP f\|\) in terms of the
conditional variance of \(f(Y_-)\) given \(Y_+\).  Differentiating the
corresponding conditional expectation shows that this variance controls how
quickly \(\opUPP f\) changes with \(y\).  The Poincar\'e inequality for
\(\pi_\eta^Y\) then bounds the norm of a centered \(f\) in terms of
\(\|\opUperpP f\|\), giving the stated estimate.

Since \(\opGammaP\) is boundedly invertible on \(\cH_{\mathsf P,0}\),
\eqref{eq:reflection-Gamma-identity} gives the polar decomposition
\begin{equation}
 \opUperpP=\opVperpP\opGammaP\,,
 \qquad
 \opVperpP\deq\opUperpP\opGammaP^{-1}\,:
 \cH_{\mathsf P,0}\longrightarrow\cH_\perp\,,
 \qquad
 \opVperpP^*\opVperpP=I_{\cH_{\mathsf P,0}}\,.
 \label{eq:reflection-polar-decomposition}
\end{equation}
Here \(\opGammaP\) measures the size of the microscopic component created by
reflection, while the isometry \(\opVperpP\) records its direction in
\(\cH_\perp\).
The next proposition says that the half-turn acts nearly as \(-I\) on the
corresponding microscopic directions \(\opVperpP\opGammaP^{-1}f\).

\begin{proposition}[Global half-turn bound]
\label{prop:global-halfturn-cutoff}
There exist universal constants \(c,C>0\) such that, if
\(\beta\eta\le c\), then
\begin{equation}
 \bigl\|(I+\opHperpperp)\opVperpP v\bigr\|^2
 \le C\log(1/\gamma_{\rm gap})\,\|\opGammaP v\|^2\,,
 \qquad v\in\cH_{\mathsf P,0}\,.
 \label{eq:global-halfturn-cutoff}
\end{equation}
\end{proposition}

We prove \cref{prop:global-halfturn-cutoff} in
\cref{sec:affine-halfturn,app:global-halfturn-cutoff}.

Together, \cref{lem:macroscopic-coercivity,prop:global-halfturn-cutoff}
provide the estimates used in the modified \(L^2\) argument below.

\subsection{Modified \texorpdfstring{\(L^2\)}{L2} contraction and proof of the main theorem}
\label{sec:block}

We now prove \cref{thm:ideal-convergence} using a discrete modified  \(L^2\)
argument.  For \(f\in L_0^2(\pi_\eta)\), write
\[
 f_{\mathsf P}\deq\opP f\,,
 \qquad f_\perp\deq\opP^\perp f\,.
\]
Here \(f_{\mathsf P}\) and \(f_\perp\) are, respectively, the macroscopic
and microscopic parts of \(f\).

Fix \(0<\rho\le1/2\).  The ordinary \(L^2\) energy does not by itself give
the required one-step contraction.  Indeed, since \(\opU\) is unitary and
\(\opHperpperp\) is a contraction,
\begin{align}
 \|\opK f\|^2-\|f\|^2
 &=(1-\rho)^2\,\|\opHperpperp f_\perp\|^2-\|f_\perp\|^2\notag\\
 &\le-\rho\,(2-\rho)\,\|f_\perp\|^2
 \le-\rho\,\|f_\perp\|^2\,.
 \label{eq:unmodified-L2-dissipation}
\end{align}
This estimate controls \(f_\perp\) but gives nothing when
\(f=f_{\mathsf P}\).
Reflection moves part of \(f_{\mathsf P}\) into the
microscopic subspace, but this transfer is not captured by the \(L^2\) norm,
which is a standard difficulty in hypocoercivity.  The modified \(L^2\)
method addresses this difficulty by adding a corrector to the Lyapunov
function.  The corrector records the transfer from \(f_{\mathsf P}\) to
\(f_\perp\), allowing the contraction in
\eqref{eq:unmodified-L2-dissipation} to also control \(f_{\mathsf P}\).
This idea goes back to Dolbeault, Mouhot, and Schmeiser~\cite{dms}; here we
follow the refinement of Fan, Li, and
Lu~\cite{fan-li-lu}.

To construct the corrector, we first isolate the part of \(f_\perp\) that
couples to \(f_{\mathsf P}\) under reflection.  Define
\[
 f_{\mathsf V}\deq\opVperpP^*f_\perp\,.
\]
By \eqref{eq:reflection-polar-decomposition},
\(\opUperpP^*f_\perp=\opGammaP f_{\mathsf V}\), while the component of
\(f_\perp\) orthogonal to \(\ran(\opVperpP)\) is annihilated by
\(\opUperpP^*\).  Since \(\opVperpP\) is an isometry,
\(\opVperpP f_{\mathsf V}=\opVperpP\opVperpP^*f_\perp\) is the orthogonal
projection of \(f_\perp\) onto \(\ran(\opVperpP)\), and hence
\(\|f_{\mathsf V}\|\le\|f_\perp\|\).  We work with \(f_{\mathsf V}\), rather
than with this projection itself, because \(f_{\mathsf V}\) and
\(f_{\mathsf P}\) both belong to \(\cH_{\mathsf P,0}\) and can therefore be
paired in the corrector.

With this notation, let \(\omega>0\) be a weight to be chosen below and
define the modified Lyapunov function
\begin{equation}
 \mathscr L_\omega(f)
 \deq\|f\|^2
   +\omega\mathscr C(f_{\mathsf P},f_{\mathsf V})\,,
 \label{eq:modified-L2-general}
\end{equation}
where the corrector is
\begin{equation}
 \mathscr C(f_{\mathsf P},f_{\mathsf V})
 \deq\frac12\,\bigl(\|f_{\mathsf P}\|^2-\|f_{\mathsf V}\|^2\bigr)
   -\bigl\langle\opUPP\opGammaP^{-1}f_{\mathsf P},f_{\mathsf V}\bigr\rangle\,.
 \label{eq:discrete-corrector}
\end{equation}
Here \(\opGammaP^{-1}\) is applied only on \(\cH_{\mathsf P,0}\), where
\(\opGammaP\succeq\gamma_{\rm gap}I\), so the corrector is well-defined.

To motivate this choice, recall from \eqref{eq:ideal-markov-operator} that
\(\opK\) acts in two stages: it first maps
\[
 f_{\mathsf P}+f_\perp
 \longmapsto
 f_{\mathsf P}+(1-\rho)\opHperpperp f_\perp
\]
and then applies \(\opU\).  For the moment, set \(\rho=0\) and replace
\(\opHperpperp\) by \(-I_{\cH_\perp}\).  The error
\(I+\opHperpperp\) introduced by this replacement is controlled by
\cref{prop:global-halfturn-cutoff}.  Denote the resulting idealized operator
by
\[
 \opK_{\rm id}\deq\opU(\opP-\opP^\perp)\,.
\]
Set \(g\deq\opK_{\rm id}f\), and define the corresponding components by
\[
 g_{\mathsf P}\deq\opP g\,,
 \qquad
 g_{\mathsf V}\deq\opVperpP^*\opP^\perp g\,.
\]
In terms of these components, \(\opK_{\rm id}\) is represented by the
operator-valued rotation matrix
\[
 \begin{pmatrix}g_{\mathsf P}\\ g_{\mathsf V}\end{pmatrix}
 =
 \begin{pmatrix}
  \opUPP&-\opGammaP\\
  \opGammaP&\opUPP
 \end{pmatrix}
 \begin{pmatrix}f_{\mathsf P}\\ f_{\mathsf V}\end{pmatrix}\,.
\]
Indeed, the first component follows immediately from the block representation
of \(\opU\) and the polar decomposition
\eqref{eq:reflection-polar-decomposition}:
\begin{align*}
 g_{\mathsf P}
 &=\opP\opU(f_{\mathsf P}-f_\perp)
 =\opUPP f_{\mathsf P}-\opUperpP^*f_\perp
 =\opUPP f_{\mathsf P}-\opGammaP\opVperpP^*f_\perp
 =\opUPP f_{\mathsf P}-\opGammaP f_{\mathsf V}\,.
\end{align*}
For the second component, substitute
\(\opUperpP^*=\opGammaP\opVperpP^*\) into
\(\opUperpP^*\opUperpperp=-\opUPP\opUperpP^*\) from
\eqref{eq:involution-identities}.  Since \(\opUPP\) commutes with
\(\opGammaP\), this gives
\[
 \opGammaP\opVperpP^*\opUperpperp
 =-\opGammaP\opUPP\opVperpP^*\,.
\]
The ranges of \(\opVperpP^*\opUperpperp\) and
\(\opUPP\opVperpP^*\) lie in \(\cH_{\mathsf P,0}\), where
\(\opGammaP^{-1}\) is bounded.  Applying \(\opGammaP^{-1}\) to the
displayed identity therefore gives
\(\opVperpP^*\opUperpperp=-\opUPP\opVperpP^*\).  Hence
\begin{align*}
 g_{\mathsf V}
 &=\opVperpP^*\opP^\perp\opU(f_{\mathsf P}-f_\perp)
 =\opVperpP^*(\opUperpP f_{\mathsf P}
                  -\opUperpperp f_\perp)
 =\opGammaP f_{\mathsf P}+\opUPP\opVperpP^*f_\perp
 =\opGammaP f_{\mathsf P}+\opUPP f_{\mathsf V}\,.
\end{align*}
This proves the claimed matrix formula.  Since
\(\opUPP^2+\opGammaP^2=I\), the block matrix acts as a rotation on
\(\cH_{\mathsf P,0}\oplus\cH_{\mathsf P,0}\).  Under this rotation,
direct substitution into \eqref{eq:discrete-corrector} gives
\begin{align}
 &\mathscr C(g_{\mathsf P},g_{\mathsf V})
   -\mathscr C(f_{\mathsf P},f_{\mathsf V})\notag\\
 &\quad=\mathscr C(\opUPP f_{\mathsf P}-\opGammaP f_{\mathsf V}\,,
                    \opGammaP f_{\mathsf P}+\opUPP f_{\mathsf V})
          -\mathscr C(f_{\mathsf P},f_{\mathsf V})\notag\\
 &\quad=-\|f_{\mathsf P}\|^2+\|f_{\mathsf V}\|^2\,.
 \label{eq:corrector-cohomological-identity}
\end{align}
For \(0<\omega<\rho\), combining the leading terms from
\eqref{eq:unmodified-L2-dissipation} and
\eqref{eq:corrector-cohomological-identity}, and using
\(\|f_{\mathsf V}\|\le\|f_\perp\|\), gives
\begin{align*}
 &-\rho\,\|f_\perp\|^2
   +\omega\,\bigl(-\|f_{\mathsf P}\|^2+\|f_{\mathsf V}\|^2\bigr)
 \le-\omega\,\|f_{\mathsf P}\|^2
             -(\rho-\omega)\,\|f_\perp\|^2\,.
\end{align*}
Thus the leading terms control both components.  This calculation makes
explicit the role of the corrector in the modified \(L^2\) hypocoercivity
argument; see also Fan, Li, and Lu~\cite{fan-li-lu}.

We now turn from the idealized calculation to the contraction proof.  The
modified \(L^2\) argument has two steps: we first show that the Lyapunov
function is equivalent to the squared \(L^2\) norm and then establish its
one-step decay.

\begin{lemma}[Equivalence with the squared \(L^2\) norm]
\label{lem:modified-L2-equivalence}
If \(0<\omega\le\gamma_{\rm gap}\), then every
\(f\in L_0^2(\pi_\eta)\) satisfies
\begin{equation}
 \frac12\,\|f\|^2
 \le\mathscr L_\omega(f)
 \le\frac32\,\|f\|^2\,.
 \label{eq:modified-L2-equivalence}
\end{equation}
\end{lemma}

\begin{proof}
The operators \(\opUPP\) and \(\opGammaP\) commute, so
\(\opUPP\opGammaP^{-1}\) is self-adjoint on
\(\cH_{\mathsf P,0}\).  Moreover,
\begin{equation*}
 I+(\opUPP\opGammaP^{-1})^2
 =I+\opUPP^2\opGammaP^{-2}
 =\opGammaP^{-2}\,,
\end{equation*}
where the last identity uses \(\opUPP^2+\opGammaP^2=I\).
By \eqref{eq:discrete-corrector},
\begin{align*}
 2\mathscr C(f_{\mathsf P},f_{\mathsf V})
 &=\Bigl\langle
   \begin{pmatrix}f_{\mathsf P}\\f_{\mathsf V}\end{pmatrix}\,,
   \begin{pmatrix}
    I&-\opUPP\opGammaP^{-1}\\
    -\opUPP\opGammaP^{-1}&-I
   \end{pmatrix}
   \begin{pmatrix}f_{\mathsf P}\\f_{\mathsf V}\end{pmatrix}
  \Bigr\rangle\,.
\end{align*}
The square of the block operator in this display is
\[
 \begin{pmatrix}
  I&-\opUPP\opGammaP^{-1}\\
  -\opUPP\opGammaP^{-1}&-I
 \end{pmatrix}^{\!2}
 =\begin{pmatrix}\opGammaP^{-2}&0\\0&\opGammaP^{-2}\end{pmatrix}\,.
\]
Since the block operator is self-adjoint and
\(\opGammaP\succeq\gamma_{\rm gap}I\),
\[
 \Bigl\|
  \begin{pmatrix}
   I&-\opUPP\opGammaP^{-1}\\
   -\opUPP\opGammaP^{-1}&-I
  \end{pmatrix}
 \Bigr\|
 =\|\opGammaP^{-1}\|
 \le\frac1{\gamma_{\rm gap}}\,.
\]
Applying this operator norm bound to the preceding quadratic form gives
\begin{equation}
 |\mathscr C(f_{\mathsf P},f_{\mathsf V})|
 \le\frac1{2\gamma_{\rm gap}}\,
 \bigl(\|f_{\mathsf P}\|^2+\|f_{\mathsf V}\|^2\bigr)
 \le\frac1{2\gamma_{\rm gap}}\,\|f\|^2\,.
 \label{eq:corrector-norm}
\end{equation}
Here the second inequality uses $\|f_{\mathsf P}\|^2+\|f_{\mathsf V}\|^2
 \le\|f_{\mathsf P}\|^2+\|f_\perp\|^2
 =\|f\|^2$.
Finally, \eqref{eq:modified-L2-general},
\eqref{eq:corrector-norm}, and \(\omega\le\gamma_{\rm gap}\) imply
\begin{align}
 \bigl|\mathscr L_\omega(f)-\|f\|^2\bigr|
 &=\omega\,|\mathscr C(f_{\mathsf P},f_{\mathsf V})|
 \le\frac{\omega}{2\gamma_{\rm gap}}\,\|f\|^2
 \le\frac12\,\|f\|^2\,. \qedhere
\end{align}
\end{proof}

\begin{lemma}[One-step decay of the Lyapunov function]
\label{lem:modified-L2-one-step}
There exist universal constants \(c_{\rm hyp},c_0,c>0\) such that the
following holds.  Set
\begin{equation}
 \Lambda_\rho\deq\log(1/\gamma_{\rm gap})
  +\Bigl(\frac{\rho}{\gamma_{\rm gap}}\Bigr)^2\,,
 \qquad
 \omega\deq\frac{c_{\rm hyp}\rho}{\Lambda_\rho}\,.
 \label{eq:refreshment-scale}
\end{equation}
If \(\beta\eta\le c_0\), then every
\(f\in L_0^2(\pi_\eta)\) satisfies
\begin{equation}
 \mathscr L_\omega(\opK f)
 \le\Bigl(1-\frac{c\rho}{\Lambda_\rho}\Bigr)\,
      \mathscr L_\omega(f)\,.
 \label{eq:modified-L2-one-step}
\end{equation}
\end{lemma}

\begin{proof}
Take \(c_0\le1/16\).  Then
\(\gamma_{\rm gap}\le2\sqrt{\alpha\eta}\le1/2\), and consequently
\(\Lambda_\rho\ge\log 2\).
We compare the actual update with the idealized operator analyzed above.
Let \(g\deq\opK_{\rm id}f\),
\(g_{\mathsf P}\deq\opP g\), and
\(g_{\mathsf V}\deq\opVperpP^*\opP^\perp g\).
By the definitions of \(\opK\) and \(\opK_{\rm id}\),
\begin{align*}
 \opK f-\opK_{\rm id}f
 &=\opU\bigl[I+(1-\rho)\opHperpperp\bigr]f_\perp\\
 &=\opU\bigl[\rho f_\perp
       +(1-\rho)(I+\opHperpperp)f_\perp\bigr]\,.
\end{align*}
Using \(\opUperpP^*=\opGammaP\opVperpP^*\) and
\(\opVperpP^*\opUperpperp=-\opUPP\opVperpP^*\), the two components of
this difference are
\begin{align*}
 \opP(\opK f-\opK_{\rm id}f)
 &=\opGammaP\opVperpP^*
     \bigl[I+(1-\rho)\opHperpperp\bigr]f_\perp\,,\\
 \opVperpP^*\opP^\perp(\opK f-\opK_{\rm id}f)
 &=-\opUPP\opVperpP^*
     \bigl[I+(1-\rho)\opHperpperp\bigr]f_\perp\,.
\end{align*}
The same vector therefore determines both components of the error.  Denote
it by
\[
 r_\rho
 \deq\opVperpP^*\bigl[I+(1-\rho)\opHperpperp\bigr]f_\perp
 =\rho f_{\mathsf V}+(1-\rho)r\,,
 \qquad
 r\deq\opVperpP^*(I+\opHperpperp)f_\perp\,.
\]
Both \(r\) and \(r_\rho\) belong to \(\cH_{\mathsf P,0}\).  In
\(r_\rho\), the first
term comes from refreshment, while the second measures the departure of
the half turn from \(-I_{\cH_\perp}\).
Consequently,
\begin{equation}
 (\opK f)_{\mathsf P}=g_{\mathsf P}+\opGammaP r_\rho\,,
 \qquad
 (\opK f)_{\mathsf V}=g_{\mathsf V}-\opUPP r_\rho\,.
 \label{eq:actual-versus-ideal-components}
\end{equation}
Substituting \eqref{eq:actual-versus-ideal-components} into the definition
\eqref{eq:discrete-corrector} of the corrector \(\mathscr C\) and expanding
gives
\begin{align*}
 &\mathscr C\bigl((\opK f)_{\mathsf P},(\opK f)_{\mathsf V}\bigr)
   -\mathscr C(g_{\mathsf P},g_{\mathsf V})\\
 &\qquad =\bigl\langle g_{\mathsf P},\opGammaP r_\rho\bigr\rangle
   +\bigl\langle g_{\mathsf V},\opUPP r_\rho\bigr\rangle
   +\frac12\,\bigl(\|\opGammaP r_\rho\|^2
                    -\|\opUPP r_\rho\|^2\bigr)\\
 &\qquad \qquad
                    +{}\bigl\langle \opUPP\opGammaP^{-1}g_{\mathsf P},
      \opUPP r_\rho\bigr\rangle
   -\bigl\langle \opUPP r_\rho,g_{\mathsf V}\bigr\rangle
   +\|\opUPP r_\rho\|^2\\
 &\qquad =\bigl\langle g_{\mathsf P},\opGammaP r_\rho\bigr\rangle
   +\bigl\langle \opUPP\opGammaP^{-1}g_{\mathsf P},
      \opUPP r_\rho\bigr\rangle
   +\frac12\,\bigl(\|\opGammaP r_\rho\|^2
                    +\|\opUPP r_\rho\|^2\bigr)\,.
\end{align*}
In the last equality, the two terms containing \(g_{\mathsf V}\) cancel.
For the terms linear in \(g_{\mathsf P}\), the commutation of
\(\opUPP\) with \(\opGammaP\) and the identity
\(\opUPP^2+\opGammaP^2=I\) give
\begin{align*}
 &\bigl\langle g_{\mathsf P},\opGammaP r_\rho\bigr\rangle
 +\bigl\langle \opUPP\opGammaP^{-1}g_{\mathsf P},
    \opUPP r_\rho\bigr\rangle\\
 &\qquad=\bigl\langle g_{\mathsf P},
   (\opGammaP+\opGammaP^{-1}\opUPP^2)r_\rho\bigr\rangle
 =\bigl\langle g_{\mathsf P},\opGammaP^{-1}r_\rho\bigr\rangle\,.
\end{align*}
The same identity reduces the quadratic terms to
\[
 \frac12\,\bigl(\|\opGammaP r_\rho\|^2
                   +\|\opUPP r_\rho\|^2\bigr)
 =\frac12\,\|r_\rho\|^2\,.
\]
Consequently, combined with \eqref{eq:corrector-cohomological-identity},
\begin{align}
 &\mathscr C\bigl((\opK f)_{\mathsf P},(\opK f)_{\mathsf V}\bigr)
   -\mathscr C(f_{\mathsf P},f_{\mathsf V})\notag\\
 &\quad=\bigl[
   \mathscr C\bigl((\opK f)_{\mathsf P},(\opK f)_{\mathsf V}\bigr)
   -\mathscr C(g_{\mathsf P},g_{\mathsf V})\bigr]
   +\bigl[\mathscr C(g_{\mathsf P},g_{\mathsf V})
   -\mathscr C(f_{\mathsf P},f_{\mathsf V})\bigr]\notag\\
 &\quad=-\|f_{\mathsf P}\|^2+\|f_{\mathsf V}\|^2
   +\bigl\langle g_{\mathsf P},\opGammaP^{-1}r_\rho\bigr\rangle
   +\frac12\,\|r_\rho\|^2\,.
 \label{eq:corrector-perturbation-identity}
\end{align}
Thus, it remains to estimate \(r_\rho\).  Let
\[
 \mathsf L\deq(I+\opHperpperp)\opVperpP\,.
\]
Since \(r=\mathsf L^*f_\perp\), applying
\eqref{eq:global-halfturn-cutoff} to \(\opGammaP^{-1}v\) and then taking
adjoints gives
\begin{align}
 \|\opGammaP^{-1}r\|^2
 =\|\opGammaP^{-1}\opVperpP^*(I+\opHperpperp)f_\perp\|^2
 &\le C\log(1/\gamma_{\rm gap})\,\|f_\perp\|^2\,.
 \label{eq:global-corrector-remainder}
\end{align}
Together with \(\opGammaP\succeq\gamma_{\rm gap}I\) and
\(\|f_{\mathsf V}\|\le\|f_\perp\|\), this gives
\begin{align}
 \|\opGammaP^{-1}r_\rho\|^2
 &\le2\rho^2\,\|\opGammaP^{-1}f_{\mathsf V}\|^2
    +2(1-\rho)^2\,\|\opGammaP^{-1}r\|^2
 \le C\Lambda_\rho\,\|f_\perp\|^2\,.
 \label{eq:actual-corrector-remainder}
\end{align}
Moreover, since \(\opVperpP^*\) and \(\opHperpperp\) are contractions, $\|r_\rho\|
 \le\rho\,\|f_{\mathsf V}\|+(1-\rho)\,\|r\|
 \le2\,\|f_\perp\|$.

The rotation formula for \(\opK_{\rm id}\) gives $\|g_{\mathsf P}\|^2
 \le\|g_{\mathsf P}\|^2+\|g_{\mathsf V}\|^2
 =\|f_{\mathsf P}\|^2+\|f_{\mathsf V}\|^2$.
Hence, Young's inequality and
\eqref{eq:actual-corrector-remainder} give
\begin{align*}
 \bigl|\bigl\langle g_{\mathsf P},
    \opGammaP^{-1}r_\rho\bigr\rangle\bigr|
 &\le\frac14\,\|g_{\mathsf P}\|^2
       +C\,\|\opGammaP^{-1}r_\rho\|^2\\
 &\le\frac14\,\|f_{\mathsf P}\|^2
       +\frac14\,\|f_{\mathsf V}\|^2
       +C\Lambda_\rho\,\|f_\perp\|^2\,.
\end{align*}
Inserting these estimates into
\eqref{eq:corrector-perturbation-identity}, and using
\(\|f_{\mathsf V}\|\le\|f_\perp\|\) and
\(\Lambda_\rho\ge\log 2\), gives
\begin{equation}
 \begin{aligned}
 &\mathscr C\bigl((\opK f)_{\mathsf P},(\opK f)_{\mathsf V}\bigr)
   -\mathscr C(f_{\mathsf P},f_{\mathsf V})
 \le-\frac12\,\|f_{\mathsf P}\|^2
   +C\Lambda_\rho\,\|f_\perp\|^2\,.
 \end{aligned}
 \label{eq:corrector-one-step-bound}
\end{equation}

Finally, \eqref{eq:unmodified-L2-dissipation} and
\eqref{eq:corrector-one-step-bound} give
\begin{align*}
 \mathscr L_\omega(\opK f)-\mathscr L_\omega(f)
 &\le-\rho\,\|f_\perp\|^2
  +\frac{c_{\rm hyp}\rho}{\Lambda_\rho}\,
  \Bigl(-\frac12\,\|f_{\mathsf P}\|^2
         +C\Lambda_\rho\,\|f_\perp\|^2\Bigr)\\
 &=-\frac{c_{\rm hyp}\rho}{2\Lambda_\rho}\,
      \|f_{\mathsf P}\|^2
   -(1-Cc_{\rm hyp})\,\rho\,\|f_\perp\|^2\,.
\end{align*}
Choose \(c_{\rm hyp}\le1\) sufficiently small so that
\(Cc_{\rm hyp}\le1/2\).  Since \(\Lambda_\rho\ge\log 2\), the last display is
then bounded above by
\[
 -\frac{c\rho}{\Lambda_\rho}\,
 \bigl(\|f_{\mathsf P}\|^2+\|f_\perp\|^2\bigr)
 =-\frac{c\rho}{\Lambda_\rho}\,\|f\|^2\,.
\]
It remains to compare the last norm with the Lyapunov function.  The
elementary inequality \(x/(a+x^2)\le1/(2\sqrt a)\), for \(a>0\) and
\(x\ge0\), gives
\[
 \frac{\omega}{\gamma_{\rm gap}}
 =c_{\rm hyp}\,\frac{\rho/\gamma_{\rm gap}}
 {\log(1/\gamma_{\rm gap})+(\rho/\gamma_{\rm gap})^2}
 \le\frac{c_{\rm hyp}}
 {2\sqrt{\log(1/\gamma_{\rm gap})}}
 \le\frac{c_{\rm hyp}}{2\sqrt{\log 2}}
 \le1\,.
\]
Thus \(\omega\le\gamma_{\rm gap}\), and the upper bound in
\cref{lem:modified-L2-equivalence} gives
\[
 \|f\|^2\ge\frac23\,\mathscr L_\omega(f)\,.
\]
Consequently,
\[
 \mathscr L_\omega(\opK f)
 \le\Bigl(1-\frac{2c\rho}{3\Lambda_\rho}\Bigr)\,
       \mathscr L_\omega(f)\,,
\]
which is \eqref{eq:modified-L2-one-step} after renaming the universal
constant.
\end{proof}

Combining
\cref{lem:modified-L2-equivalence,lem:modified-L2-one-step}, we can now
prove the \(L^2\) contraction in \cref{thm:ideal-convergence}.

\begin{proof}[Proof of \cref{thm:ideal-convergence}]
We first turn the one-step Lyapunov decay into a fixed \(L^2\) contraction
over a block of \(N\) steps, and then apply this contraction to densities.

Take \(c_0\) from \cref{lem:modified-L2-one-step} and choose a fixed
universal \(\rho_\star\in(0,1/2]\).  The one-step estimate requires
\(\rho\le1/2\).  Since
\(\alpha\eta\le\beta\eta\le1\) and
\(x\log(e/x)\le1\) for \(0<x\le1\), the prescribed refreshment
probability satisfies
\[
 0<\rho
 =\rho_\star\sqrt{\alpha\eta\log\frac e{\alpha\eta}}
 \le\rho_\star\le\frac12\,.
\]

The decay rate in \cref{lem:modified-L2-one-step} is
\(\rho/\Lambda_\rho\), so a fixed contraction is obtained by taking
\(N\) proportional to \(\Lambda_\rho/\rho\).  We first bound this ratio.
Recall that
\(\gamma_{\rm gap}\asymp\sqrt{\alpha\eta}\) and
\(\log(1/\gamma_{\rm gap})\asymp\log(e/(\alpha\eta))\).  Hence
\[
 \Bigl(\frac{\rho}{\gamma_{\rm gap}}\Bigr)^2
 \le C\log(1/\gamma_{\rm gap})\,,
 \qquad
 \rho\ge c\gamma_{\rm gap}\sqrt{\log(1/\gamma_{\rm gap})}\,.
\]
It follows from the definition of \(\Lambda_\rho\) that
\begin{align*}
 \frac{\Lambda_\rho}{\rho}
 &\le C\,\frac{\sqrt{\log(1/\gamma_{\rm gap})}}{\gamma_{\rm gap}}
 \le\frac{C}{\sqrt{\alpha\eta}}
       \sqrt{\log\frac e{\alpha\eta}}\,.
\end{align*}
Choose a sufficiently large universal constant \(C_1\) and consider
\begin{equation}
 N\geq C_1\,\frac{\Lambda_\rho}{\rho}\,.
 \label{eq:block-length}
\end{equation}
For \(f\in L_0^2(\pi_\eta)\), take \(\omega\) as in
\eqref{eq:refreshment-scale}.  Iterating
\eqref{eq:modified-L2-one-step} over this block gives
\[
 \mathscr L_\omega(\opK^Nf)
 \le\Bigl(1-\frac{c\rho}{\Lambda_\rho}\Bigr)^N\,
      \mathscr L_\omega(f)\,.
\]
Applying the lower bound in \eqref{eq:modified-L2-equivalence} to
\(\opK^Nf\) and the upper bound to \(f\), we obtain
\begin{align*}
 \|\opK^Nf\|^2
 &\le2\mathscr L_\omega(\opK^Nf)\le2\,\Bigl(1-\frac{c\rho}{\Lambda_\rho}\Bigr)^N\,
       \mathscr L_\omega(f)\le3\,\Bigl(1-\frac{c\rho}{\Lambda_\rho}\Bigr)^N\,\|f\|^2\le3e^{-cC_1}\,\|f\|^2\,.
\end{align*}
The last step uses \(1-x\le e^{-x}\) and
\(N\ge C_1\Lambda_\rho/\rho\).
Taking \(C_1\) sufficiently large makes \(3e^{-cC_1}\le e^{-1}\).  Thus,
\begin{equation}
 \|\opK^Nf\|^2
 \le e^{-1}\,\|f\|^2\,.
 \label{eq:block-L2}
\end{equation}
Iterating by blocks, the \(L^2\) contraction
\eqref{eq:block-L2} is equivalently, by \eqref{eq:chi-L2} and the
adjoint remark following it, the \(\chi^2\)-estimate
\[
 \chi^2\bigl(\nu_0\mathcal K^{JN}\|\pi_\eta\bigr)
 \le e^{-J}\chi^2(\nu_0\|\pi_\eta)\,,
\]
which proves \eqref{eq:block-chi}.  The \(X\)-marginal of
\(\pi_\eta\) is \(\mu\), and total variation cannot increase under
projection.  Moreover, Cauchy--Schwarz gives
\(\TV(P,Q)\le\tfrac12\sqrt{\chi^2(P\|Q)}\), which gives \eqref{eq:block-TV}.
\end{proof}

\section{Macroscopic coercivity and the half-turn bound}
\label{app:key-estimates}

This appendix establishes the two model-specific estimates used by the
modified \(L^2\) method argument of Appendix~\ref{sec:micro-macro}: the
macroscopic coercivity lemma
(\cref{lem:macroscopic-coercivity}) and the global half-turn bound
(\cref{prop:global-halfturn-cutoff}).

\subsection{Proof of
  \texorpdfstring{\cref{lem:macroscopic-coercivity}}
  {Lemma~\ref*{lem:macroscopic-coercivity}}}
\label{app:proof-coercivity}

\begin{proof}[Proof of \cref{lem:macroscopic-coercivity}]
We first prove
\eqref{eq:UPP-sharp-smoothing}.  It suffices to consider
\(f\in C_{\rm c}^\infty(\R^d)\), the general case following by density and
closedness of the gradient.

Recall that the conditional
density \(p_y\) of \(Y_-\) given \(Y_+=y\) is
\[
 p_y(u)\propto
 \exp \Bigl\{
   -V((y+u)/2)-\frac{\|y-u\|^2}{8\eta}
 \Bigr\}\,,
\]
and define
\[
 s_y(u)\deq
 -\frac12\,\nabla V((y+u)/2)-\frac{y-u}{4\eta}\,.
\]
Differentiating the normalized conditional density gives
\begin{equation}
 \nabla\opUPP f(y)
 =
 \Cov \bigl(f(Y_-),s_y(Y_-)\bigm\vert Y_+=y\bigr)\,.
 \label{eq:UPP-sharp-derivative}
\end{equation}

The negative log-density of \(p_y\) has Hessian
\[
 \frac14\,\bigl\{
   \nabla^2V((y+u)/2)+\eta^{-1}I
 \bigr\}
 \succeq
 \frac{\alpha+\eta^{-1}}4\,I\,.
\]
Hence, the conditional Poincar\'e inequality implies, for every unit
vector \(a\),
\begin{align*}
 \Var \bigl(a^\top s_y(Y_-)\bigm\vert Y_+=y\bigr)
 &\le
 \frac4{\alpha+\eta^{-1}}\,
 \E \bigl[
   \|\nabla_u(a^\top s_y)(Y_-)\|^2
   \,\bigm|\,Y_+=y
 \bigr]\,.
\end{align*}
The assumptions
\(\alpha I\preceq\nabla^2V\preceq\beta I\) and
\(\beta\eta\le1\) give
\[
 \|\nabla_us_y(u)\|_{\rm op}
 \le \frac{\eta^{-1}-\alpha}{4}\,.
\]
Consequently,
\[
 \Var \bigl(a^\top s_y(Y_-)\bigm\vert Y_+=y\bigr)
 \le
 \frac{(\eta^{-1}-\alpha)^2}
      {4\,(\alpha+\eta^{-1})}\,.
\]
Applying conditional Cauchy--Schwarz in
\eqref{eq:UPP-sharp-derivative} and taking the supremum over
\(\|a\|=1\) yields
\[
 \|\nabla\opUPP f(y)\|^2
 \le
 \frac{(\eta^{-1}-\alpha)^2}
      {4\,(\alpha+\eta^{-1})}
 \Var(f(Y_-)\mid Y_+=y)\,.
\]
Integrating in \(y\) and using \eqref{eq:UPP-exchangeable},
\begin{align}
  \eta\,\|\nabla\opUPP f\|^2
 \le
 \frac{(1-\alpha\eta)^2}{4\,(1+\alpha\eta)}\,
 \|\opUperpP f\|^2\,.
 \label{eq:UPP-sharp-smoothing-full}
\end{align}
Finally,
\(\|\opUperpP f\|^2
 =\|\opGammaP f\|^2
 =\|f\|^2-\|\opUPP f\|^2\),
which proves \eqref{eq:UPP-sharp-smoothing}.

We next bound the spectrum of \(\opUPP\).  Let \(f\in\cH_{\mathsf P,0}\).
Since \(\opUPP\) preserves constants, \(\opUPP f\) is centered.
The Poincar\'e inequality for the marginal law \(\pi_\eta^Y\) gives
\begin{equation}
 \|\opUPP f\|^2=\Var_{\pi_\eta^Y}(\opUPP f)
 \le
 \Bigl(\frac1\alpha+\eta\Bigr)\,
 \|\nabla\opUPP f\|^2\,.
 \label{eq:Y-marginal-poincare-short}
\end{equation}
Combining \eqref{eq:Y-marginal-poincare-short} with
\eqref{eq:UPP-sharp-smoothing-full} and rearranging gives
\begin{equation}
 \|\opUPP f\|
 \le
 \frac{1-\alpha \eta}{1+\alpha \eta}\,\|f\|
 =\rho_{\rm mac}\,\|f\|\,.
 \label{eq:UPP-sharp-contraction-proof}
\end{equation}
Since \(\opUPP\) is self-adjoint, this proves
\[
 -\rho_{\rm mac}I\preceq\opUPP\preceq\rho_{\rm mac}I\,.
\]
It remains only to obtain the uniform lower bound
\(\opUPP\succeq-\frac12\,I\).
For every \(g\in H^1(\pi_\eta^Y)\), exchangeability and
\eqref{eq:Yplusminus} give
\begin{align}
 \bigl\langle g,(I-\opUPP)g\bigr\rangle
 &=
 \frac12\,\E\bigl[\bigl(
     g(Y_+)-g(Y_-)\bigr)^2
 \bigr]
 =
 \frac12\,\E_X\E_Z
 \bigl[\bigl(
   g(X+\sqrt\eta Z)-g(X-\sqrt\eta Z)
 \bigr)^2\bigr]\,.
 \label{eq:I-minus-UPP-dirichlet}
\end{align}
For fixed \(X=x\), the Gaussian Poincar\'e inequality implies
\begin{align*}
 \E_Z
 \bigl[\bigl(
     g(x+\sqrt\eta Z)-g(x-\sqrt\eta Z)\bigr)^2
 \bigr]
 &\le
 \eta\,\E_Z
 \bigl\|
   \nabla g(x+\sqrt\eta Z)
   +\nabla g(x-\sqrt\eta Z)
 \bigr\|^2\,.
\end{align*}
Averaging in \(X\) and using that \(Y_+\) and \(Y_-\) have the same
marginal gives
\begin{equation}
 \bigl\langle g, (I-\opUPP)g\bigr\rangle
 \le
 2\eta\,\|\nabla g\|^2\,.
 \label{eq:I-minus-UPP-gradient}
\end{equation}

Applying \eqref{eq:I-minus-UPP-gradient} with \(g=\opUPP f\) and using
\eqref{eq:UPP-sharp-smoothing}, we obtain
\begin{align*}
 \bigl\langle f,\opUPP^2(I-\opUPP)f\bigr\rangle
 &=\bigl\langle \opUPP f,(I-\opUPP)\opUPP f\bigr\rangle
 \le2\eta\,\|\nabla\opUPP f\|^2
 \le
 \frac12\,\bigl\langle f,(I-\opUPP^2)f\bigr\rangle\,.
\end{align*}
Thus, as an operator inequality, $2\opUPP^2(I-\opUPP)\preceq I-\opUPP^2$,
or equivalently
\begin{equation}
 (I-\opUPP)^2(I+2\opUPP)\succeq0\,.
 \label{eq:UPP-polynomial-lower-bound}
\end{equation}
Since \(\opUPP\) is a self-adjoint contraction, its spectrum lies in
\([-1,1]\).  The scalar inequality corresponding to
\eqref{eq:UPP-polynomial-lower-bound} is
\[
 (1-\lambda)^2(1+2\lambda)\ge0\,,
 \qquad \lambda\in\operatorname{Spec}(\opUPP)\,,
\]
and therefore \(\operatorname{Spec}(\opUPP)\subseteq[-1/2,1]\).  Hence
\(-\frac12I\,\preceq\opUPP\).
Together with \eqref{eq:UPP-sharp-contraction-proof}, this proves
\eqref{eq:UPP-two-sided-bound}.

Finally,
\[
 \opGammaP^2
 =I-\opUPP^2
 \succeq
 (1-\rho_{\rm mac}^2)\,I
 =
 \frac{4\alpha\eta}{(1+\alpha\eta)^2}\,I\,,
\]
and therefore
\[
 \opGammaP
 \succeq
 \frac{2\sqrt{\alpha\eta}}{1+\alpha\eta}\,I
 =\gamma_{\rm gap}I\,.
\]
This proves \eqref{eq:macroscopic-coercivity}.
\end{proof}

The half-turn analysis below uses the following alternative form of
\eqref{eq:UPP-sharp-derivative}.

\begin{lemma}[Derivative formula for \(\opUPP\)]
\label{lem:UPP-derivative}
For every \(f\in H^1(\pi_\eta^Y)\) and \(\pi_\eta^Y\)-almost every \(y\),
\begin{equation}
 \nabla\opUPP f(y)
 =\E[\nabla f(Y_-)\mid Y_+=y]
 -\Cov(f(Y_-),\nabla V(X)\mid Y_+=y)\,.
 \label{eq:UPP-derivative-app}
\end{equation}
\end{lemma}

\begin{proof}
First let \(f\in C_{\rm c}^\infty(\R^d)\).  The representation
\eqref{eq:UPP-sharp-derivative} gives
\begin{equation*}
 \nabla\opUPP f(y)
 =\Cov(f(Y_-),s_y(Y_-)\mid Y_+=y)\,.
\end{equation*}
Using the definition of \(s_y\), we write
\begin{equation*}
 s_y(u)
 =\Bigl\{\frac12\nabla V((y+u)/2)-\frac{y-u}{4\eta}\Bigr\}
  -\nabla V((y+u)/2)\,.
\end{equation*}
To rewrite \eqref{eq:UPP-sharp-derivative} in the claimed form, we
only need to identify the covariance generated by the vector in braces.
Integration by parts in \(u\) against the conditional density \(p_y\) gives
\begin{equation}
 \E[\nabla f(Y_-)\mid Y_+=y]
 =\E \Bigl[
 f(Y_-)\,\Bigl\{\frac12\nabla V(X)-\frac{y-Y_-}{4\eta}\Bigr\}
 \ \Bigm|\ Y_+=y\Bigr]\,.
 \label{eq:Yminus-IBP}
\end{equation}
Applying \eqref{eq:Yminus-IBP} to \(f\equiv1\), justified by compactly
supported cutoffs converging to \(1\), gives
\begin{equation*}
 \E \Bigl[
  \frac12\nabla V(X)-\frac{y-Y_-}{4\eta}
  \,\Bigm|\,Y_+=y\Bigr]=0\,.
\end{equation*}
Combining this with \eqref{eq:Yminus-IBP}, we obtain
\begin{align*}
 \nabla\opUPP f(y)
 &=\Cov \Bigl(
    f(Y_-),\frac12\nabla V(X)-\frac{y-Y_-}{4\eta}
    \,\Bigm|\,Y_+=y\Bigr)
   -\Cov(f(Y_-),\nabla V(X)\mid Y_+=y)\\
 &=\E \Bigl[
    f(Y_-)\,\Bigl\{\frac12\nabla V(X)-\frac{y-Y_-}{4\eta}\Bigr\}
    \,\Bigm|\,Y_+=y\Bigr]
   -\Cov(f(Y_-),\nabla V(X)\mid Y_+=y)\\
 &=\E[\nabla f(Y_-)\mid Y_+=y]
   -\Cov(f(Y_-),\nabla V(X)\mid Y_+=y)\,.
\end{align*}
This is \eqref{eq:UPP-derivative-app} for smooth \(f\).

The extension to \(H^1(\pi_\eta^Y)\) is a standard approximation
argument once we check that the covariance term is continuous in
\(L^2\).  By \eqref{eq:conditional-poincare-prelim}, the conditional
law of \(X\) given \(Y_+=y\) has Poincar\'e constant at most
\(\eta/(1+\alpha\eta)\le\eta\).  Hence, for every unit vector
\(a\in\R^d\),
\begin{equation*}
 \Var(a^\top\nabla V(X)\mid Y_+=y)
 \le\eta\,\E[\|\nabla^2V(X)\,a\|^2\mid Y_+=y]
 \le\beta^2\eta\,.
\end{equation*}
Conditional Cauchy--Schwarz and
\eqref{eq:UPP-exchangeable} therefore imply
\[
 \bigl\|
 \Cov(f(Y_-),\nabla V(X)\mid Y_+)
 \bigr\|_{L^2}
 \le\beta\sqrt\eta\,\|\opUperpP f\|\,,
\]
and the same bound applies to differences.  Conditional expectation is
an \(L^2\) contraction, and \eqref{eq:UPP-sharp-smoothing} controls the
left side.  Thus every term in \eqref{eq:UPP-derivative-app} is
continuous under \(H^1\) approximation.  Density of
\(C_{\rm c}^\infty(\R^d)\) in \(H^1(\pi_\eta^Y)\) completes the proof.
\end{proof}

\subsection{The half-turn operator on affine functions}
\label{sec:affine-halfturn}

This subsection establishes a conditional half-turn estimate for affine
observables.  After rescaling \(X-y\) by the Gaussian scale
\(\sqrt\eta\), we show that the half-turn nearly maps every affine
function with mean zero to its negative.  The proof combines a Duhamel
representation of the half-turn with covariance identities for the
rescaled conditional law.

In \cref{app:global-halfturn-cutoff}, this estimate will be applied to
the conditional affine part of \(\opUperpP w\) in the proof of
\cref{prop:global-halfturn-cutoff}.

We first fix \(y\in\R^d\) and work under the conditional law of \(X\)
given \(Y=y\).  The bounds below are uniform in \(y\) and will later be
integrated over \(y\).  To put the Gaussian confinement and harmonic
dynamics on unit scale, set
\begin{equation*}
 W\deq\frac{X-y}{\sqrt\eta}\,,
 \qquad
 q_y\deq\Law(W\mid Y=y)\,.
\end{equation*}
The density of \(q_y\) and the scaled potential gradient are
\begin{equation*}
 q_y(\dd w)
 \propto
 \exp \Bigl\{-V(y+\sqrt\eta w)-\frac12\|w\|^2\Bigr\}\,\dd w\,,
 \qquad
 g_y(w)\deq\sqrt\eta\,\nabla V(y+\sqrt\eta w)\,.
\end{equation*}
Thus, the negative log-density of \(q_y\) has Hessian
\(I+\nabla g_y\), and
\begin{equation}
 \nabla\log q_y(w)=-(w+g_y(w))\,,
 \qquad
 \alpha\eta I\preceq\nabla g_y(w)\preceq\beta\eta I\,.
 \label{eq:normalized-log-density-curvature}
\end{equation}
In particular, \(q_y\) is \(1\)-strongly log-concave.  We use the same
symbol \(\opH_y\) for the conditional half-turn operator in the
normalized coordinate \(w=(x-y)/\sqrt\eta\):
\begin{equation*}
 (\opH_y f)(w)
 \deq\int_{\R^d}
 f \Bigl(\frac{x'-y}{\sqrt\eta}\Bigr)\,
 \mathcal H_y(y+\sqrt\eta w,\dd x')\,.
\end{equation*}
By \cref{prop:half-invariance}, \(\opH_y\) is a self-adjoint Markov
contraction on \(L^2(q_y)\).

Put
\begin{equation*}
 \bar w_y\deq\E_{q_y}W\,,
 \qquad
 \Sigma_y\deq\Cov_{q_y}(W)\,.
\end{equation*}
Thus, \(\Sigma_y\) is the conditional covariance in the normalized
coordinate; equivalently, we have the identity
\(\Cov(X\mid Y=y)=\eta\Sigma_y\).  Accordingly,
\[
 I-\Sigma_y=\eta\nabla^2U_\eta(y)\,,
\]
by \eqref{eq:Moreau-Hessian}.  Thus \(I-\Sigma_y\) is both the covariance
defect relative to a standard Gaussian and the normalized curvature of
the \(Y\)-marginal.  For \(a\in\R^d\), write
\(\ell_{a,y} : w\mapsto a^\top(w-\bar w_y)\); this affine function has mean zero
under \(q_y\).

The main result of this subsection is the following estimate.

\begin{proposition}[Weighted half-turn estimate for affine functions]
\label{prop:weighted-affine-halfturn}
Assume \(\beta\eta\le1/2\).  Uniformly in \(y\), every
\(\ell_{a,y} : w\mapsto a^\top(w-\bar w_y)\) satisfies
\begin{equation}
 \|(I+\opH_y)\ell_{a,y}\|_{L^2(q_y)}^2
 \le C\beta\eta\,a^\top(I-\Sigma_y)a\,.
 \label{eq:weighted-affine-halfturn}
\end{equation}
More generally, suppose that \(y\mapsto a_y\) is measurable and
\(\int a_y^\top(I-\Sigma_y)a_y\,\pi_\eta^Y(\dd y)<\infty\), and define
\[
 \ell(x,y)\deq
 a_y^\top\Bigl(\frac{x-y}{\sqrt\eta}-\bar w_y\Bigr)\,.
\]
Then, \(\ell\in\cH_\perp\) and
\begin{equation*}
 \|(I+\opHperpperp)\ell\|^2
 \le C\beta\eta
 \int a_y^\top(I-\Sigma_y)a_y\,\pi_\eta^Y(\dd y)\,.
\end{equation*}
\end{proposition}

We prove the proposition after two preparatory results.  The first gives a
Duhamel representation of the deviation from exact sign reversal; the
second bounds the resulting term in terms of $I - \Sigma_y$.

To analyze \(\opH_y\ell_{a,y}\), first fix a reference point
\(\widetilde w\in\R^d\) in the rescaled coordinate.  The resulting
phase-space process is
\begin{equation*}
 \begin{gathered}
  \dot w=p\,,
  \qquad
  \dot p=-\{w+g_y(\widetilde w)\}\,,\\
  \text{event rate }[p^\top\{g_y(w)-g_y(\widetilde w)\}]_+\,,\\
  p\longmapsto R_{g_y(w)-g_y(\widetilde w)}p
  \quad\text{at a bounce}\,.
 \end{gathered}
\end{equation*}
The construction and path reversal arguments of
Appendix~\ref{app:halfturn-process} apply verbatim, with phase space
law \(q_y\otimes\cN(0,I)\).  Denote the resulting strongly continuous
contraction semigroup by
\((\mathsf T_t^{y,\widetilde w})_{t\ge0}\).  On smooth compactly
supported observables, its generator is
\begin{equation*}
 \begin{split}
 \opL_{y,\widetilde w}f(w,p)
 &\deq p^\top\nabla_w f(w,p)
   -\{w+g_y(\widetilde w)\}^\top\nabla_p f(w,p)\\
   &\qquad{}+[p^\top\{g_y(w)-g_y(\widetilde w)\}]_+
 \bigl\{f(w,R_{g_y(w)-g_y(\widetilde w)}p)-f(w,p)\bigr\}\,.
 \end{split}
\end{equation*}
The operator \(\opH_y\) is
recovered by drawing \(\widetilde W\sim q_y\), running the corresponding
process for time \(\pi\), and averaging over \(\widetilde W\) and the
initial Gaussian momentum.

If \(g_y\equiv0\), then \(q_y=\cN(0,I)\), \(\bar w_y=0\), and the bounce
rate vanishes.  The dynamics reduce to harmonic motion,
\[
 w_t=w_0\cos t+p_0\sin t\,,
 \qquad
 p_t=-w_0\sin t+p_0\cos t\,,
\]
so \(\ell_{a,y}(w_\pi)=-\ell_{a,y}(w_0)\).  For general \(g_y\), the
discrepancy between the harmonic center \(-g_y(\widetilde w)\) and the
conditional mean \(\bar w_y\), together with the residual gradient
bounces, produces an inhomogeneous term.  The next lemma identifies this
term.

\begin{lemma}[Affine Duhamel formula]
\label{lem:affine-duhamel}
For \(a\in\R^d\) and \(\widetilde w\in\R^d\), set
\begin{equation*}
 k_{a,\widetilde w}(w,p)
 \deq-a^\top\{\bar w_y+g_y(\widetilde w)\}
 -2\,[p^\top\{g_y(w)-g_y(\widetilde w)\}]_+^2\,
   \frac{a^\top\{g_y(w)-g_y(\widetilde w)\}}
        {\|g_y(w)-g_y(\widetilde w)\|^2}\,,
\end{equation*}
where the quotient is defined as zero when
\(g_y(w)=g_y(\widetilde w)\).  Then
\(k_{a,\widetilde w}\in L^2(q_y\otimes\cN(0,I))\), and for every
\(t\ge0\),
\begin{equation*}
 \mathsf T_t^{y,\widetilde w}\ell_{a,y}
 =\ell_{a,y}\cos t+a^\top p\sin t
  +\int_0^t
    \mathsf T_s^{y,\widetilde w}k_{a,\widetilde w}\sin(t-s)\,\dd s
\end{equation*}
in \(L^2(q_y\otimes\cN(0,I))\).  In particular,
\begin{equation}
 (I+\mathsf T_\pi^{y,\widetilde w})\ell_{a,y}
 =\int_0^\pi
   \mathsf T_s^{y,\widetilde w}k_{a,\widetilde w}\sin s\,\dd s\,.
 \label{eq:Duhamel-halfturn-rigorous}
\end{equation}
\end{lemma}

\begin{proof}
The generator maps the position observable \(\ell_{a,y}\) to the
momentum observable \(a^\top p\), and maps \(a^\top p\) back to
\(-\ell_{a,y}+k_{a,\widetilde w}\).  Thus, this pair evolves as a harmonic
system with the additional term \(k_{a,\widetilde w}\), and variation of
constants gives the claimed formula.  We first establish the required
\(L^2\) bounds and justify these generator identities for the unbounded
observables.

First, \(k_{a,\widetilde w}\) is square integrable.  Its first term is
constant in \((w,p)\), while the absolute value of its second term is at
most
\begin{align*}
 2\,[p^\top\{g_y(w)-g_y(\widetilde w)\}]_+^2\,
 \frac{|a^\top\{g_y(w)-g_y(\widetilde w)\}|}
      {\|g_y(w)-g_y(\widetilde w)\|^2}
 &\le2\,\|a\|\,\|p\|^2\,
       \|g_y(w)-g_y(\widetilde w)\|\\
 &\le2\beta\eta\,\|a\|\,\|p\|^2\,\|w-\widetilde w\|\,.
\end{align*}
Since \(q_y\) is \(1\)-strongly log-concave and the momentum is Gaussian,
this bound shows that \(k_{a,\widetilde w}\) belongs to
\(L^2(q_y\otimes\cN(0,I))\).  The same moment bounds give
\(\ell_{a,y},a^\top p\in L^2(q_y\otimes\cN(0,I))\).

Using
\(a^\top(R_hp-p)=-2(p^\top h)(a^\top h)/\|h\|^2\) and
\([u]_+u=[u]_+^2\), the generator formula gives
\begin{align*}
 \opL_{y,\widetilde w}\ell_{a,y}
 &=a^\top p\,,\qquad
 \opL_{y,\widetilde w}(a^\top p)
 =-\ell_{a,y}+k_{a,\widetilde w}\,.
\end{align*}
To justify these identities for the unbounded observables, set $H_{\widetilde w}(w,p)
 \deq\frac12\,\{\|w+g_y(\widetilde w)\|^2+\|p\|^2\}$,
which is preserved by both the flow and the bounces.  Let \(\chi_R\) be
a smooth cutoff equal to one on \([0,R]\) and zero on
\([2R,\infty)\).  Conservation of \(H_{\widetilde w}\) gives, for
\(F\in\{\ell_{a,y},a^\top p\}\),
$\opL_{y,\widetilde w}\{\chi_R(H_{\widetilde w})F\}
 =\chi_R(H_{\widetilde w})\,\opL_{y,\widetilde w}F$.
Applying Dynkin's formula and then letting \(R\to\infty\) in \(L^2\),
using the preceding integrability bounds and semigroup contractivity,
yields
\begin{align*}
 \mathsf T_t^{y,\widetilde w}\ell_{a,y}-\ell_{a,y}
 &=\int_0^t\mathsf T_s^{y,\widetilde w}(a^\top p)\,\dd s\,,\\
 \mathsf T_t^{y,\widetilde w}(a^\top p)-a^\top p
 &=\int_0^t\mathsf T_s^{y,\widetilde w}
   (-\ell_{a,y}+k_{a,\widetilde w})\,\dd s\,.
\end{align*}
Applying variation of constants to this integral system gives
\[
 \mathsf T_t^{y,\widetilde w}\ell_{a,y}
 =\ell_{a,y}\cos t+a^\top p\sin t
 +\int_0^t
   \mathsf T_s^{y,\widetilde w}k_{a,\widetilde w}\sin(t-s)\,\dd s\,.
\]
This is the first identity in the lemma.  At \(t=\pi\), the homogeneous
part is \(-\ell_{a,y}\) and \(\sin(\pi-s)=\sin s\), which gives
\eqref{eq:Duhamel-halfturn-rigorous}.
\end{proof}

The proof of \cref{prop:weighted-affine-halfturn} will also use the
following identities and bounds.

\begin{lemma}[Conditional gradient covariance bound]
\label{lem:conditional-gradient-covariance}
Fix \(y\) and assume \(\beta\eta\le1/2\).  Then
\(\E_{q_y}g_y(W)=-\bar w_y\), and
\begin{equation}
 \Cov_{q_y}(W,g_y(W))=I-\Sigma_y\,,
 \qquad
 \Cov_{q_y}(g_y(W))
 \preceq\frac{\beta\eta}{1-\beta\eta}\,(I-\Sigma_y)
 \preceq2\beta\eta\,(I-\Sigma_y)\,.
 \label{eq:residual-gradient-covariance-bound}
\end{equation}
\end{lemma}

\begin{proof}
Under \(Y_+=y\), we have \(X=y+\sqrt\eta W\) and
\(Y_-=y+2\sqrt\eta W\).  Applying \eqref{eq:Yminus-IBP} to
\(f : u\mapsto F((u-y)/(2\sqrt\eta))\) therefore gives
\[
 \E_{q_y}\nabla F(W)
 =\E_{q_y}[F(W)\,\{W+g_y(W)\}]\,.
\]
Taking \(F\equiv1\) and then the coordinate functions of \(W\) and
\(g_y(W)\), with the same cutoff approximation as in
\cref{lem:UPP-derivative}, gives
\[
 \E g_y(W)=-\bar w_y\,,
 \qquad
 \Cov(W,g_y(W))=I-\Sigma_y\,,
 \qquad
 \E\nabla g_y(W)=I-\Sigma_y+\Cov(g_y(W))\,.
\]
For every \(a\in\R^d\), the same conditional Poincar\'e estimate used in
the proof of \cref{lem:UPP-derivative}, together with
\((\nabla g_y)^2\preceq\beta\eta\,\nabla g_y\) then give
\begin{align*}
 a^\top\Cov(g_y(W))a
 &\le\E\|\nabla g_y(W)\,a\|^2
 \le\beta\eta\,
 a^\top\{I-\Sigma_y+\Cov(g_y(W))\}a\,.
\end{align*}
Rearranging proves \eqref{eq:residual-gradient-covariance-bound}.
\end{proof}

\begin{proof}[Proof of \cref{prop:weighted-affine-halfturn}]
Fix \(y\), and let
\(\widetilde W\sim q_y\) be independent of
\((W,P)\sim q_y\otimes\cN(0,I)\).  We first apply the Duhamel formula to
identify the estimate that remains to be proved.  Using, in order,
Jensen's inequality in \(\widetilde W\),
\eqref{eq:Duhamel-halfturn-rigorous}, the Bochner triangle inequality,
the contractivity of Gaussian momentum averaging and
\(\mathsf T_t^{y,\widetilde w}\), and
\(\int_0^\pi\sin t\,\dd t=2\), we obtain
\begin{align*}
 \|(I+\opH_y)\ell_{a,y}\|_{L^2(q_y)}^2
 &\le \E_{\widetilde W}\bigl\|
  \E_P[(I+\mathsf T_\pi^{y,\widetilde W})\ell_{a,y}]
 \bigr\|_{L^2(q_y)}^2\\
 &\le \E_{\widetilde W}\Bigl(
   \int_0^\pi
   \|\mathsf T_t^{y,\widetilde W}k_{a,\widetilde W}
     \|_{L^2(q_y\otimes\cN(0,I))}\sin t\,\dd t
 \Bigr)^2
 \le4\,\E_{W,\widetilde W,P}k_{a,\widetilde W}^2\,.
\end{align*}
It remains to bound the second moment on the right.

Conditional on \((W,\widetilde W)\), a direct calculation with the
Gaussian momentum \(P\), together with
\(\E_{q_y}g_y(W)=-\bar w_y\) from
\cref{lem:conditional-gradient-covariance}, gives
\begin{align*}
 \E_P[k_{a,\widetilde W}\mid W,\widetilde W]
 &=-a^\top\{g_y(W)-\E_{q_y}g_y(W)\}\,,\\
 \Var_P(k_{a,\widetilde W}\mid W,\widetilde W)
 &=5\,[a^\top\{g_y(W)-g_y(\widetilde W)\}]^2\,.
\end{align*}
After averaging over the independent variables \(W\) and
\(\widetilde W\), the conditional variance term contributes
\(10\Var_{q_y}(a^\top g_y(W))\), while the squared conditional mean
contributes one more copy of the same variance.  Consequently,
\[
 \E_{W,\widetilde W,P}k_{a,\widetilde W}^2
 =11\Var_{q_y}(a^\top g_y(W))\,.
\]
By \cref{lem:conditional-gradient-covariance},
\[
 \E_{W,\widetilde W,P}k_{a,\widetilde W}^2
 \le C\beta\eta\,a^\top(I-\Sigma_y)a\,.
\]
Combining the last estimate with the Duhamel bound proves
\eqref{eq:weighted-affine-halfturn}.  For the global
assertion, \cref{prop:augmentation} gives
\(I-\Sigma_y\succeq\alpha\eta\,(1+\alpha\eta)^{-1}\,I\), so the stated
integrability condition implies \(\ell\in L^2(\pi_\eta)\).  Its
conditional mean is zero.  Applying the conditional inequality with
\(a=a_y\) and integrating in \(y\) proves the result.
\end{proof}

\subsection{Proof of
  \texorpdfstring{\cref{prop:global-halfturn-cutoff}}
  {Proposition~\ref*{prop:global-halfturn-cutoff}}:
  From affine estimates to the half-turn bound}
\label{app:global-halfturn-cutoff}

The following proposition controls the part of the \(\opUPP\)-spectrum near
\(+1\).

\begin{proposition}[{Interpolation bound on the \(\opUPP\)-spectral subspace
associated with \([\frac12,1]\)}]
\label{prop:global-positive-edge-interpolation}
Assume \(\beta\eta\le c\).  If
\(w\in\ran(\mathbf1_{[1/2,1]}(\opUPP))\), then
\begin{equation}
 \|(I+\opHperpperp)\opUperpP w\|^2
 \le C\,\|\opGammaP w\|\,\|\opGammaP^3w\|\,.
 \label{eq:global-positive-edge-interpolation}
\end{equation}
\end{proposition}

We now prove \cref{prop:global-halfturn-cutoff}.

\begin{proof}[Proof of \cref{prop:global-halfturn-cutoff}]
We regard \(\opUPP\) and \(\opGammaP\) as operators on
\(\cH_{\mathsf P,0}\) and work with the spectral subspaces
\[
\mathcal H\deq
 \ran(\mathbf1_{[1/2,1]}(\opUPP))\,,\qquad
\mathcal H'\deq
 \ran(\mathbf1_{[-1/2,1/2)}(\opUPP))\,.
\]
Recall that
\(\opVperpP=\opUperpP\opGammaP^{-1}\), and therefore
\cref{prop:global-positive-edge-interpolation} yields
\begin{align*}
 \|(I+\opHperpperp)\opVperpP v\|^2
 &=\|(I+\opHperpperp)\opUperpP \opGammaP^{-1}v\|^2
 \le C\,\|v\|\,\|\opGammaP^2v\|\,\qquad
 \forall v\in \mathcal H\,.
\end{align*}
Recall that by \cref{lem:macroscopic-coercivity},
\(\gamma_{\rm gap}^2I\preceq\opGammaP^2\preceq I\).
Applying \cref{lem:operator-comparison} with $\mathsf A = C^{-1/2}\,(\mathsf T^*\mathsf T)^{1/2}$, $\mathsf B = \opGammaP^2|_{\mathcal H}$,
\(m=\gamma_{\rm gap}^2\), and \(M=1\), where $\mathsf T \deq (I+\opHperpperp)\opVperpP|_{\mathcal H}$, gives a factor
\(1+\log(1/\gamma_{\rm gap})\), up to a universal constant.  Shrinking the
universal constant \(c\) in the assumption if necessary, we may suppose that
\(\beta\eta\le1/16\).  Then
\(\gamma_{\rm gap}\le2\sqrt{\alpha\eta}\le1/2\), so the additive constant is
absorbed by \(\log(1/\gamma_{\rm gap})\).  Therefore,
\begin{align*}
 \|(I+\opHperpperp)\opVperpP v\|^2
 \le C\log(1/\gamma_{\rm gap})\,\|\opGammaP v\|^2\,\qquad
 \forall v\in \mathcal H\,.
\end{align*}
On the other hand, on \(\mathcal H'\), no further spectral localization
is needed.  Since \(\opHperpperp\) is a contraction,
\(\opVperpP\) is an isometry, and
\(\opGammaP^2\succeq\frac34I\) on \(\mathcal H'\), we have
\[
 \|(I+\opHperpperp)\opVperpP v\|
 \le2\,\|v\|\le 4\,\|\opGammaP v\|\,\qquad
 \forall v\in \mathcal H'\,.
\]
For a general \(v\in\cH_{\mathsf P,0}\), decompose
\(v=v_{\mathcal H}+v_{\mathcal H'}\).  Combining the preceding two
estimates with \((a+b)^2\le2a^2+2b^2\), and using that \(\opGammaP\)
commutes with \(\opUPP\), proves
\eqref{eq:global-halfturn-cutoff}.
\end{proof}

We prove \cref{prop:global-positive-edge-interpolation} after the following preparatory results.  The first
approximates \(\opUperpP w\), conditionally on \(Y_+\), by a function
that is affine in the rescaled coordinate \(W=(X-Y_+)/\sqrt\eta\), so
that \cref{prop:weighted-affine-halfturn} applies.  The second bounds the resulting
affine data by the gradients of \(w\) and \((I-\opUPP)w\), using the
derivative formula \eqref{eq:UPP-derivative-app}.  Both results hold
for arbitrary functions in \(H^1(\pi_\eta^Y)\); the spectral
condition enters when we invert $\opUPP$.

Recall from \cref{sec:affine-halfturn} the conditional law
\(q_y\deq \Law(W\mid Y_+=y)\), its mean \(\bar w_y\) and covariance
\(\Sigma_y\), and that, conditionally on \(Y_+=y\), we have
\(X=y+\sqrt\eta\, W\) and \(Y_-=y+2\sqrt\eta\, W\).

\begin{lemma}[Conditional affine approximation and energy bound]
\label{lem:conditional-affine-approximation-energy}
Let \(w\in H^1(\pi_\eta^Y)\), and let
\(a_y\deq\E[\nabla w(Y_-)\mid Y_+=y]\) denote the conditional mean of
the gradient.  For \(\pi_\eta^Y\)-almost every \(y\), define
\begin{equation*}
 e_y(u)\deq w(y+2\sqrt\eta u)-(\opUPP w)(y)
   -2\sqrt\eta\,a_y^\top(u-\bar w_y)\,,
 \qquad u\in\R^d\,.
\end{equation*}
Then, \(e_y\) has \(q_y\)-mean zero, the decomposition
\begin{equation}
 \opUperpP w
 =2\sqrt\eta\,a_{Y_+}^\top(W-\bar w_{Y_+})+e_{Y_+}(W)
 \label{eq:conditional-affine-decomposition}
\end{equation}
holds in \(L^2(\pi_\eta)\), and, with
\begin{equation}
 \mathcal R
 \deq\|\nabla w\|^2-\int\|a_y\|^2\,\pi_\eta^Y(\dd y)
 =\E\|\nabla w(Y_-)-a_{Y_+}\|^2\,,
 \label{eq:conditional-affine-R}
\end{equation}
one has
\begin{equation}
 \E[e_{Y_+}(W)^2]\le4\eta\mathcal R\,.
 \label{eq:conditional-affine-remainder}
\end{equation}
Moreover, if \(\beta\eta\leq c\) for a sufficiently small
universal constant $c>0$, then
\begin{equation}
 \mathcal R+
 \int a_y^\top(I-\Sigma_y)a_y\,\pi_\eta^Y(\dd y)
 \le
 C\,\|\nabla(I-\opUPP)w\|\,\|\nabla w\|\,.
 \label{eq:affine-approximation-targets}
\end{equation}
\end{lemma}

\begin{proof}
By definition,
\[
 \opUperpP w
 =w(Y_-)-\E[w(Y_-)\mid Y_+]
 =w(Y_-)-(\opUPP w)(Y_+)\,.
\]
Conditioning on \(Y_+=y\) and substituting \(Y_-=y+2\sqrt\eta\, W\)
gives, for \(\pi_\eta^Y\)-almost every \(y\),
\[
 (\opUperpP w)(u)
 =w(y+2\sqrt\eta u)-(\opUPP w)(y)\,,
 \qquad u\in\R^d\,,
\]
which is \eqref{eq:conditional-affine-decomposition}.  This function
belongs to \(L^2(q_y)\) and has \(q_y\)-mean zero; since the
subtracted affine function is \(q_y\)-centered, \(e_y\) has
\(q_y\)-mean zero as well.

The coefficient \(2\sqrt\eta\,a_y\) is chosen to center the gradient
of \(e_y\): differentiating the preceding identity and averaging
under \(q_y\) gives
\[
 \int_{\R^d}\nabla_u(\opUperpP w)(u)\,q_y(\dd u)
 =2\sqrt\eta\,\E[\nabla w(Y_-)\mid Y_+=y]
 =2\sqrt\eta\,a_y\,,
\]
and hence
\[
 \nabla_ue_y(W)=2\sqrt\eta\,\{\nabla w(Y_-)-a_y\}\,.
\]
Since \(q_y\) has Poincar\'e constant at most $1$, it follows that
\[
 \E[e_y(W)^2\mid Y_+=y]
 \le\E[\|\nabla_ue_y(W)\|^2\mid Y_+=y]
 =4\eta\,\E[\|\nabla w(Y_-)-a_y\|^2\mid Y_+=y]\,.
\]
Integrating in \(y\) and expanding the square using the definition of
\(a_y\) proves \eqref{eq:conditional-affine-remainder}.

We turn to the energy estimate.  We may assume \(\beta\eta\le1/2\).
Recall that
\(g_y(u)=\sqrt\eta\,\nabla V(y+\sqrt\eta u)\), so that
\(\nabla V(X)=\eta^{-1/2}g_y(W)\), and that
\(\Cov(W,g_y(W)\mid Y_+=y)=I-\Sigma_y\) by
\cref{lem:conditional-gradient-covariance}.  The decomposition
\eqref{eq:conditional-affine-decomposition} therefore gives
\[
 \Cov(w(Y_-),\nabla V(X)\mid Y_+=y)
 =2(I-\Sigma_y)a_y
  +\eta^{-1/2}\varepsilon_y\,,
 \qquad
 \varepsilon_y\deq\Cov(e_y(W),g_y(W)\mid Y_+=y)\,.
\]
Substituting this into \eqref{eq:UPP-derivative-app} yields
\begin{equation}
 [\nabla(I-\opUPP)w](y)
 =\nabla w(y)-a_y+2(I-\Sigma_y)a_y
   +\eta^{-1/2}\varepsilon_y\,.
 \label{eq:exact-matrix-balance}
\end{equation}
Conditional Cauchy--Schwarz and
\cref{lem:conditional-gradient-covariance} give, for every
\(v\in\R^d\),
\begin{equation}
 |\varepsilon_y^\top v|^2
 \le C\beta\eta\,
 \E[e_y(W)^2\mid Y_+=y]\,
 v^\top(I-\Sigma_y)v\,.
 \label{eq:conditional-affine-error}
\end{equation}
We will also use
\begin{equation}
 0\preceq I-\Sigma_y\preceq\beta\eta I\,.
 \label{eq:conditional-covariance-defect-simple}
\end{equation}
The two-sided bound in
\cref{lem:macroscopic-coercivity}, together with the fact that
\(\opUPP\) fixes constants, gives
\[
 (I-\opUPP)^2\preceq3\,(I-\opUPP^2)
\]
on \(L^2(\pi_\eta^Y)\).  Since
\(a_y=(\opUPP\nabla w)(y)\), with \(\opUPP\) acting coordinatewise,
\begin{equation}
 \mathcal R
 =\bigl\langle \nabla w, (I-\opUPP^2)\nabla w\bigr\rangle\,,
 \qquad
 \int\|\nabla w(y)-a_y\|^2\,\pi_\eta^Y(\dd y)
 \le3\mathcal R\,.
 \label{eq:gradient-defect-by-R}
\end{equation}

We now control \(\mathcal R\) and the affine energy simultaneously.
Take the \(L^2(\pi_\eta^Y)\) inner product of
\eqref{eq:exact-matrix-balance} with the vector field
\(y\mapsto\nabla w(y)+a_y\).  Since
\[
 \int\{\nabla w(y)-a_y\}^\top
          \{\nabla w(y)+a_y\}\,\pi_\eta^Y(\dd y)
 =\mathcal R\,,
\]
we obtain the exact identity
\begin{align}
 \mathcal R
 +4\int a_y^\top(I-\Sigma_y)a_y\,\pi_\eta^Y(\dd y)
 &=
 \int [\nabla(I-\opUPP)w](y)^\top
       \{\nabla w(y)+a_y\}\,\pi_\eta^Y(\dd y)
 \notag\\
 &\qquad
 -2\int a_y^\top(I-\Sigma_y)
       \{\nabla w(y)-a_y\}\,\pi_\eta^Y(\dd y)
 \notag\\
 &\qquad
 -\eta^{-1/2}\int
   \varepsilon_y^\top\{\nabla w(y)+a_y\}\,
   \pi_\eta^Y(\dd y)\,.
 \label{eq:short-affine-energy-identity}
\end{align}

We bound the three terms on the right separately.  For the first,
conditional Jensen gives $\int\|a_y\|^2\,\pi_\eta^Y(\dd y)\le\|\nabla w\|^2$,
and hence
\begin{equation}
 \Bigl\lvert \int [\nabla(I-\opUPP)w](y)^\top
       \{\nabla w(y)+a_y\}\,\pi_\eta^Y(\dd y)\Bigr\rvert
  \leq2\,\|\nabla(I-\opUPP)w\|\,\|\nabla w\|\,.
 \label{eq:short-affine-first-term}
\end{equation}
For the second, Cauchy--Schwarz,
\eqref{eq:conditional-covariance-defect-simple}, and
\eqref{eq:gradient-defect-by-R} give
\begin{align}
 &2\,\Bigl|
 \int a_y^\top(I-\Sigma_y)
       \{\nabla w(y)-a_y\}\,\pi_\eta^Y(\dd y)
 \Bigr|
 \le
 C\sqrt{\beta\eta}\,
 \mathcal R^{1/2}\,
 \Bigl\{
 \int a_y^\top(I-\Sigma_y)a_y\,\pi_\eta^Y(\dd y)
 \Bigr\}^{1/2}\,.
 \label{eq:short-affine-second-term}
\end{align}
Finally, using the pointwise identity
\(\nabla w(y)+a_y=\{\nabla w(y)-a_y\}+2a_y\),
\eqref{eq:conditional-covariance-defect-simple} and
\eqref{eq:gradient-defect-by-R} imply
\begin{align}
 &\int
 \{\nabla w(y)+a_y\}^\top(I-\Sigma_y)
 \{\nabla w(y)+a_y\}\,\pi_\eta^Y(\dd y)
 \le
 C\beta\eta\,\mathcal R
 +8\int a_y^\top(I-\Sigma_y)a_y\,\pi_\eta^Y(\dd y)\,.
 \label{eq:affine-sum-defect}
\end{align}
Using \eqref{eq:conditional-affine-error}, Cauchy--Schwarz in \(y\),
\eqref{eq:conditional-affine-remainder}, and
\eqref{eq:affine-sum-defect}, we obtain
\begin{align}
 &\eta^{-1/2}\,\Bigl|
 \int\varepsilon_y^\top\{\nabla w(y)+a_y\}\,
      \pi_\eta^Y(\dd y)\Bigr|
 \notag\\
 &\qquad\le
 C\sqrt{\beta\eta\,\mathcal R}\,
 \Bigl\{
   \beta\eta\,\mathcal R
   +\int a_y^\top(I-\Sigma_y)a_y\,\pi_\eta^Y(\dd y)
 \Bigr\}^{1/2}
 \notag\\
 &\qquad\le
 C\sqrt{\beta\eta}\,
 \mathcal R^{1/2}\,
 \Bigl\{
 \int a_y^\top(I-\Sigma_y)a_y\,\pi_\eta^Y(\dd y)
 \Bigr\}^{1/2}
 +C\beta\eta\,\mathcal R\,.
 \label{eq:short-affine-third-term}
\end{align}

Substituting these three estimates into
\eqref{eq:short-affine-energy-identity} gives
\begin{align*}
 \mathcal R
 +4\int a_y^\top(I-\Sigma_y)a_y\,\pi_\eta^Y(\dd y)
 &\le
 2\,\|\nabla(I-\opUPP)w\|\,\|\nabla w\|\\
 &\qquad
 +C\sqrt{\beta\eta}\,
 \mathcal R^{1/2}
 \Bigl\{
 \int a_y^\top(I-\Sigma_y)a_y\,\pi_\eta^Y(\dd y)
 \Bigr\}^{1/2}
 +C\beta\eta\,\mathcal R\,.
\end{align*}
For sufficiently small universal \(\beta\eta\), the last two terms
are absorbed into the left-hand side.  This proves \eqref{eq:affine-approximation-targets}.
\end{proof}

\begin{proof}[Proof of \cref{prop:global-positive-edge-interpolation}]
Choose a small enough universal constant in the hypothesis, if necessary, so
that \cref{lem:conditional-affine-approximation-energy} applies.  On
\(\ran(\mathbf1_{[1/2,1]}(\opUPP))\), the operator
\(\opUPP^{-1}\) is bounded.
Applying \eqref{eq:UPP-sharp-smoothing} first to
\(\opUPP^{-1}w\) and then to
\(\opUPP^{-1}(I-\opUPP)w\) gives
\begin{align}
 \sqrt\eta\,\|\nabla w\|
 &\le C\,\|\opGammaP\opUPP^{-1}w\|
 \le C\,\|\opGammaP w\|\,,\notag\\
 \sqrt\eta\,\|\nabla(I-\opUPP)w\|
 &\le C\,\|\opGammaP\opUPP^{-1}(I-\opUPP)w\|
 \le C\,\|\opGammaP^3w\|\,.
 \label{eq:positive-edge-gradient-bounds}
\end{align}
Here, the final inequality in each line follows by spectral calculus
on \([\frac12,1]\).

Decompose \(\opUperpP w\) according to
\eqref{eq:conditional-affine-decomposition}.  The affine part has
coefficient \(2\sqrt\eta\,a_y\), so
\cref{prop:weighted-affine-halfturn}, \(\|I+\opHperpperp\|\le2\)
applied to the error term, and \((r+s)^2\le2r^2+2s^2\) give
\begin{align*}
 \|(I+\opHperpperp)\opUperpP w\|^2
 &\le C\beta\eta^2
       \int a_y^\top(I-\Sigma_y)a_y\,\pi_\eta^Y(\dd y)
      +C\,\E[e_{Y_+}(W)^2]\\
 &\le
 C\eta\,\|\nabla(I-\opUPP)w\|\,\|\nabla w\|\,,
\end{align*}
where the last line uses \(\beta\eta\le1\) and
\cref{lem:conditional-affine-approximation-energy}.  Finally,
\eqref{eq:positive-edge-gradient-bounds} gives \eqref{eq:global-positive-edge-interpolation}.
\end{proof}

\section{Technical tools}
\label{app:tools}

This appendix collects the functional analysis and log-concavity tools
used throughout the process and convergence analyses.

\subsection{Hilbert space and Markov operator conventions}
\label{app:hilbert-conventions}

\begin{definition}[Hilbert space notation]
Let \(\zeta\) be a probability measure.  The real Hilbert space
\(L^2(\zeta)\) consists of equivalence classes of measurable functions
\(f\) such that \(\int f^2\,\dd\zeta<\infty\).  Its inner product and
norm are
\begin{equation}
 \langle f,g\rangle_{L^2(\zeta)}\deq\int fg\,\dd\zeta\,,
 \qquad
 \|f\|_{L^2(\zeta)}\deq\langle f,f\rangle_{L^2(\zeta)}^{1/2}\,.
 \label{eq:L2-def}
\end{equation}
We write
\begin{equation}
 L_0^2(\zeta)\deq\Bigl\{f\in L^2(\zeta):\int f\,\dd\zeta=0\Bigr\}\,.
 \label{eq:centered-L2}
\end{equation}
When \(\zeta=\pi_\eta\), we abbreviate the inner product and norm by
\(\langle f,g\rangle\) and \(\|f\|\).
\end{definition}

Two functions equal \(\zeta\)-almost everywhere represent the same
vector.  A closed linear subspace \(\mathcal M\subseteq L^2(\zeta)\)
has orthogonal complement
\(\mathcal M^\perp\deq\{g:\langle f,g\rangle=0\text{ for all }f\in\mathcal M\}\).
Its orthogonal projection is the unique bounded operator
\(\mathsf P_{\mathcal M}\) for which $\mathsf P_{\mathcal M} f \in \mathcal M$ and
\(f-\mathsf P_{\mathcal M}f\in\mathcal M^\perp\).

All operators in the convergence analysis are bounded linear operators
on the indicated Hilbert spaces unless explicitly stated otherwise.
The adjoint \(\mathsf L^*\) is characterized by
\(\langle\mathsf Lf,g\rangle=\langle f,\mathsf L^*g\rangle\).
An operator is self-adjoint if \(\mathsf L=\mathsf L^*\), a contraction
if \(\|\mathsf Lf\|\le\|f\|\), and unitary if it is onto and preserves
norms.  A self-adjoint involution, meaning \(\mathsf L^2=I\), is
unitary.  For self-adjoint operators \(\mathsf L_1,\mathsf L_2\),
\begin{equation}
 \mathsf L_1\preceq\mathsf L_2
 \quad\Longleftrightarrow\quad
 \langle f,(\mathsf L_2-\mathsf L_1)f\rangle\ge0
 \quad\text{for every }f
 \label{eq:operator-order}
\end{equation}
is the Loewner order.  If \(\mathsf L\) is self-adjoint, the
spectral theorem defines \(\psi(\mathsf L)\) for bounded Borel functions
\(\psi\); in particular \(\mathbf 1_E(\mathsf L)\) is the orthogonal
projection onto the spectral set \(E\), and a non-negative operator has
a unique non-negative square root.

\begin{definition}[Markov operators and density evolution]
Let \(\mathcal K(z,\dd z')\) be a Markov kernel preserving \(\zeta\).  Its
Markov operator on observables is
\begin{equation}
 (\mathsf Kf)(z)\deq\int f(z')\,\mathcal K(z,\dd z')\,.
 \label{eq:markov-operator}
\end{equation}
It is a contraction on \(L^2(\zeta)\).  Its adjoint
\(\mathsf K^*\) governs density evolution: if a law has density \(h\) with
respect to \(\zeta\), then its density after one step is \(\mathsf K^*h\).
Constants are fixed by both operators.
\end{definition}

If \(P\ll\zeta\) has density \(h\), then
\begin{equation}
 \chi^2(P\|\zeta)=\|h-1\|_\zeta^2\,.
 \label{eq:chi-L2}
\end{equation}
Consequently, an operator norm bound for \(\mathsf K\) on
\(L_0^2(\zeta)\) gives the same bound for its adjoint
\(\mathsf K^*\), and hence gives \(\chi^2\) contraction.

\subsection{Strong log-concavity and covariance bounds}

We repeatedly use the following standard consequences of
Bakry--\'Emery curvature.  If a probability measure
\(\nu(\dd z)\propto e^{-W(z)}\,\dd z\) satisfies
\(\nabla^2W\succeq mI\), then, for smooth \(f\),
\begin{equation}
 \Var_\nu(f)\le\frac1m\,\E_\nu\|\nabla f\|^2\,,
 \label{eq:poincare-tool}
\end{equation}
and \(\nu\) satisfies the log-Sobolev inequality with constant
\(1/m\).  In particular, every \(L\)-Lipschitz \(f\) obeys Gaussian
concentration
\begin{equation}
 \Prob\{f-\E f\ge t\}
 \le\exp \Bigl(-\frac{mt^2}{2L^2}\Bigr)\,.
 \label{eq:concentration-tool}
\end{equation}
These statements follow from the Bakry--\'Emery criterion; see, for
example,~\cite[Chapters 4 and 5]{bgl}.

\begin{lemma}[Two-sided covariance bound]
\label{lem:covariance-bounds}
Let \(\nu(\dd z)\propto e^{-W(z)}\,\dd z\) with
\begin{equation}
 mI\preceq\nabla^2W(z)\preceq LI\,.
 \label{eq:W-Hessian-bounds}
\end{equation}
Then
\begin{equation}
 L^{-1}I\preceq\Cov_\nu(Z)\preceq m^{-1}I\,.
 \label{eq:covariance-bounds}
\end{equation}
\end{lemma}

\begin{proof}
The upper bound is the Brascamp--Lieb
inequality~\cite[Chapter~4]{bgl} applied to linear functions.  The lower
bound follows from the Cram\'er--Rao inequality for the location family
generated by \(\nu\):
\begin{align*}
 \Cov_\nu(Z)
 &\succeq \bigl(\E_\nu \nabla^2W(Z) \bigr)^{-1}
 \succeq L^{-1}I\,. \qedhere
 \end{align*}
\end{proof}

\paragraph{Hessian of the augmented marginal.}

Let \(\bar x_y\deq\E[X\mid Y=y]\) and
\(\Sigma_y\deq\eta^{-1}\,\Cov(X\mid Y=y)\), as in the normalized
coordinates of Appendix~\ref{app:key-estimates}.  Differentiating the
log-density of \(\pi_\eta^Y\) gives
\begin{equation}
 \nabla U_\eta(y)=\eta^{-1}(y-\bar x_y)\,,
 \qquad
 \nabla^2U_\eta(y)=\eta^{-1}(I-\Sigma_y)\,.
 \label{eq:Moreau-Hessian}
\end{equation}
The conditional potential in \(x\) has Hessian between
\((\alpha+\eta^{-1})I\) and \((\beta+\eta^{-1})I\).  Applying
\cref{lem:covariance-bounds} in \eqref{eq:Moreau-Hessian} gives
\eqref{eq:Y-curvature}.

\subsection{Comparison of positive operators}

\begin{lemma}[Comparison of positive operators]
\label{lem:operator-comparison}
Let $\mathsf A,\mathsf B$ be bounded positive semidefinite self-adjoint
operators on a Hilbert space $\mathcal H$, and suppose that for $M\geq m>0$, $0<m I \preceq \mathsf B \preceq M I$.
Assume that, for every $v\in\mathcal H$, $\|\mathsf A v\|^2
    \le
    \|v\|\,\|\mathsf Bv\|$.
Then, for $K\deq \frac{M}{m}$,
\[
    \mathsf A^2
    \preceq
    \bigl(1 + \frac{\log K}{\pi}\bigr)\,\mathsf B\,.
\]
\end{lemma}
\begin{proof}
By replacing
\[
    \mathsf A\leftarrow m^{-1/2}\mathsf A\,,
    \qquad
    \mathsf B\leftarrow m^{-1}\mathsf B\,,
\]
we may assume without loss of generality that $I\preceq \mathsf B\preceq K I$.
For non-zero $v \in \mathcal H$,
\begin{align*}
    \langle v,\mathsf A^2 v\rangle
    &= \langle \mathsf B^{1/2} v, \mathsf B^{-1/2} \mathsf A^2 \mathsf B^{-1/2} \mathsf B^{1/2} v \rangle
    \le \|\mathsf B^{-1/2} \mathsf A^2 \mathsf B^{-1/2}\|\,\langle v, \mathsf B v\rangle
    = \|\mathsf A\mathsf B^{-1} \mathsf A\|\,\langle v, \mathsf B v\rangle\,.
\end{align*}
Also,
\begin{align*}
    \|\mathsf A\mathsf B^{-1}\mathsf A\|
    &= \Bigl\lVert \frac{2}{\pi}\int_0^\infty \mathsf A(t^2 I + \mathsf B^2)^{-1}\mathsf A\,\dd t\Bigr\rVert
    \le \frac{2}{\pi} \int_0^\infty \|\mathsf A(t^2 I + \mathsf B^2)^{-1}\mathsf A\|\,\dd t\,.
\end{align*}
Moreover, by hypothesis,
\begin{align*}
    \|\mathsf A(t^2 I + \mathsf B^2)^{-1} \mathsf A\|
    &= \|(t^2 I + \mathsf B^2)^{-1/2} \mathsf A^2 (t^2 I + \mathsf B^2)^{-1/2}\| \\
    &= \sup_{v \in \mathcal H \setminus \{0\}} \frac{\|\mathsf A (t^2 I + \mathsf B^2)^{-1/2} v\|^2}{\|v\|^2}
    = \sup_{u \in \mathcal H \setminus \{0\}} \frac{\|\mathsf A u\|^2}{t^2 \,\|u\|^2 + \|\mathsf Bu\|^2} \\
    &\le \sup_{u \in \mathcal H \setminus \{0\}} \frac{\|u\|\,\|\mathsf Bu\|}{t^2\,\|u\|^2 + \|\mathsf Bu\|^2}\,.
\end{align*}
Setting $r \deq \|\mathsf Bu\|/\|u\| \in [1,K]$,
\begin{align*}
    \|\mathsf A(t^2 I + \mathsf B^2)^{-1} \mathsf A\|
    &\le \sup_{r \in [1,K]} \frac{r}{t^2 + r^2}
    = \begin{cases}
        1/(t^2 + 1)\,, & t < 1\,, \\
        1/(2t)\,, & t \in [1, K]\,, \\
        K/(t^2 + K^2)\,, & t > K\,.
    \end{cases}
\end{align*}
Hence,
\begin{align*}
    \|\mathsf A\mathsf B^{-1}\mathsf A\|
    &\le \frac{2}{\pi}\,\Bigl(\int_0^1 \frac{1}{t^2+1}\,\dd t + \int_1^K \frac{1}{2t}\,\dd t + \int_K^\infty \frac{K}{t^2+K^2}\,\dd t\Bigr)
    = 1 + \frac{\log K}{\pi}\,.
\end{align*}
This completes the proof.
\end{proof}

 The logarithmic dependence on $K$ in \cref{lem:operator-comparison} is sharp in general.
\begin{remark}
    Let $K > 1$, take $\mathcal H = L^2([1,K], \dd x)$, and $\mathsf Af(x) \deq (\pi\,(K-1))^{-1/2}\int_1^K f(y)\,\dd y$, $\mathsf Bf(x) \deq xf(x)$.
    Then, by the classical integral Carlson inequality,
    \[
        \|\mathsf Af\|^2
        = \frac{1}{\pi}\,\Bigl(\int_1^K f(y)\,\dd y\Bigr)^2
        \le \|f\|\,\|\mathsf Bf\|\,.
    \]
    Thus, the hypotheses of \cref{lem:operator-comparison} are met.
    If we take $f(x) \deq 1/x$, then $\|\mathsf Af\|^2 = (\log K)^2/\pi$, whereas $\langle f, \mathsf B f\rangle = \log K$, so the best possible constant in \cref{lem:operator-comparison} can be no better than $\max\{1, (\log K)/\pi\}$, where the maximum with $1$ arises from the trivial case $\mathsf A = \mathsf B = I$.
\end{remark}

\small
\bibliographystyle{alpha}
\bibliography{ref}

\end{document}